\documentclass{article}
\usepackage{graphicx}
\usepackage{amssymb, amsmath, amsthm,mathtools}
\usepackage{epstopdf}

\usepackage[marginratio=1:1,height=600pt,width=460pt,tmargin=100pt]{geometry}
\usepackage{layout, color}
\usepackage{enumerate}
\usepackage{authblk} 

\usepackage{array} 
\usepackage[font=small]{caption}

\usepackage{algorithm}
\usepackage{algpseudocode}

\usepackage{setspace}

\usepackage{hyperref}

\usepackage{url}
\usepackage[caption=false]{subfig}
\usepackage{mdwlist}
\usepackage{multirow}
\usepackage{amsthm}
\usepackage{color}

\newcommand{\R}{\mathbb{R}} 
\newcommand{\E}{\mathbb{E}}

\newcommand{\Vol}{\mathrm{Vol}}

\newcommand{\calM}{\mathcal{M}}
\newcommand{\calL}{\mathcal{L}}
\newcommand{\calE}{\mathcal{E}}
\newcommand{\calS}{\mathcal{S}}

\newtheorem{theorem}{Theorem}[section]
\newtheorem{proposition}[theorem]{Proposition}
\newtheorem{assumption}[theorem]{Assumption}

\newtheorem{lemma}{Lemma}[section]

\newtheorem{definition}[theorem]{Definition}

\theoremstyle{remark}
\newtheorem{remark}{Remark}[section]

\newtheorem{conjecture*}{Conjecture}

\theoremstyle{plain}

\usepackage[dvipsnames]{xcolor}

\newcommand{\rev}[1]{\textcolor{black}{#1}}

\begin{document}

\title{Convergence of Graph Laplacian with kNN Self-tuned Kernels}

\author{
Xiuyuan Cheng\thanks{Department of Mathematics, Duke University. Email: xiuyuan.cheng@duke.edu}
~~~~~~~~~~~~~~~~
Hau-Tieng Wu\thanks{Department of Mathematics and Department of Statistical Science, Duke University. Email: hauwu@math.duke.edu }
}

\date{}

\maketitle

\begin{abstract}

\noindent
Kernelized Gram matrix $W$ constructed from data points $\{x_i\}_{i=1}^N$ as $W_{ij}= k_0(  \frac{ \| x_i - x_j \|^2} {\sigma^2} ) $
is widely used in graph-based geometric data analysis
and unsupervised learning. 
An important question is how to choose the kernel bandwidth $\sigma$,
and a common practice called self-tuned kernel adaptively sets a $\sigma_i$ at each point $x_i$ by the $k$-nearest neighbor (kNN) distance.
When $x_i$'s are sampled from a $d$-dimensional manifold embedded in a possibly high-dimensional space,
unlike with fixed-bandwidth kernels,
theoretical results of graph Laplacian convergence with self-tuned kernels
have been incomplete. 
This paper proves the convergence of graph Laplacian operator $L_N$
to manifold (weighted-)Laplacian for
a new family of kNN self-tuned kernels $W^{(\alpha)}_{ij}
 = k_0( \frac{  \| x_i - x_j \|^2}{ \epsilon \hat{\rho}(x_i) \hat{\rho}(x_j)})/\hat{\rho}(x_i)^\alpha \hat{\rho}(x_j)^\alpha$, 
where $\hat{\rho}$ is the estimated bandwidth function {by kNN},
and the limiting operator is also parametrized by $\alpha$.
When $\alpha = 1$, the limiting operator is the weighted manifold Laplacian $\Delta_p$.
Specifically, we prove the  point-wise  convergence of $L_N f $ and convergence of the graph Dirichlet form with rates.
Our analysis is based on first establishing 
a $C^0$ consistency for 
$\hat{\rho}$
which bounds the relative estimation error  $|\hat{\rho} - \bar{\rho}|/\bar{\rho}$ uniformly 
with high probability, 
where $\bar{\rho} = p^{-1/d}$, and $p$ is the data density function. 
Our theoretical results reveal the advantage of self-tuned kernel over fixed-bandwidth kernel 
via smaller variance error in low-density regions.
In the algorithm, no prior knowledge of $d$ or data density is needed. 
The theoretical results are supported by numerical experiments  on simulated  data and hand-written digit image data.
\end{abstract}

\section{Introduction}

Kernelized Gram matrix computed from data vectors $ \{ x_i \}_{i=1}^N$ in $\R^D$ has been a 
\rev{pivotal}
tool in kernel methods \cite{scholkopf2018learning}, graph-based manifold learning and geometric data analysis 
\cite{balasubramanian2002isomap,
belkin2003laplacian,coifman2005geometric,coifman2006diffusion,van2009dimensionality, talmon2013diffusion},
 and semi-supervised learning \cite{nadler2009semi, slepcev2019analysis, li2018deeper, li2019label}{, among others}.
 Applications range {broadly} from model reduction of chemical systems
  \cite{singer2009detecting, rohrdanz2011determination, talmon2013empirical, crosskey2017atlas} to 
general data visualization \cite{borg2003modern,hinton2003stochastic, maaten2008visualizing,wang2020novel}.
The  graph {\it affinity matrix} $W$, which is real symmetric and has non-negative entries,
can be viewed as weights on the edges of a weighted undirected graph, denoted as ${\cal G} = (V,E)$, 
$V = \{ 1, \cdots, N\} $, and
$E = \{ (i,j), W_{ij} > 0\}$. 
The kernelized affinity matrix $W$ takes the form of 
a  kernelized Gram matrix,
that is,
 $W_{ij} = K(x_i, x_j)$ for some real symmetric kernel function $K: \R^D \times \R^D \to \R$.
In particular, 
a widely  used setting is
\begin{equation}\label{eq:W-fixed-eps}
W_{ij } = k_0 \left( \frac{\| x_i - x_j\|^2 }{\epsilon} \right),
\end{equation}
for some univariate kernel function $k_0$ and some $\epsilon > 0$. E.g., $k_0 (r) = e^{-r}$ gives the Gaussian affinity $W$. 

Given an affinity matrix $W$, the un-normalized and normalized (random-walk) graph Laplacian matrices are usually constructed as
$(D - W)$ and $I - D^{-1}W$ respectively, where  $D$, called the {\it degree matrix}, is the diagonal matrix with $D_{ii} = \sum_{j=1}^{N} W_{ij}$.
The row-stochastic matrix $D^{-1}W$ gives the transition law of a random walk on the graph ${\cal G}$, 
which is a discrete diffusion process on the graph \cite{masuda2017random}. 
The eigenvalues and eigenvectors of the graph Laplacian matrix
provide a  dimension-reduced representation of the data samples,
and are used for downstream tasks like clustering and dimension reduced embedding.
Variants of the basic form \eqref{eq:W-fixed-eps}
includes adaptive kernel bandwidth \cite{zelnik2005self, perrault2017improved},
anisotropic kernel 
\cite{singer2009detecting, talmon2013empirical,berry2016local,cheng2020two},
adoption of landmark sets 
 \cite{bermanis2016measure,long2017landmark, shen2020scalability},
kernel normalization schemes \cite{marshall2019manifold,wormell2020spectral}, 
and neural network approaches \cite{gong2006neural,mishne2017diffusion,shaham2018spectralnet}.
The current paper focuses on the adaptive bandwidth problem
 and the analysis of the kernelized graph Laplacian.

The convergence of the graph Laplacian to certain limiting operator 
as the graph size (number of data samples) $N \to \infty$,
is a classical theoretical problem. 
Under the {\it manifold setting}, 
namely when $x_i$ are sampled 
from a low-dimensional manifold ${\calM}$ embedded in the ambient space $\R^D$, 
the convergence 
of the graph Laplacian to a differential operator on the manifold ${\calM}$
has been proved in several places, 
when $N \to \infty$ and $\epsilon \to 0$ under a joint limit
\cite{belkin2003laplacian,hein2005graphs,belkin2007convergence,coifman2006diffusion,singer2006graph,ting2011analysis},
and more recently \cite{calder2019improved,dunson2019diffusion,trillos2020error}.
Of particular importance is when the limiting operator recovers
the manifold Laplace-Beltrami operator $\Delta_{{\calM}}$,
or the weighted Laplace operator
\begin{equation}\label{eq:def-weighted-detla-p}
\Delta_p : =  \Delta_{\calM} +   \frac{ \nabla_{\calM} p}{p } \cdot \nabla_{\calM},
\end{equation}
where $ \nabla_{\calM}$ denote the manifold derivative,
and $p$ is the density of a positive measure (not necessarily integrated to 1, i.e., a probability density).
Both $\Delta_{\calM} $ 
and $\Delta_p$ are intrinsic operators 
associated with different measures 
independent of the particular embedding
of ${\calM}$ in $\R^D$. 
The weighted laplacian $\Delta_p$  is the Fokker-Plank operator of the diffusion process on the manifold ${\calM}$,
and its spectral decomposition in $({\calM}, p dV)$ 
reveals key physics quantities of the stochastic process, 
e.g., the low-lying eigenfunctions  of $\Delta_p$ 
indicate the meta-stable states of the diffusion process \cite{matkowsky1981eigenvalues,eckhoff2005precise}. 
Thus when the discrete diffusion process on the finite-sample graph has a continuous limit,
the graph Laplacian that approximates the limiting operator $\Delta_p$
can be applied to data-driven analysis of dynamical systems
and  clustering analysis of data clouds \cite{nadler2006diffusion}. 
The asymptotic analysis of kernelized graph Laplacian 
thus
lays the theoretical foundation for applications of graph Laplacian methods in high dimensional data analysis.

A problem in the affinity matrix construction \eqref{eq:W-fixed-eps}
is the choice of the scaler parameter $\epsilon$, 
or $\sigma:=\sqrt{\epsilon}$, 
which is called the {\it kernel bandwidth}.
In practice, especially when data vectors are in high dimensional space or the sampling density is not even, the choice of $\epsilon$ can be challenging
and the performance of kernel Laplacian methods may also become sensitive to the choice,
see, e.g., \cite{perrault2017improved}.
A common practice to overcome the issue of choosing $\epsilon$ is called ``self-tuning'',
that is, 
setting a personalized bandwidth for each point $x_i$ 
 and then using these adaptive bandwidths to compute the kernelized affinity. 
Specifically, the original self-tune spectral clustering method \cite{zelnik2005self} considered the affinity matrix as
\begin{equation}\label{eq:selftune-1}
W_{ij} = k_0 \left(  \frac{\| x_i - x_j\|^2}{ \hat{R}_i \hat{R}_j } \right)
\end{equation}
where $\hat{R}_i$ equals the distance 
{from $x_i$ to its} $k$-th nearest neighbor 
{(kNN)} in the dataset $X$ itself, 
where $k$ is a parameter chosen by the user.
While several works in the  literature have addressed such kernel construction,
to the knowledge of the authors, 
the same type of results of graph Laplacian convergence as for the fixed-bandwidth kernel
have only been partially established, particularly when the kernel bandwidth is unknown and needs to be estimated from data. 
A more detailed review of related works  is given in Section \ref{subsec:lit}.

In this paper,
we introduce the following family of graph affinity matrices, defined for $\alpha \in \R$, $\epsilon > 0$, 
\begin{equation}\label{eq:selftune-family}
W_{ij}^{(\alpha)} 
:= k_0 \left(  \frac{\| x_i - x_j\|^2}{ \epsilon \hat{\rho}(x_i)  \hat{\rho}(x_j) } \right)
\frac{1}{ \hat{\rho}(x_i)^\alpha \hat{\rho}(x_j)^\alpha},
\end{equation}
where the kernel function $k_0$ is non-negative and satisfies {certain regularity and decay conditions}
(c.f. Assumption \ref{assump:h-selftune}).
We also 
split the dataset into $X$ and $Y$,
and the stand-alone $Y$ is used to estimate $\hat{\rho}(x_i)$ for each $x_i $ in $X$, 
where
$\hat{\rho}$ is a function mapping from $\R^D$ to $\R$,
called the (estimated) {\it bandwidth function}.
Specifically,
$\hat{\rho}$ is normalized from the kNN distance $\hat{R}_i$
by $\hat{\rho}(x_i) = \hat{R}_i (\frac{1}{m_0} \frac{k}{N_y})^{-1/d}$,
$k$ is the parameter in kNN, $N_y = |Y|$, 
and $m_0 > 0$ is a constant with analytical expression.
The normalization in $\hat{\rho}(x_i)$
by 
$(\frac{1}{m_0} \frac{k}{N_y})^{-1/d}$ is to guaranteed that $\hat{\rho}$
has an $O(1)$ limit.
 In view of \eqref{eq:selftune-family},
  \eqref{eq:selftune-1} is a special case where 
$\alpha =0$ and $\epsilon =\left( \frac{1}{m_0}\frac{k}{N_y} \right)^{2/d} $.
The usage of a stand-alone $Y$ to estimate $\hat{\rho}$ reduces dependence, 
and in practice can reduce variance error (c.f. Section \ref{subsec:standaloneY}).
Note that while the above definition of $\hat{\rho}$ involves the manifold dimensionality $d$, 
the practical algorithm \rev{(c.f. Algorithm \ref{algo1})} 
does not require knowledge or estimation of $d$.
\rev{We summarize the algorithm  in Section \ref{subsec:algo},}
where we introduce the construction of $\hat{\rho}$ and choice of parameters in more detail.

 \begin{table}[t]
 \small
   \caption{
\label{tab:notations}
List of default notations
} \vspace{3pt}
 \begin{minipage}[t]{0.44\linewidth}
 \begin{tabular}{  p{0.5cm}  p{5.5cm}   }
 \hline
 ${\calM}$ 	& $d$-dimensional manifold in $\R^D$  	  \\
 $p$			& data sampling density on ${\calM}$, uniformly bounded between $p_{min}$ and $p_{max}$	\\
 $\Delta_{\calM}$ & Laplace-Beltrami operator, also as $\Delta$ 		\\	
 $\Delta_{p}$ & weighed Laplace operator		\\	
 $\nabla_{\calM}$ & manifold gradient, also as $\nabla$ 		\\	
  $\bar{\nabla}$ & gradient in ambient space $\R^D$ 		\\	
 $\bar{\rho}$	& $p^{-1/d}$, population bandwidth function, 
 			uniformly bounded between $\rho_{min}:=p_{max}^{-1/d}$ and $\rho_{max}:=p_{min}^{-1/d}$	\\
$\hat{\rho}$      & estimated bandwidth function  		\\
$\varepsilon_\rho$    & error bound of $\sup_{\calM}|\hat{\rho} - \bar{\rho}|/\bar{\rho}$ \\
$k$      		& $k$-the nearest neighbor in kNN 		\\
$W$ 		 	& general kernelized affinity matrix 			\\
$D$			& degree matrix, $D_{ii} = \sum_j W_{ij}$\\
$\epsilon$ 	&  kernel bandwidth parameter used in   theoretical analysis 		\\
$\sigma_0$ 	&  kernel bandwidth parameter used in Algorithm \ref{algo1}   		\\
$\alpha$ 		& the power in $\hat{\rho}^{\alpha}$ used to normalized the self-tuned kernel 	\\
$k_0$, $h$       & function  $\R \to \R$, $h$ used in kNN estimation, $k_0$ used to construct affinity kernel			\\
$X$      		& dataset used for computing  $W$ 		\\
$Y$      		& dataset used for estimating  $\hat{\rho}$ \\
$\hat{R}_i$	& distance to $k$-th nearest neighbor of $x_i$ in $X$ 	\\
$\hat{R}$		& $\hat{R}(x)$ is the distance to $k$-th nearest neighbor of $x$ in $Y$\\
$N_x$      		& number of samples in $X$ 			\\
$N_y$      		& number of samples in $Y$ 			\\
$N$ 			& $N_x$ or $N_y$ depending on context 		\\
  \hline
\end{tabular}
\end{minipage}
 \begin{minipage}{0.5\linewidth}
 \begin{tabular}{  p{0.5cm}  p{7.5cm}   }
 \hline
 $m_0$		& $m_0[h]:=\int_{\R^d} h(|u|^2) du$ 	\\
$m_2$		& $m_2[h]:= \frac{1}{d}\int_{\R^d} |u|^2 h(|u|^2) du$ 	\\
 $\hat{K}$		& normalized self-tuned kernel function $\R^D \times \R^D \to \R$ parametrized by $\alpha$ 	\\
 $W^{(\alpha)}$			& $W$ computed using $\hat{K}$ parametrized by $\alpha$	\\
 ${\calL}^{(\alpha)}$		& limiting differential operator parametrized by $\alpha$\\
 $L_N$ 				& graph Laplacian operator of $W$ \\
 $f$				& function on ${\calM}$, also denote the vector $\{f(x_i)\}_i \in \R^{N_x}$ \\
 ${\calE}_p$		& differential Dirichlet form of $\Delta_p$\\
 ${\calE}^{(\alpha)}$		& kernelized Dirichlet form of kernel $\hat{K}$ \\
  $ E_N$				& graph Dirichlet form of $W$ \\
 $\hat{p}$				& estimated density \\
$\varepsilon_p$    & error bound of $\sup_{\calM}|\hat{p} - p |/p$ \\
  $\beta$ 		& the power in $\hat{p}^{\beta}$ used to normalized the fixed-bandwidth kernel 	\\
  $\rho$				& general kernel bandwidth function \\
  $G_\epsilon^{(\rho)}$ 	& kernel integral operator with variable bandwidth $\rho$ \\  
  $\tilde{\rho}$			& perturbed bandwidth function from $\rho$\\
  \hline
\end{tabular}
 \begin{tabular}{  p{0.5cm}  p{7.5cm}   }
  \hline
 \multicolumn{2}{|c|}{Asymptotic Notations} \\
 \hline
$O(\cdot)$ & 		$f = O(g)$: $|f| \le C |g|$ in the limit, $C> 0$, 
				$O^{[a]}(\cdot)$ declaring the constant dependence on $a$	\\ 
$\Theta(\cdot)$ 	& $f = \Theta(g)$: for $f$, $g \ge 0$, $C_1 g \le f \le C_2 g$ in the limit, $C_1, C_2 >0$,
				$\Theta^{[a]}(\cdot)$ declaring the constant dependence on $a$	\\ 
$\sim$			& $f \sim g$ same as $f = \Theta(g)$								\\
$o(\cdot)$ 	& $f = o(g)$: for $g > 0$, $|f|/g \to 0$ in the limit, 
				$o^{[a]}(\cdot)$ declaring the constant dependence on $a$\\ 
$\Omega(\cdot)$ 	& $f = \Omega(g)$: for $f, g > 0$, $f/g \to \infty$ in the limit \\ 
 \multicolumn{2}{p{8.0cm}}{ 
~ When the superscript $^{[a]}$ is $^{[1]}$,
 it declares that the constants are absolute ones.
 
~~ 
\rev{
By definition, for finite $m$, $ O(  g_1) + \cdots + O( g_m) = O(|g_1|+\cdots + |g_m| ) $,
which is also denoted as $O( g_1, ..., g_m) $.} 
				 }\\
\hline
\end{tabular}
\end{minipage}
\end{table}

The theoretical results of our work are twofold. \rev{To summarize, suppose $N=N_x$, and  $N_y= \Theta(N)$},

$\bullet$
We prove that when data are sampled according to a smooth density $p$ on a smooth compact $d$-dimensional manifold ${\calM}$,
$\hat{\rho}$ uniformly converges to $\bar{\rho} = p^{-1/d}$ on ${\calM}$, 
where the  point-wise relative error ${| \hat{\rho} - \bar{\rho}|}/{ \bar{\rho}}$ is uniformly bounded with high probability (w.h.p.)  by 
(c.f. Theorem \ref{thm:hatrho})
\[
\varepsilon_\rho =O (  ( {k}/{N} )^{2/d}  \rev{,\,}  \sqrt{{\log N}/{k}}  ).
\]
The choice of $k$ that balances the two errors is
$k \sim N^{1/(1+d/4)}$, 
which leads to  $\varepsilon_\rho =O ( N^{-1/(2+d/2)} )$
(c.f. Remark \ref{rk:choice-k-thm-hatrho}).
In particular, 
the constant in front of the second term (the ``variance error'') is independent from the sampling density $p$. 
This result augments previous analysis of kNN estimation in literature,
and the bound of the relative error (rather than absolute error) illustrates a difference between kNN $\hat{\rho}$ and the standard fixed-bandwidth kernel density estimation (KDE) (c.f. Remark \ref{rk:compare-fix-eps-kde}).
The relative error bound is also useful for our analysis of self-tuned kernel Laplacian.

$\bullet$
Conditioning on the good event of an accurate $\hat{\rho}$ estimation, 
we establish a series of convergence results of kernelized graph Laplacians  with self-tuned kernels. 
Specifically,
we prove

-
The convergence of the graph Dirichlet form (c.f. Theorem \ref{thm:limit-form}), where the error is
\[
O( \epsilon ,\, \varepsilon_\rho, \, N^{-1/2} \epsilon^{-d/4} ).
\]
The choice of $\epsilon$ to balance bias and variance errors is $\epsilon \sim N^{-1/(2+d/2)}$,
which 
\rev{makes $O( \epsilon , \, N^{-1/2} \epsilon^{-d/4} ) = O( N^{-1/(2+d/2)} )$,}
same as the order of $\varepsilon_\rho$ with the optimal scaling of $k$
(c.f. Remark \ref{rk:optimal-eps-thm-limit-form}).

-
The convergence of $L_N f(x)$ for the  (re-normalized form of) random-walk 
and un-normalized
graph Laplacian operators $L_N$ 
respectively (c.f. Theorems \ref{thm:limit-pointwise-rw-3} and \ref{thm:limit-pointwise-un-3}),
where the error is
\[
O(\epsilon, \, { \varepsilon_\rho } /{ \epsilon}, \, N^{-1/2} \epsilon^{-1/2-d/4} ).
\]
A weak convergence result in Theorem \ref{thm:limit-weak-un} shows that the error is 
$O(\epsilon ,\, \varepsilon_\rho, \, N^{-1/2} \epsilon^{-1/2} )$.

These results are compared with the counterparts for fixed bandwidth kernel,
in terms of the Dirichlet form convergence (c.f. Theorem \ref{thm:limit-form-fixeps})
and the  point-wise operator convergence (c.f. Theorem \ref{thm:limit-pointwise-rw-fix-epsilon}).
\vspace{3pt}

\rev{
As for how $N_y$ should scale with $N_x$,
if we set $N_x$ and $N_y$ independently, 
for the graph Dirichlet form convergence result, 
the overall error is optimized when $N_y \sim N_x$ (c.f. Remark \ref{rk:optimal-eps-thm-limit-form}).
The other rates, e.g., the point-wise convergence rate, 
lead to different theoretical optimal choice of $N_y$ with respect to $N_x$.
We empirically study the splitting of stand-alone $Y$ in Section \ref{subsec:standaloneY} with further discussion.
}

A key difference comparing the self-tuned kernel and the fixed-bandwidth kernel lies in the constant dependence on density $p$ 
in front of the variance term.
For example, 
in Theorem \ref{thm:limit-pointwise-rw-fix-epsilon}, the fixed-bandwidth kernel leads to a variance bound proportional to $p^{-1/2}$,
while the factor is $p^{1/d}$ in Theorem \ref{thm:limit-pointwise-un-3} for the self-tuned kernel. 
The negative factor $-1/2$ suggests a large variance error at place where $p(x)$ is small, 
reflecting the difficulty of fixed bandwidth kernel in such cases. 
The positive factor $1/d$ for self-tuned kernel suggests an improvement, which can be intuitively expected,
from a theoretical perspective.

\rev{Tab. \ref{tab:theory-summary}  gives a roadmap of analysis to facilitate the reading of Sections \ref{sec:knn-hatrho} and \ref{sec:laplacian}.}
Our theoretical results are supported by numerical experiments in Section \ref{sec:experiment}. 
Besides model \eqref{eq:selftune-family},
we also propose another affinity kernel using mixed normalization by $\hat{\rho}$ and a density estimator $\hat{p}$ 
that recovers the Laplace-Beltrami operator
and does not require knowledge of $d$,
and we give empirical results in Section \ref{subsec:exp-d-not-known},
Apart from simulated manifold data, 
we apply the self-tuned diffusion kernels  to the
MNIST dataset of hand-written digit images \cite{mnist-download},
where the self-tuned kernel shows a better stability than  the fixed-bandwidth kernel
at data points sampled at low-density places,
as  is consistent with the theory.

{\bf Notations.}
A list of default notations is provided in Tab. \ref{tab:notations}.
We use superscripts in big-$O$ notations to declare dependence of the implied constants.
In this work, 
we mainly track the dependence on the data density function $p$,
and we treat constants depending on \rev{$\calM$, $h$, $k_0$, including  $d$,} as absolute ones
so as to simplify presentation. 
\rev{For $O^{[\calM, h, \text{x}]}(\cdot)$ we may omit $(\calM, h)$ in the superscript and write as 
$O^{[\text{x}]}(\cdot)$ so as to track dependence on x only.}
Specific constant dependence can be recovered from proof.

\begin{algorithm}[t]
\small
	\caption{\small
	 kNN Self-tuned kernel graph Laplacian (with stand-alone $Y$)}\label{algo1}
	
	{\bf Input:} Datasets $X$ and $Y$, 
	and algorithm parameters $k$, $1 <  k < N_y$, $\sigma_0 > 0$, $\alpha \in \R$.

	{\bf Output:} kNN values $\hat{R}$,
	graph affinity matrix $W$, 
	degree matrix $D$, the eigenpairs $(\Psi, \Lambda)$.

	{\bf External:} Subroutine $\textproc{eig}$ ($\textproc{eig-gen}$)
	that solves eigen-problem (generalized eigen-problem) for real symmetric matrices. 
	
	\begin{algorithmic}[1]
		
		\Function{SelftunedKernel}{$X$, $Y$, $k$, $\sigma_0$, $\alpha$}
		
		\State ${N_x} \leftarrow \text{size}(X)$, ${N_y} \leftarrow \text{size}(Y)$

		\State Compute kNN distances $\hat{R}_i  \leftarrow \| x - y^{(x,k)}\|$, where
		$y^{(x,k)}$ is the $k$-th nearest neighbor of $x$ in $Y$.

		\State Compute the matrix $W$ by 
\begin{equation}\label{eq:def-W-alpha-algo1}
W_{ij}= k_0\left( \frac{ \| x_{i}-x_{j} \|^{2}}{ \sigma_0^2 \hat{R}_{i} \hat{R}_{j}}\right)  
\frac{ 1}{ \hat{R}_{i}^{\alpha} \hat{R}_{j}^{\alpha}}
\end{equation}
		\Comment{$W$ can be constructed as  a sparse matrix}

		\State  Compute the degree matrix $D$, $D_{ii} \leftarrow \sum_j W_{ij}$

		\State Compute the eigenvalue/eigenvectors $(\Psi, \Lambda)$ of 
		
		 either (1) $L_{un}$ by $\textproc{eig}(D-W)$
		
		or (2) $L_{rw'}$ by $\textproc{eig-gen}(D-W, D D_{\hat{R}}^2)$,
		where the diagonal matrix $D_{\hat{R}} = \text{diag}\{ \hat{R}_i\}_{i=1}^{N_x}$

		\State \textbf{return} $\hat{R}$, $W$, $D$, $(\Psi, \Lambda)$
		
		\EndFunction
		
	\end{algorithmic}
\end{algorithm}

\subsection{Summary of Algorithm}\label{subsec:algo}

The algorithm to 
compute the self-tuned kernel graph Laplacian on dataset $X$ with a stand-alone dataset  $Y$ to estimate the kNN bandwidth function
is summarized  in Algorithm \ref{algo1}.
It has the following parameters \rev{to be specified by the user},
\begin{itemize}

\item
$k$:  The $k$ in kNN bandwidth estimation. By Theorem \ref{thm:hatrho}, 
the choice of $k$ that balances the bias and variance errors is theoretically $k \sim  N_y^{1/(1+d/4)}$.

\item
$\alpha$: The kernel is normalized by $\hat{R}^\alpha$, where $\alpha$ is a real number. 
Different choice of $\alpha$ leads to a family of 
{different} limiting operators depending on $\alpha$, as 
{is} analyzed in Section \ref{sec:laplacian}.
In particular, choosing $\alpha  =1$ recovers the weighted Laplacian $\Delta_p$.

\item
$\sigma_0$: 
The kernel bandwidth parameter. 
{To compare with} the notation in Section \ref{sec:laplacian},
$\sigma_0 \hat{R}_i = \sqrt{\epsilon} \hat{\rho}(x_i)$,
and that is 
$\epsilon = \sigma_0^2 (\frac{1}{m_0 [h_{knn}]} \frac{k}{N_y})^{2/d}$.

\end{itemize}

Choosing $\sigma_0 = \Theta (1)$ corresponds to  $\epsilon \sim   (k/N_y)^{2/d} $.
Under the optimal scaling of $k$ which is $k \sim N_y^{1/(1+d/4)}$ (c.f. Remark \ref{rk:choice-k-thm-hatrho}),
the above scaling of $\epsilon $ is also the optimal one to balance errors of 
estimating the Dirichlet form (c.f. Theorem \ref{thm:limit-form}, Remark \ref{rk:optimal-eps-thm-limit-form}).
The algorithm does not require any prior knowledge of the intrinsic dimensionality $d$,
which can be applied to general manifold data.
We also introduce a variant form of kernel matrix to recover $\Delta_{\calM}$ in Section \ref{subsec:exp-d-not-known}.

A common usage, as in the original self-tune kernel in \cite{zelnik2005self}, 
is to estimate $\hat{R}_i$ from the dataset $X$ itself rather than a stand-alone $Y$. 
The analysis of the former case will be more complicated, 
as estimating kNN from $X$ introduces more dependence across the kernel matrix entries,
while a stand-alone $Y$ allows a two-step analysis:
showing the uniform  point-wise convergence of the kNN bandwidth function first and then analyzing the kernel matrix $W$ conditioning on a good event of $Y$,  
as we do in the current paper. 
Empirically, we find that using a stand-alone $Y$ gives a comparable result and can improve {the performance by reducing} the variance error when $N_y$ is large (Section \ref{subsec:standaloneY}).
This suggests the usage of extra data samples in the bandwidth estimation when they are not used in the kernel matrix construction
due to memory or computational constraints.

\subsection{Related Works}\label{subsec:lit}

kNN estimated kernel bandwidth was used in the original self-tune based algorithm, particularly for the spectral clustering purpose \cite{zelnik2005self}. To fully understand the role of the kNN estimated kernel bandwidth, \rev{an} understanding of its relationship with the nearest-neighbor density estimator (NNDE), an approach closely related to but different from the well known KDE, is necessary. 
NNDE was studied in the classical statistical literature dating back to 60s
\cite{loftsgaarden1965nonparametric, mack1979multivariate, hall1983near},
and more general variable bandwidth KDE was studied in 90s \cite{terrell1992variable, hall1995improved}. 
The uniform convergence with probability (w.p.) 1 was proved in \cite{devroye1977strong}, 
Our result bounds the relative error $| \hat{\rho}- \bar{\rho}|/\bar{\rho}$ uniformly w.h.p.,
which implies uniform convergence 
w.p. 1, and is under the manifold data density setting.

With this relationship in mind, to our knowledge, the first paper dealing with the asymptotical behavior of self-tune based algorithm in the manifold setup is \cite{ting2011analysis}.
The limiting operator of self-tune kernel has been identified in 
following a general framework of kernel construction of the graph Laplacian.
In particular,
the uniform convergence of NNDE in \cite{devroye1977strong} 
was used to derive the limiting operator of kNN constructed graph Laplacian. 
However, the proof is without error rate,
and the impact of NNDE has not been fully analyzed.
A recent paper addressing this issue is \cite{berry2016variable}, where the convergence rate has been proved. 
However, 
the formulation in  \cite{berry2016variable} assumes knowledge of the desired bandwidth function to use,
or that of the density function,
and the impact of NNDE is not discussed.
Also, the algorithm in \cite{berry2016variable} needs to estimate the intrinsic dimension $d$ if not given,
which may be difficult in practice.

Self-turning based algorithms are natural generalizations of those with fixed bandwidths. When the bandwidth is fixed, its asymptotical analysis has been widely discussed, particularly under the manifold setup. For example, the  point-wise convergence of the graph Laplacian to the Laplace-Beltrami operator, or more general weighted Laplace-Beltrami operator, has been extensively discussed in, for example, \cite{belkin2003laplacian,coifman2006diffusion,singer2006graph,hein2006uniform}. The spectral convergence of the graph Laplacian to the (weighted-) Laplace-Beltrami operator is more challenging, and has attracted several attentions. See, for example, \cite{belkin2007convergence,von2008consistency,wang2015spectral,singer2016spectral,eldridge2017unperturbed,trillos2020error}. 
Recently, the spectral convergence in the $L^2$ sense with rate has been provided with rate in \cite{calder2019improved}; the spectral convergence in the $L^\infty$ sense with rate, as well as the uniform heat kernel reconstruction, has been provided with rate in \cite{dunson2019diffusion}. 
We establish a  point-wise convergence of the graph Laplacian operator and convergence of the graph Dirichlet form for the kNN self-tuned kernel graph Laplacian with rates.

Other approaches of adaptive kernel bandwidth choice include multiscale SVD \cite{little2009multiscale, rohrdanz2011determination, crosskey2017atlas}.
A bandwidth selection method  
based on preserving data geometric information  was proposed in \cite{perrault2017improved}.
These methods can be computationally more expensive than kNN self-tuned bandwidth.

 \begin{table}[t]
\centering
\small
   \caption{\small
\label{tab:theory-summary}
Roadmap of analysis
} 
\vspace{3pt}
 \begin{tabular}{  m{1.1cm} | m{2.2cm} | m{3cm} |  m{3cm}   m{3cm} }
 \hline
				&	Estimation of bandwidth $\hat{\rho}$				    				  & \multicolumn{3}{ c }{ Convergence of graph Laplacian 	}		 \\
\hline 			
\hspace{-15pt} Parameters		& 	 \centering{$(k, N_y)$}								      					      &	  \multicolumn{3}{c }{ $(\epsilon, N_x)$}  			  		\\
 \hline
\hspace{-5pt}\multirow{7}{*}{ Results}	& 	 			& 		 & kNN self-tuned 	& 	fixed-bandwidth  \\
\cline{3-5}
	 					& 												& Dirichlet form & Thm. \ref{thm:limit-form}  	&	Thm. \ref{thm:limit-form-fixeps}		\\
\cline{3-5}						
						& 	Prop. \ref{prop:hatrho-one-point}			& \multirow{2}{*}{Point-wise } 	& $L_{rw}'$: Thm. \ref{thm:limit-pointwise-rw-3} &  \multirow{2}{*}{$L_{rw}$: Thm. \ref{thm:limit-pointwise-rw-fix-epsilon}}	\\					
						&	 + Lemma \ref{lemma:knn-rx} 					&			   			&$L_{un}$:  Thm. \ref{thm:limit-pointwise-un-3}	&			\\
\cline{3-5}						
					 	&	 $\Rightarrow$ Thm. \ref{thm:hatrho}			&  Weak form	  			&  Thm. \ref{thm:limit-weak-un}				&	-		\\
\cline{3-5}
				 		&											&	\multicolumn{3}{ c }{ Needed lemmas: Lemma \ref{lemma:right-operator-3} $\Rightarrow$ Prop. \ref{prop:hatE(f,f)} $\Rightarrow$ Thm. \ref{thm:limit-form} \& \ref{thm:limit-weak-un} }	\\
						&											&	\multicolumn{3}{ c }{  Lemma \ref{lemma:right-operator-repalce-rho} $\Rightarrow$ Thm. \ref{thm:limit-pointwise-rw-3} \& \ref{thm:limit-pointwise-un-3}}		\\
						
\hline
\end{tabular}
\hspace{-5pt}
\end{table}

\section{kNN Estimation of Kernel Bandwidth}\label{sec:knn-hatrho}

In this section, we prove the uniform convergence of the kNN constructed bandwidth function $\hat{\rho}$,
which is computed from the stand-alone dataset $Y$,
to $\bar{\rho} = p^{-1/d}$ w.h.p. and in terms of relative error (c.f. Theorem \ref{thm:hatrho}). 
We simplify notation by setting $N=N_y$  in this section.
All proofs are in Section \ref{sec:proofs} and 
Appendix.

Let ${\calM}$ denote the low-dimensional manifold, 
and $dV$ the volume element of $\calM$. 
When $\calM$ is orientable, $dV$ is the Riemann volume form; 
otherwise, $dV$ is the measure associated with the local volume form.
In \rev{both cases}, 
$({\calM}, dV)$ is a measure space. 
More differential geometry set-ups are provided in Appendix \ref{app:diffgeo}.

\begin{assumption}[Assumption on manifold ${\calM}$ and data density $p$.]\label{assump:M-p}

(A1) ${\calM}$ is a $d$-dimensional $C^{\infty}$ and compact manifold without boundary,
isometrically embedded in $\mathbb{R}^{D}$ 
{via $\iota$. When there is no danger of confusion, we use the same notation $x$ to denote $x\in {\calM}$ and $\iota(x)\in \mathbb{R}^D$}.

(A2) $p\in C^{\infty}(\calM)$ and uniformly bounded both from below and above, that is, $\exists p_{min}, \, p_{max} > 0$ s.t.
\[
0< p_{min} 
\le p(x) \le
p_{max} < \infty,
\quad\forall x\in{\calM}.
\]
\end{assumption}

Smoothness of $\calM$ and $p$ suffices \rev{most} application scenarios, 
and theoretically can be relaxed by standard \rev{functional approximation techniques}.
For simplicity we consider smooth {$\calM$ and} $p$ only.

\subsection{kNN Construction of $\hat{\rho}$}\label{subsec:def-hatrho}

Given $Y = \{ y_j\}_{j=1}^N$,
the kNN-estimated bandwidth function $\hat{\rho}(x)$ is a scalar field on $x \in {\calM}$ computed from $Y$, defined as
\begin{equation}\label{eq:def-knn-hatR}
\hat{\rho}(x) :=  \hat{R}(x) \left( \frac{1}{m_0[h]}\frac{ \rev{k} }{N_y} \right)^{-1/d},
\quad
\hat{R} (x) := \inf_{r} \left\{ r > 0,\, \text{ s.t. } \sum_{j=1}^{N_y} {\bf 1 }_{\{ \| y_j - x\| < r \}} \ge k \right\},
\end{equation}
and
 $m_0$ is a functional defined for function $h$ on $[0, \infty)$ sufficiently decayed as 
$m_0[h] := \int_{\R^d} h( |u|^2 )  du$.
We use $m_0$ to denote the scalar when not to emphasize the dependence on the function $h$. 
Note that for $h = {\bf 1}_{[0,1)}$, $m_0[h]$  equals the volume of unit $d$-\rev{ball}.
The definition \eqref{eq:def-knn-hatR}
is equivalent to that 
$\hat{R} (x)= \| x - y^{(k,x)}\|$,
where ${y^{(k,x)}}$ is the $k$-th nearest neighbor of $x$ in $Y$.
The following lemma gives a direct proof of the piecewise differentiability and 
 Lipschitz continuity of the knn-constructed $\hat{R}$.
The  Lipschitz constant of $\hat{R}$ is important in our proof of uniform convergence of $\hat{\rho}$. 

\begin{lemma}\label{lemma:knn-rx}
Suppose $Y$ has distinct data points $y_j$ 
and  $1< k <N_y$. 
Then $\hat{R}$ defined in \eqref{eq:def-knn-hatR}
is Lipchitz continuous on $\R^D$ with  $\text{Lip}_{\R^D}(\hat{R}) \le 1$.
Moreover, $\hat{R}$ is $C^\infty$ on $\R^D \backslash E$, 
where $E$ is a finite union of ($D$-1)-hyperplanes
(finitely many points when $D=1$).
\end{lemma}

One may consider variants of kNN-estimator.
Specifically, 
\rev{a} generalization of \eqref{eq:def-knn-hatR} \rev{can be}
$\hat{R} (x) = \inf_{r} \left\{ r > 0,\, \text{ s.t. } \sum_{j=1}^{N_y} h \left(  \frac{ \| x- y_j \|^2}{r^2}  \right) \ge k \right\}$,
\rev{where one can} 
introduce weights proportional to the distance $\|x-y_j\|$ by considering a more general $h$.
\rev{
The definition \eqref{eq:def-knn-hatR} is equivalent to taking $h = {\bf 1}_{[0,1)}$.}
One advantage of the classical kNN-estimator \eqref{eq:def-knn-hatR} is its efficient computation by \rev{the} fast kNN algorithm.
We are not aware of any other widely used weighted version of kNN-estimator, 
thus we postpone the possible extension  to larger class of $h$  to future work.

\subsection{$C^0$ Consistency of $\hat{\rho}$}\label{subsec:concen-one-point}

The concentration of $\hat{\rho}$ at $\bar{\rho}$ at a point $x_0$ is a result of the concentration of the independent sum
in \eqref{eq:def-knn-hatR},
which we prove in the following proposition:

\begin{proposition}\label{prop:hatrho-one-point}
Under Assumption \ref{assump:M-p},
if as $N \to \infty$, $k = o(N)$
and $k = \Omega( \log N)$,
then, \rev{for any $s > 0$,
when $N$ is sufficiently large,
for any $x \in {\calM}$,}
w.p. $> 1- 2N^{- s}$, 
\begin{equation}\label{eq:bound-hatrho}
\frac{ |\hat{\rho}(x) - \bar{\rho}(x) |}{\bar{\rho}(x)}
= O^{[p]} \left( \left(\frac{k}{N} \right)^{2/d} \right) + 
 \frac{\rev{3}}{d} \sqrt{ \frac{ s \log N }{k}},
\end{equation}
where 
\rev{
the constant in $O^{[p]}(\cdot)$ is determined by $p$,
the threshold for large $N$ depends on $p$ and $s$,
and both are uniform for all $x$.
}
\end{proposition}

Combined with the global Lipschitz continuity of $\hat{\rho}$ (Lemma \ref{lemma:knn-rx})
and a bound of the covering number of ${\calM}$ (Lemma \ref{lemma:covering-2}),
we are ready to prove the main result of this section:

\begin{theorem}\label{thm:hatrho}
Under Assumption \ref{assump:M-p},  $\hat{\rho}$ defined as in \eqref{eq:def-knn-hatR},
$N=N_y$. 
If as $N \to \infty$, $k = o(N)$ and
$k = \Omega( \log N)$,
then when $N$ is sufficiently large,
w.p. higher than  $1- N^{-10}$,
\[
\sup_{x \in {\calM}} \frac{|  \hat{\rho}(x)- \bar{\rho}(x) |}{ \bar{\rho}(x)}
=
O^{[p]} \left(\left(\frac{k }{N }\right)^{2/d} \right) 
+   \frac{ \rev{3\sqrt{13}}}{d} \sqrt{  \frac{  \log N }{k }} ,
\]
and the right-hand side (r.h.s.) is $o^{[p]}(1)$.
\end{theorem}

\begin{remark}\label{rk:choice-k-thm-hatrho}
In the error bound, the $O( (k/N)^{2/d} )$ term is the ``bias'' error,
and the $O ( \sqrt{ {\log N}/{k}} )$ term is the ``variance'' error.
To balance the two  errors,
$k$ should be chosen according to
$k^{-1/2} \sim ( {k}/{N} )^{2/d}$, where we omit the $\sqrt{\log N}$ factor, and that is
$ k \sim N^{{1}/( 1+d/4 )}$. In this scaling,  the constant in front  theoretically depends on $p$
and generally is impractical to estimate. 
\end{remark}

\begin{remark}\label{rk:compare-fix-eps-kde}
One can compare Theorem \ref{thm:hatrho} to the estimation error bound of 
a fixed bandwidth KDE estimator of  $p$, e.g., for $\epsilon > 0$,
\begin{equation}\label{eq:def-hatp-fixeps}
\hat{p}(x) :=  \frac{\epsilon^{-d/2}}{ m_0 [h_{kde}]} 
\frac{1}{N_y} \sum_{j=1}^{N_y} h_{kde} \left(  \frac{ \| x- y_j\|^2}{ \epsilon} \right),
\end{equation}
where $ h_{kde}: \R_+ \to \R $ is usually a non-negative regular function.
When $h_{kde} \ge 0$ and satisfies Assumption \ref{assump:h-diffusionmap},
by analyzing the bias and variance errors of the independent sum in \eqref{eq:def-hatp-fixeps}
and using  Lemma \ref{lemma:h-integral-diffusionmap}, one can verify that
$
\hat{p}(x) = p(x) + O^{[p]}( \epsilon ) + O^{[1]} \left( {p(x)}^{1/2} \sqrt{\frac{\log N}{N  \epsilon^{d/2}}} \right)$.
The relative error, i.e., $ {|  \hat{p}(x)- p(x) |}/{ p(x)}$, at point $x$ is  then bounded by 
\[
O^{[p]} (\epsilon )+  O^{[1]} \left( p(x)^{-1/2} \sqrt{  \frac{  \log N }{ N \epsilon^{d/2} }}\right).
\]
To make a comparison to kNN $\hat{\rho}$,
we set \rev{$\epsilon =\Theta^{[1]} ( (k/N)^{2/d} )$},
 then Theorem \ref{thm:hatrho}  gives that the relative error
of $\hat{\rho}$, i.e.,
$ {|  \hat{\rho}(x)- \bar{\rho}(x) |}/{ \bar{\rho}(x)}$, is uniformly bounded by
$
O^{[p]} (\epsilon )+  O^{[1]} \left(  \sqrt{  \frac{  \log N }{ N \epsilon^{d/2} }}\right)$.
This illustrates that the variance error in the relative error of $\hat{p}$
has a factor $p(x)^{-1/2}$, 
while for $\hat{\rho}$ by kNN the \rev{variance error} term is uniformly bounded for all $x$ 
\rev{by $O^{[1]} \left(  \sqrt{  \frac{  \log N }{ N \epsilon^{d/2} }}\right)$ and the constant is} independent of $p(x)$. 
This difference between kNN $\hat{\rho}$ and fixed-bandwidth KDE $\hat{p}$
 is numerically verified in Section \ref{subsec:exp-hatrho} (Fig. \ref{fig:hatrho-hatp-1}).
\end{remark}

\subsection{$C^1$ Divergence of $\hat{\rho}$}\label{subsec:C1-inconsist}

As shown in the proof of Lemma \ref{lemma:knn-rx}, for any $x \in \R^D \backslash E $,
$| \bar{\nabla} \hat{R}(x) | =  1$,  where $ \bar{\nabla} $ denotes the gradient in the ambient space $\R^D$.
Thus,
\[
| \bar{\nabla} \hat{\rho}(x)  | = \left( \frac{1}{m_0[h]}\frac{ \rev{k} }{N_y} \right)^{-1/d}, \quad \forall x \in \R^D \backslash E,
\]
which is $\Omega(1)$ as $\frac{ \rev{k} }{N_y} \to 0$. 
This means that $\nabla_{\calM} \hat{\rho}(x)$   point-wise diverges almost everywhere, 
and cannot have  point-wise consistency to $\nabla_{\calM} \bar{\rho}(x)$, which is $O(1)$.

While the $l$-th $\R^D$-derivative of $\hat{\rho}$ can be bounded to be $O( (\frac{k}{N})^{-l/d} )$ 
(Lemma \ref{lemma:knn-rx-part2}),
this $C^1$ inconsistency of $\hat{\rho}$ by kNN estimation poses challenge to the graph Laplacian convergence,
because the limiting operator (see Section \ref{subsec:lap-operators})
 involves $\nabla_{\calM} \rho$ when $\rho$ is a deterministic bandwidth function \cite{berry2016variable}.
 On the other hand,
 the wide usage of self-tuned diffusion kernel
 in spectral clustering and spectral embedding
 suggests that the kNN-estimated $\hat{\rho}$ can lead to a consistent estimator of certain limiting manifold differential operators,
 though the $C^0$ consistency of $\hat{\rho}$ alone may not be able to directly prove that. 

In Section \ref{sec:laplacian}, we will show theoretically that the point-wise consistency of the graph Laplacian operator has a different and worse error rate 
than the consistency of the graph \rev{Dirichlet} form,
where consistency in both cases is obtained 
but under different conditions on $\rev{k}, N_y$ related to the bandwidth parameter  $\epsilon \to 0$.
The distinction is also revealed in experiments in Section \ref{sec:experiment}:
while  point-wisely $L_N f(x)$ can be oscillating and deviating from the ${\calL} f(x)$,
where $L_N$ is the graph Laplacian and ${\calL}$ is the limiting differential operator,
the Dirichlet form has much smaller error especially when $\epsilon$ is small
 (Fig. \ref{fig:Ln-error-1d} and Fig. \ref{fig:Ln-1d}).

\section{Analysis of  Graph Laplacian}\label{sec:laplacian}

In this section, 
we analyze the convergence of self-tuned graph Laplacian computed 
from dataset $X = \{x_i\}_{i=1}^{N_x}$, $x_i \sim p$,
sampled on ${\calM}$,
where $\hat{\rho}(x_i)$ has been computed from a stand-alone dataset $Y$,
and we assume that Theorem \ref{thm:hatrho} holds.
We first introduce the notations of limiting operators and the Dirichlet forms in Section \ref{subsec:lap-operators},
and then prove 

\begin{itemize}
\item
The convergence of the kernelized Dirichlet form in Section \ref{subsec:dir-form-kernel},
as a middle-step result;

\item
The convergence of the graph Dirichlet form in Section \ref{subsec:dir-form-graph};

\item
The convergence of $L_N f(x)$ for un-normalized and random-walk graph Laplacian operators $L_N$ 
in Section \ref{subsec:Lnf-convergence}.
\end{itemize}

We simplify notation $N = N_x$ in the section.
 All proofs are in Section \ref{sec:proofs} and 
 Appendix. 
The following regularity and decay condition 
is needed for \rev{the function} $k_0$ in \eqref{eq:selftune-family}.
The condition on $k_0$ in  \cite{coifman2006diffusion} is in Assumption \ref{assump:h-diffusionmap},
and here
we further assume 
non-negative $k_0$,  and $C^4$ regularity for simplicity.

\begin{assumption}[Assumption on $k_0$]
\label{assump:h-selftune}
$k_0$ satisfies Assumption \ref{assump:h-diffusionmap} and in addition,

(C1) Regularity. $k_0$ is continuous on $[0,\infty)$,  $C^4$ on $(0, \infty)$. 

(C2) Decay condition.  $\exists a, a_k >0$, s.t., $ |h^{(k)}(\xi )| \leq a_k e^{-a \xi}$ for all $\xi > 0$, $k=0, 1, \cdots, 4 $.

(C3) {Non-negativity}. $k_0 \ge 0$ on $[0, \infty)$.
 \end{assumption}
 
{We use} $m_0 = m_0[k_0]$ and $m_2 = m_2 [k_0]$
if the kernel function dependence is not clarified.

\subsection{Notation of Manifold Laplacian Operators and Dirichlet Forms}
\label{subsec:lap-operators}

Recall the weighted laplacian $\Delta_p$ defined as in \eqref{eq:def-weighted-detla-p}
on $({\calM}, p dV) $, where $p$ is the density of a positive measure on ${\calM}$.
Below, we write $\Delta_{\calM}$ as $\Delta$,  $\nabla_{\calM}$ as $\nabla$ when there is no danger of confusion.

Take a positive $C^1$ function $\rho$ on $\calM$.
As will appear in the analysis, we introduce ${\calL}^{(\alpha)}_{{\rho}}$ as 
\begin{equation}\label{eq:def-Lrho-alpha}
{\calL}_{\rho}^{(\alpha)} :=  
\Delta + 2  \frac{\nabla p}{p}\cdot \nabla + (d-2\alpha+2) \frac{\nabla \rho}{\rho } \cdot \nabla\,.
\end{equation}
When $\rho = \bar{\rho}= p^{-1/d}$, one can verify that 
${\calL}^{(\alpha)}:={\calL}^{(\alpha)}_{\bar{\rho}} $ satisfies
\begin{equation}\label{eq:def-calL-alpha}
{\calL}^{(\alpha)}  
=  
 \Delta +  \left(1 + \frac{2(\alpha -1)}{d}\right) \frac{\nabla p}{p } \cdot \nabla\,.
\end{equation}
We will show that the operators ${\calL}^{(\alpha)}  $ and $ p^{\frac{2(\alpha-1)}{d}}  {\calL}^{(\alpha)} $ 
are the limiting operators of the  (modified) random-walk and un-normalized graph Laplacians respectively.

The  {\it differential Dirichlet form} associated with $\Delta_p$ is defined as 
\[
{ \calE  }_{p} (f,f)
: = - \langle f, \Delta_p f\rangle_p = \int_{\calM} p |\nabla f|^2 \rev{dV},
\]
where $\langle f, g \rangle : = \int_{\calM} f g dV$ for $f, g \in L^\infty({\calM})$,
and $\langle f, g \rangle_q = \int_{\calM} f g  q dV$ for $q$ a density of a positive measure on ${\calM}$.
\rev{In below, we may omit $dV$ in the notation of integral over $\calM$,  that is, $\int_\calM f$ means $\int_\calM f dV$.}
Given a graph affinity matrix $W$ and a vector $ f: V \to \R$, $f \in \R^N$,
we consider the  
(normalized) {\it graph Dirichlet form} defined as 
\begin{align}\label{eq:def-EN(f,f)}
E_N(f,f)
 &:= \frac{2}{ \epsilon m_2} \frac{1}{N^2} \epsilon^{-d/2} f^T (D-W) f \\
& =  \frac{2}{ \epsilon m_2} \frac{1}{N^2}
\sum_{i,j=1}^N \epsilon^{-d/2} W_{ij}  f_i (f_i - f_j)   
 = \frac{1}{ \epsilon m_2} \frac{1}{N^2}
\sum_{i,j=1}^N \epsilon^{-d/2} W_{ij} (f_i - f_j)^2.  \nonumber
\end{align}
We will  prove that the graph Dirichlet form converges to the differential Dirichlet form of 
a density $p_{\alpha}: =  p^{ 1 + \frac{2(\alpha -1)}{d} }$ on ${\calM}$.
This is consistent with the above limiting operator, as  one can verify that 
\[
- \langle f,  p^{\frac{2(\alpha-1)}{d}} {\calL}^{(\alpha)} f \rangle_p
=
- \langle f, \Delta_{ p_{\alpha} } f \rangle_{p_{\alpha}} = {\calE}_{p_\alpha}(f,f).
\]

\subsection{Convergence of the Kernelized Dirichlet Form}
\label{subsec:dir-form-kernel}

Consider  $W = W^{(\alpha)}$ as in \eqref{eq:selftune-family},
and define
\begin{equation}\label{eq:def-hatK}
\hat{K}(x,y) := \epsilon^{-d/2} k_0  \left(  \frac{\| x - y \|^2}{ \epsilon \hat{\rho}(x)  \hat{\rho}( y ) } \right)
\frac{1}{ \hat{\rho}( x )^\alpha \hat{\rho}(y)^\alpha}.
\end{equation}
Then by definition  $\epsilon^{-d/2} W_{ij}^{(\alpha)} =  \hat{K}(x_i,x_j)$.
We use ``hat'' to emphasize the dependence on the estimated bandwidth $\hat{\rho}$.
When $f_i = f(x_i)$ for $f: {\calM} \to \R$ (here we use the notation $f$ for both the function and the vector),
$E_N(f,f)$ has the following population counterpart which is an integral form on ${\calM}$,
\begin{equation}\label{eq:def-hatE(f,f)}
{\calE}^{(\alpha)} (f,f)
 : =
 \frac{1}{\epsilon m_2} 
\int_{\calM} \int_{\calM}   (f(x) - f(y))^2 
\hat{K}(x,y) p(x) p(y)
dV(x) dV(y).
\end{equation}
We call ${\calE}^{(\alpha)} (f,f)$ the  {\it kernelized Dirichlet form}.
The following proposition proves the convergence of 
${\calE}^{(\alpha)} (f,f)$ to the differential Dirichlet form ${\calE}_{p_\alpha}(f,f)$: 
\begin{proposition}
\label{prop:hatE(f,f)}
Suppose $\hat{\rho}$ satisfies that 
$\sup_{x \in {\calM}}\frac{ | \hat{\rho}(x) - \bar{\rho}(x) | }{ | \bar{\rho}(x)|} < \varepsilon_\rho < 0.1$.
{Then} for any $f \in C^\infty( {\calM} )$,
\[
 {\calE}^{(\alpha)} (f,f) =   {\calE}_{p_\alpha}(f,f) (1 + O^{[\alpha]}( \varepsilon_\rho ))  + O^{[f,p]}( \epsilon),
\quad
p_{\alpha} =  p^{ 1 + {2(\alpha -1)}/{d} }.
\]
\end{proposition}
\rev{
In the proposition, we omit the dependence on $\alpha$ in the notation of the $O^{[f,p]}( \epsilon)$ term.
Here and in below, 
we omit the dependence on $\alpha$ and track that on $p$ and $f$ in the big-$O$ notation,
unless we want to stress the former.}
The proposition leads to the convergence of $\E E_N(f,f)$, 
and is used in proving the convergence of $E_N(f,f)$ (Theorem \ref{thm:limit-form})
and the weak convergence of $L_N f$ (\rev{Theorem \ref{thm:limit-weak-un}}).
An important technical object used in the analysis of ${\calE}^{(\alpha)} (f,f)$ 
and later analysis 
is the following integral operator $G_\epsilon^{(\rho)}$ 
defined for $f \in C^\infty({\calM})$ and any $\epsilon > 0$,
\begin{equation}\label{eq:def-G-R-rho-epsilon}
G^{(\rho)}_\epsilon f(x) 
:= \epsilon^{-d/2} \int_{\calM} 
{  k_0 \left(  \frac{\| x - y\|^2}{  \epsilon  {\rho}( y)  } \right)}   f(y)  dV(y),
\end{equation}
which is well-defined when $\rho$ is positive and has some regularity so that the integral exists, e.g. $C^0$ regularity and bounded from below.
The following lemma is a reproduce of a similar step used in \cite{berry2016variable} 
where we derive point-wise error bound (see remark \ref{remark:lemma-right-operator}).

\begin{lemma}\label{lemma:right-operator-3}
Under Assumption \ref{assump:M-p}, suppose $k_0$ satisfies Assumption \ref{assump:h-selftune},
$f$ and $\rho$ are in $C^4 ( {\calM})$,
and $0 < \rho_{min} < \rho < \rho_{max}$ uniformly on ${\calM}$, then
\begin{equation}\label{eq:G-R-expansion-3}
\begin{split}
& G_\epsilon^{(\rho)} f 
 =  m_0 f {\rho}^{\frac{d}{2}}
+ \epsilon \frac{m_2}{2} (   \omega  f {\rho}^{1+\frac{d}{2}} + \Delta (f  {\rho}^{1+\frac{d}{2}} ) ) 
+ r_{\epsilon}^{(2)}, \\
&~~~
 \sup_{x \in {\calM}} | r_\epsilon^{(2)}(x) |
\le 
c_\rho (1+\sum_{l=0}^4 \| D^{(l)}  f \|_\infty ) (1+\sum_{l=0}^4  \| D^{(l)}  \rho^{-1}\|_\infty)
\epsilon^2
= O^{ [ f,\rho] } (\epsilon^2),
\end{split}
\end{equation}
where $c_\rho > 0$ is a constant depending on $({\calM}, k_0, \rho_{max},\rho_{min})$
(a rational function of $(\rho_{max},\rho_{min})$ where the coefficients depend on $({\calM}, k_0)$),
$ D^{(l)} $ is $l$-th manifold intrinsic derivatives, 
and $\omega(x)$ depends on local derivatives of the extrinsic manifold coordinates at $x$.
\end{lemma}

However, we cannot directly apply the lemma to \eqref{eq:def-hatE(f,f)} because $\hat{\rho}$ does not have $C^4$ regularity. 
The proof of Proposition \ref{prop:hatE(f,f)}
is via substituting $\hat{\rho}$ by $\bar{\rho}$ and control the error, 
and then applying Lemma \ref{lemma:right-operator-3} where $\rho = \bar{\rho}$ which is $C^\infty$.
As a postponed discussion,
modifying $\hat{\rho}$ to be $C^\infty$ is considered in Appendix \ref{app:another-pt-limit} under another limiting setting.

\subsection{Convergence of the Graph Dirichlet Form}\label{subsec:dir-form-graph}

We will show that when $N \to \infty$ and $\epsilon \to 0$ under a proper joint limit, 
$E_N(f ,f)$ converges  to
${\calE}_{p_\alpha}(f,f)$. 
This means that with the self-tuned kernel the graph Dirichlet form 
asymptotically recovers the differential Dirichlet form of
the weighted Laplacian $ \Delta_{ p_{\alpha} } $ on $({\calM}, p_{\alpha} dV)$.
In particular,
\begin{itemize}
\item 
When $\alpha =1$, $p_1 = p$, 
and the graph Laplacian recovers $\Delta_p$ on $({\calM}, p dV)$.

\item
When $\alpha = 0$, $p_0 =  p^{1-2/d}$,
thus the original self-tune graph Laplacian recovers weighted Laplacian with a modified density.  

\item
When $\alpha = 1-\frac{d}{2}$,
$p_\alpha $ is a constant,
then the graph Laplacian recovers $\Delta_{\calM}$
and the Dirichlet form with uniform density.
We provide an approach to obtain $\Delta_{\calM}$ when $d$ is not known in Section \ref{subsec:exp-d-not-known}. 
\end{itemize}

For the estimated $\hat{\rho}$ from $Y$, 
suppose Theorem \ref{thm:hatrho} holds,  
and we consider the randomness over $X$ conditioning on a realization of $Y$ under the good event.

\begin{theorem}\label{thm:limit-form}
Suppose 
$\hat{\rho}$ satisfies that 
$\sup_{x \in {\calM}}\frac{ | \hat{\rho}(x) - \bar{\rho}(x) | }{ | \bar{\rho}(x)|} < \varepsilon_\rho < 0.1$,
and as $N \to \infty$,
\[
\epsilon = o(1),
\quad
\epsilon^{d/2 } N = \Omega( \log N),
\quad
\varepsilon_\rho = o(1),
\]
then for any $f \in C^\infty({\calM})$,
 when $N$ is sufficiently large, w.p. $> 1- 2N^{-10}$, and $p_{\alpha} =  p^{ 1 + {2(\alpha -1)}/{d} }$,
\[
 E_N(f,f)
= 
{ \calE}_{p_\alpha}(f,f)(1 + O^{[\alpha]}(\varepsilon_\rho ))
+ O^{[f,p]} ( \epsilon ) 
+ O^{[1]} \left(  \sqrt{  \frac{  \log N    }{ N \epsilon^{d/2  }}  \int_{\calM}  |\nabla f |^4 p^{1+ \frac{4(\alpha-1)}{d}}  }\right).
\]
\end{theorem}

\begin{remark}\label{rk:optimal-eps-thm-limit-form}
By Remark \ref{rk:choice-k-thm-hatrho}, 
the optimal choice of $k$ to minimize $\varepsilon_\rho$
is when $k \sim N_y^{1/(1+d/4)}$
and this leads to \rev{$\varepsilon_\rho \sim N_y^{-1/(2+d/2)} $ up to a $\log N$ factor. 
The possible $\log N$ factor is no longer declared in all the scalings in this remark.
To make} $\epsilon \sim \varepsilon_\rho$, it gives 
$\epsilon \sim  (\frac{k}{N_y})^{2/d} \sim N_y^{-1/(2+d/2)}$.
The scaling $\epsilon \sim  (\frac{k}{N_y})^{2/d}$
 is the same one as in the original kNN self-tune kernel \eqref{eq:selftune-1}.
\rev{Meanwhile, in the error bound in Theorem \ref{thm:limit-form},
leaving the term due to $\varepsilon_\rho$ aside,
the other two terms of bias and variance errors are balanced 
when $\epsilon \sim N_x^{- 1/( 2 + d/2) }$,
and this gives the overall error of the two terms as $N_x^{- 1/( 2 + d/2)}$.
Compared to $\varepsilon_\rho \sim N_y^{-1/(2+d/2)} $ at the optimal scaling of $k$ with $N_y$,
the overall error bound in Theorem \ref{thm:limit-form} is balanced when $N_y = \Theta( N_x)$.
}
\end{remark}

To see the effect of self-tuning kernel, 
we compare Theorem \ref{thm:limit-form} with the following theorem for a fixed-bandwidth kernel
normalized by density estimators,
defined \rev{for $\beta \le 1$} as 
\begin{equation}\label{eq:def-Wbeta-fixepsilon}
W_{ij}^{(\beta)} = k_0 \left( \frac{\|x_i - x_j\|^2}{\epsilon} \right) \frac{1}{ \hat{p}(x_i )^\beta \hat{p}( x_j)^\beta },
\end{equation}
\rev{assuming $\hat{p}(x_i) > 0$ for all $i$.}
Let $E_{N,\epsilon}(f,f)$ equals \eqref{eq:def-EN(f,f)} with $W = W^{(\beta)}$,  
and that gives 
\begin{align*}
E_{N,\epsilon}(f,f)
 := \frac{1}{ \epsilon m_2} \frac{1}{N^2}
\sum_{i,j=1}^N \epsilon^{-d/2} W^{(\beta)}_{ij} (f_i - f_j)^2\,. 
\end{align*}
\rev{
Below, in Theorems \ref{thm:limit-form-fixeps} and \ref{thm:limit-pointwise-rw-fix-epsilon} about fixed-bandwidth kernel,
we track the constant dependence on $\beta$ more carefully since 
for the special case where $\beta =0$ no density estimation is needed.
}

\begin{theorem}\label{thm:limit-form-fixeps}
Suppose as $N \to \infty$, $ \epsilon = o(1)$, $ \epsilon^{d/2 } N = \Omega( \log N)$,
and 
if $\beta  \rev{\neq}  0$,
 the estimated density $\hat{p}$ satisfies that
$\sup_{x  \in {\calM} }\frac{ |\hat{p}(x) - p(x)| }{p(x)} < \varepsilon_p < 0.1$,
and $\varepsilon_p = o(1)$,
then for any $f \in C^{\infty} ({\calM})$, 
 when $N$ is sufficiently large,
w.p. $> 1- 2N^{-10}$, \rev{ and $c_\beta =  \max\{ 1.1^{-\beta-1}, 0.9^{-\beta-1}\}$,} 
\[
E_{N,\epsilon}(f,f)
= {\calE}_{p^{2-2\beta} } (f,f)(1+ \rev{ O^{ [1] }(\beta c_\beta} \varepsilon_p   )) 
+ O^{[f,p, \rev{ \beta } ]} \left( \epsilon \right) 
+ O^{[1]} \left(  \sqrt{  \frac{  \log N    }{ N \epsilon^{d/2  }}  \int_{\calM}  |\nabla f |^4 p^{2-4\beta}  }\right).
\]
\rev{
In particular, when $\beta =0$,
$
E_{N,\epsilon}(f,f)
= {\calE}_{p^{2} } (f,f)
+ O^{[f,p ]} \left( \epsilon \right) 
+ O^{[1]} \left(  \sqrt{  \frac{  \log N    }{ N \epsilon^{d/2  }}  \int_{\calM}  |\nabla f |^4 p^{2}  }\right)$.
}
\end{theorem}

Note that $ \Delta_{p^{2-2\beta}}  = \Delta + 2(1-\beta) \frac{ \nabla p }{p} \cdot \nabla$,
which is consistent with the limiting operator of the original Diffusion Map paper \cite{coifman2006diffusion},
and in particular,  $\beta = \frac{1}{2}$ recovers $\Delta_p$.
Strictly speaking, the setting is different because in \cite{coifman2006diffusion},
$D_{ii}^{1/2}$ is used to normalize the affinity matrix $W_{ij}=k_0 \left( \frac{\|x_i - x_j\|^2}{\epsilon} \right)$,
and $D_{ii} = \sum_{j}  k_0 \left( \frac{\|x_i - x_j\|^2}{\epsilon} \right)$.
While $D_{ii}$ can be viewed as a KDE,
normalizing by $D_{ii}$ introduces dependence and techniques to analyze normalized graph Lapalcian
are needed, e.g., as in Theorem \ref{thm:limit-pointwise-rw-3}.

\begin{remark}\label{rk:hatp-rel-error}
We have shown in Remark \ref{rk:compare-fix-eps-kde} that the relative error of $\hat{p}$
by a fixed-bandwidth KDE \eqref{eq:def-hatp-fixeps} behaves differently from that of $\hat{\rho}$.
Specifically, 
when variance error dominates, $| \hat{p} (x)  - p (x)|/p(x)$ is proportional to $p(x)^{-1/2}$,
while the variance error in $|\hat{\rho} (x)  - \bar{\rho} (x)|/ \bar{\rho}(x)$ can be made small uniformly for $x \in {\calM}$ independent of $p(x)$.
This means that, though 
the error bound in Theorem \ref{thm:limit-form-fixeps} has a $O(\varepsilon_p)$ term (when $\beta \rev{\neq} 0$)
which appears to be the counterpart of the $O(\varepsilon_\rho)$ term in the bound in Theorem \ref{thm:limit-form}, 
\rev{under situations where $p$ is} small at some places,
the kNN self-tuned kernel can have an advantage due to its ability to make $\varepsilon_\rho$ small.
To achieve \rev{the} same \rev{property} by the fixed-bandwidth kernel considered in Theorem \ref{thm:limit-form-fixeps},
it calls for \rev{ the KDE to} make $\varepsilon_p$ small,
which may need the KDE to be else than  \eqref{eq:def-hatp-fixeps}.
\end{remark}

We postpone further discussion about fixed bandwidth kernel,
since the current paper focuses on the estimated variable bandwidth kernel.

\subsection{Convergence of  $L_N f$}\label{subsec:Lnf-convergence}

We consider two types of graph Laplacian operator $L_N f$, 
where, 
using kernel $K^{(\alpha)}_{\epsilon, \hat{\rho}}$ as in \eqref{eq:selftune-family},
the un-normalized graph Laplacian operator applied to $f \in C^\infty({\calM})$ is defined as
\begin{equation}\label{eq:def-Lalapha-un}
L^{(\alpha)}_{ un}
f(x) = 
\frac{2 \epsilon^{-\frac{d}{2}-1}}{ {m_2 } }\frac{1}{\hat{\rho}(x)^\alpha}
\frac{1}{N}
\sum_{j=1}^{N}  
k_0 \left(  \frac{\| x - x_j\|^2}{ \epsilon \hat{\rho}(x)  \hat{\rho}(x_j) } \right)
\frac{f(x_j) - f(x)}{ \hat{\rho}(x_j)^\alpha},
\end{equation}
and the (modified) random-walk graph Laplacian operator is
\begin{equation}\label{eq:def-Lalapha-rw}
L^{(\alpha)}_{ rw'}
f(x) = 
\frac{1}{\epsilon \frac{m_2}{2m_0} 
\hat{\rho}(x)^{2}}
\left(\frac{\sum_{j=1}^{N}  
{  k_0 \left(  \frac{\| x - x_j\|^2}{  \epsilon  \hat{\rho}( x)  \hat{\rho}( x_j) } \right)}
\frac{f(x_{j})}{\hat{\rho}(x_j)^{\alpha} } }
{\sum_{j=1}^{N} 
{  k_0 \left(  \frac{\| x - x_j\|^2}{  \epsilon  \hat{\rho}( x)  \hat{\rho}( x_j) } \right)}
 \frac{1}{\hat{\rho}(x_j)^{\alpha} } }
 -f(x)
 \right).
\end{equation}
In the matrix form,
the operator differs from the usual random-walk Laplacian $(I - D^{-1} W)$ by multiplying another diagonal matrix $D_{\hat{\rho}}^{-2}$
(up to multiplying a constant and the sign),
thus we call it ``modified'' and denote it by ``rw-prime".

The point-wise convergence of $L_N f(x)$ at a fixed point $x \in {\calM}$
is a more traditional setting under which 
the convergence to a limiting diffusion operator has been considered
{in various papers}
\cite{coifman2006diffusion, singer2006graph, berry2016variable}.
The closest 
{one} is the result in \cite{berry2016variable}.
However, 
an extension of the method therein
 leads to a convergence to ${\calL}^{(\alpha)}_{\hat{\rho}} f$
under the asymptotic that $\epsilon = o( (k/N_y)^{4/d})$ (c.f. Theorems \ref{thm:limit-pointwise-rw} and \ref{thm:limit-pointwise-un} in Appendix \ref{app:another-pt-limit}). 
However, 
this convergence result does not imply consistency to ${\calL}^{(\alpha)}_{\bar{\rho}} f = {\calL}^{(\alpha)} f$,
due to the lack of convergence of $\frac{ \nabla \hat{\rho}}{ \hat{\rho}}$ to $\frac{ \nabla \bar{\rho}}{ \bar{\rho}}$,
as discussed in Section \ref{subsec:C1-inconsist}.
Meanwhile, note that the uniform $C^0$ consistency of $\hat{\rho}$ to $\bar{\rho}$ does imply weak convergence of 
${\calL}^{(\alpha)}_{\hat{\rho}} f \to {\calL}^{(\alpha)}_{\bar{\rho}} f$
when $\varepsilon_\rho \to 0$,
a result of the same type as Theorem \ref{thm:limit-weak-un},
while the latter shows an improved variance error (\rev{with} $\epsilon^{-1/2}$ rather than $\epsilon^{-d/4-1/2}$).

Back to the point-wise convergence of $L_N f(x)$.
To be able to establish the consistency to ${\calL}^{(\alpha)} f$,
we instead consider another limiting regime of $\epsilon$,
namely $\epsilon = \Omega(\varepsilon_\rho)$,
which is $\Omega( (k/N_y)^{2/d})$ up to a factor of $\sqrt{\log N}$ under the optimal scaling of $k$ as in Remark \ref{rk:choice-k-thm-hatrho}, 
 and we take a different approach. 
The following lemma  shows  that substituting $\bar{\rho}$ with $\hat{\rho}$ in $G_\epsilon^{(\rho)} f(x)$ incurs an extra error of $O(\varepsilon_\rho)$  point-wisely.
\begin{lemma}\label{lemma:right-operator-repalce-rho}
Under the same condition of Lemma \ref{lemma:right-operator-3}, in particular, $f$ and $\rho$ are in  $C^4( {\calM})$.
Suppose a positive integrable 
 $\tilde{\rho}: {\calM} \to \R^+$ satisfies that $\sup_{x \in {\calM}} \frac{|\tilde{\rho}(x) - \rho(x)|}{ \rho(x) } < \varepsilon < 0.1$, 
then when the $\epsilon$ in $G_\epsilon^{(\cdot)}$ is sufficiently small,
 \begin{equation}\label{eq:G-R-expansion-tilderho}
 G_\epsilon^{( \tilde{\rho})} f 
 =   G_\epsilon^{(\rho)} f + \tilde{r},
 \quad
 \sup_{x\in {\calM}} |\tilde{r}(x)| 
\le c_{\rho}' \|f\|_\infty \varepsilon
 = O^{[f, \rho]}(\varepsilon),
\end{equation}
where $c_\rho' > 0$ is a constant 
depending on  
$({\calM}, k_0, \rho_{max},\rho_{min})$. 
\end{lemma}

With the lemma, the following two theorems 
prove the  point-wise convergence to the limiting operators of the two graph Laplacians operators,
assuming that $\varepsilon_\rho = o(\epsilon)$.

\begin{theorem}\label{thm:limit-pointwise-rw-3}
Suppose $\hat{\rho}$ 
satisfies that
$\sup_{x \in {\calM}}\frac{ | \hat{\rho}(x) - \bar{\rho}(x) | }{ | \bar{\rho}(x)|} < \varepsilon_\rho < 0.1$,
and as $N \to \infty$, 
\[
\epsilon = o(1), 
\quad \epsilon^{d/2+1} N= \Omega (\log N),
\quad \varepsilon_\rho = o(\epsilon),
\]
then for any $f \in C^\infty({\calM})$, 
when $N$ is sufficiently large  \rev{and the threshold is determined by $(\calM, f,p)$ and  uniform for all $x$}, 
w.p. higher than $ 1- 4N^{-10} $, 
\[
L^{(\alpha)}_{ rw'} f(x) 
= {\calL}^{(\alpha)} f(x)
+ O^{[f,p  ]} \left(   \epsilon, \, \frac{\varepsilon_\rho}{\epsilon} \right)
+O^{[1]} \left(  \rev{ \| \nabla f \|_\infty  } p(x)^{1/d}   
\sqrt{     \frac{\log N}{   N  \epsilon^{d/2+1}}   }  \right),
\]
where the constants in big-$O$ are uniform for all $x \in {\calM}$.
\end{theorem}

\begin{remark}\label{rk:nablaf(x)=0}
\rev{
As shown in the proof of Theorem \ref{thm:limit-pointwise-rw-3},
at $x$ where $\nabla f(x) \neq 0$, 
the variance error can be bounded by
$O^{[1]} \left(  | \nabla f (x)|   p(x)^{1/d}   
\sqrt{     \frac{\log N}{   N  \epsilon^{d/2+1}}   }  \right)$ with the same high probability and the threshold of  large $N$ possibly depends on $x$. 
One can also verify 
$O^{[1]} \left(  ( | \nabla f (x)| +0.1) p(x)^{1/d}   
\sqrt{     \frac{\log N}{   N  \epsilon^{d/2+1}}   }  \right)$ as the variance error bound
for large $N$ with $x$-uniform threshold.
The addition of $0.1$ is to make the factor $( | \nabla f (x)| +0.1) $ uniformly bounded from below 
and prevent the bound to vanish when $\nabla f(x) =0$, and 0.1 can be any other positive constant. 
If the behavior at a point $x$ is of interest, theoretically the variance error can be improved in rate
at $x$ where $\nabla f(x)$ vanishes \cite{singer2006graph}.
As we mainly track the influence of $p(x)$ which may be small at some $x$,
we adopt the $\| \nabla f \|_\infty$ factor in the theorem for simplicity. 
The same applies to the point-wise convergence results in Theorems \ref{thm:limit-pointwise-un-3} and \ref{thm:limit-pointwise-rw-fix-epsilon}.
}
\end{remark}

\begin{theorem}\label{thm:limit-pointwise-un-3}
With notation and condition same  as 
those in Theorem \ref{thm:limit-pointwise-rw-3}, when $N$ is sufficiently large \rev{and the threshold is determined by $(\calM, f,p)$ and uniform for all $x$},
w.p. higher than $ 1- 2N^{-10} $,
\[
L^{(\alpha)}_{ un} f(x) 
= p^{ \frac{ 2 (\alpha-1)}{d}}  {\calL}^{(\alpha)} f(x)
+ O^{[f,p   ]} \left(   \epsilon, \, \frac{\varepsilon_\rho}{ \epsilon } \right)
+O^{[1]} \left(  \rev{ \|\nabla f \|_\infty   } 
p(x)^{\frac{2 \alpha - 1}{d}}   
\sqrt{     \frac{\log N}{   N  \epsilon^{d/2+1}}   }  \right).
\]
\end{theorem}

\rev{In Theorems \ref{thm:limit-pointwise-rw-3} and \ref{thm:limit-pointwise-un-3},}
the error bound has an additional term of $O(\frac{\varepsilon_\rho}{\epsilon})$ compared with that in \cite{berry2016variable,singer2006graph}.
The technical reason is that we use lemma \ref{lemma:right-operator-repalce-rho} to substituting $\hat{\rho}$ with $\bar{\rho}$,
which gives $O(\varepsilon_\rho)$ error at the ``$O(1)$'' level but not at the ``$O(\epsilon)$'' level.
In the proof of Theorem \ref{thm:limit-form},
the  $O(\varepsilon_\rho)$ substituting error takes place at the ``$O(\epsilon)$'' level thanks to the quadratic form.

For the un-normalized graph Laplacian operator, the additional $O(\frac{\varepsilon_\rho}{\epsilon})$  error 
can be removed if we consider the weak convergence,
which can be of interest in certain settings.

\begin{theorem}\label{thm:limit-weak-un}
Suppose $\hat{\rho}$ 
satisfies  that
$\sup_{x \in {\calM}}\frac{ | \hat{\rho}(x) - \bar{\rho}(x) | }{ | \bar{\rho}(x)|} < \varepsilon_\rho < 0.1$,
and as $N \to \infty$, 
\[
\epsilon = o(1), 
\quad 
\epsilon N= \Omega (\log N),
\quad
 \varepsilon_\rho = o(1), 
\]
then for any $\varphi \in C^{\infty}({\calM})$,
 when $N$ is large,
w.p. $> 1- 2N^{-10}$, $p_\alpha := p^{1 + \frac{2(\alpha-1)}{d}}$,
\[
\langle \varphi, L^{(\alpha)}_{ un} f \rangle_p 
= \langle \varphi, \Delta_{p_\alpha} f \rangle_{p_\alpha}
+ O^{ [ p,f,\varphi   ]} \left(  \epsilon, \, \varepsilon_\rho \right)
+O^{[1]} \left(        \|\varphi\|_\infty \| p^{ {\alpha}/{d}} \|_\infty
\sqrt{    \frac{\log N}{   N  \epsilon}   \int p_\alpha |\nabla f |^2  }  \right).
\]
\end{theorem}

Note that the above weak convergence result is only possible for the un-normalized operator,
because the $D^{-1}W$ normalization in the random-walk operator breaks the linearity.

At last, we compare with the graph Laplacian operator $L_N$ defined by fixed bandwidth kernel matrix  \eqref{eq:def-Wbeta-fixepsilon},
namely
\[
L^{(\beta)}_{ \epsilon, rw} f(x) = 
\frac{1}{\epsilon 
\frac{m_2}{2m_0} 
}
\left(\frac{\sum_{j=1}^{N}  
{  k_0 \left(  \frac{\| x - x_j\|^2}{  \epsilon   } \right)}
\frac{f(x_{j})}{\hat{p}(x_j)^{\beta} } }
{\sum_{j=1}^{N} 
{  k_0 \left(  \frac{\| x - x_j\|^2}{  \epsilon  } \right)}
 \frac{1}{\hat{p}(x_j)^{\beta} } }  -f(x) \right).
\]
The counterpart of Theorem \ref{thm:limit-pointwise-rw-3} is the following

\begin{theorem}\label{thm:limit-pointwise-rw-fix-epsilon}
Suppose as $N \to \infty$, 
$\epsilon = o(1)$, $\epsilon^{d/2+1} N= \Omega (\log N)$,
and if $\beta \rev{\neq} 0$, 
the estimated density $\hat{p}$ satisfies that
$\sup_{x  \in {\calM} }\frac{ |\hat{p}(x) - p(x)| }{p(x)} < \varepsilon_p < 0.1$,
and $\varepsilon_p = o(\epsilon)$.
Then for any $f \in C^\infty({\calM})$,
when $N$ is large \rev{and the threshold is determined by $(\calM, f,p,\beta)$ and uniform for all $x$},
\[
L^{(\beta)}_{\epsilon, rw} f(x) 
= \Delta_{p^{2-2\beta}}  f(x)
+ O^{[f,p \rev{, \beta}]} \left(   \epsilon, \, \beta \frac{\varepsilon_p}{\epsilon} \right)
+O^{[1]} \left(  \rev{ \| \nabla f \|_\infty }
p(x)^{-1/2}   
\sqrt{     \frac{\log N}{   N  \epsilon^{d/2+1}}   }  \right).
\]
\rev{
In particular, when $\beta =0$,
the bias error term is reduced to $O^{[f,p]} (  \epsilon )$.
}
\end{theorem}

The counterpart for un-normalized graph Laplacian can be derived similarly and omitted.
The limiting operator is the same as in Theorem \ref{thm:limit-form-fixeps},
and consistent with the result in \cite{coifman2006diffusion}.
Compared 
with Theorem \ref{thm:limit-pointwise-rw-3},
apart from the needed condition on the relative error of $\hat{p}$ (c.f. Remarks \ref{rk:compare-fix-eps-kde} and  \ref{rk:hatp-rel-error}),
the variance error term has a factor of $p(x)^{-1/2}$,
while for self-tuned kernel $W^{(\alpha)}$ the factor is $p(x)^{1/d}$. 
This can be expected because the self-tuned  bandwidth is designed to overcome the difficulty of low data density
by enlarging the kernel bandwidth at those places, 
and our analysis reveals  the effect  by the reduced the variance error at $x$ where $p(x)$ is small. 
Such advantage is supported by experiments on the hand-written digit image dataset in Section \ref{subsec:mnist}.

\begin{figure}[t]
\centering
\hspace{-10pt}
\raisebox{5pt}{
\includegraphics[height=.22\linewidth]{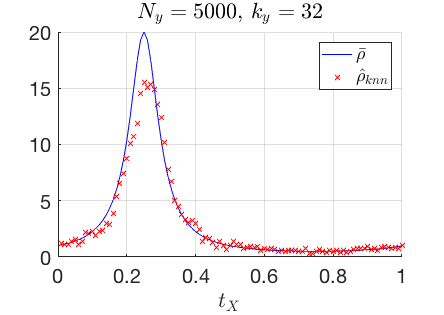} 
\includegraphics[height=.22\linewidth]{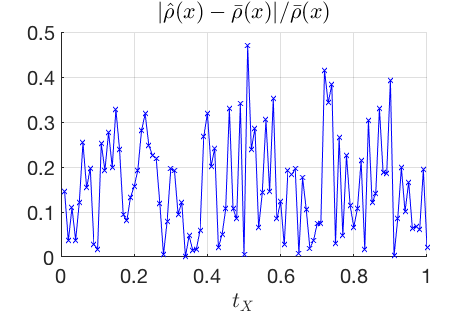} 
}
\includegraphics[height=.25\linewidth]{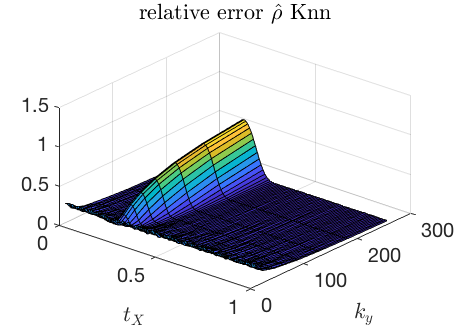}  \\
\vspace{3pt}
\hspace{-10pt}
\raisebox{5pt}{
\includegraphics[height=.22\linewidth]{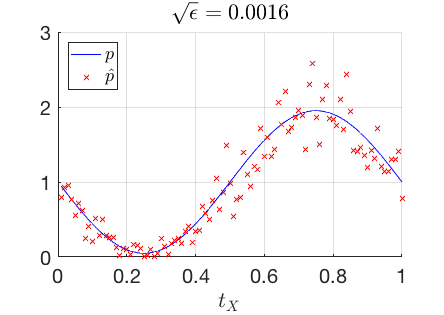}
\includegraphics[height=.22\linewidth]{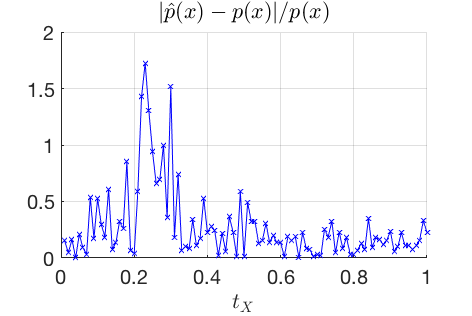}
 }
\includegraphics[height=.25\linewidth]{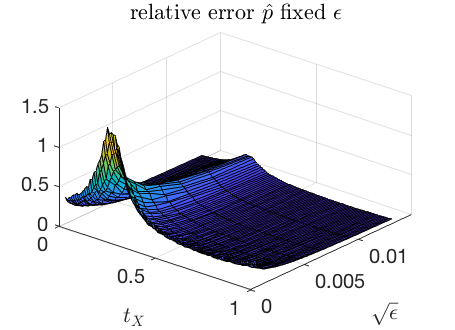}
\vspace{-10pt}
\caption{
\small
kNN estimation of $\bar{\rho} = p^{-1/d}$ 
and KDE estimation of $p$ for $x \in S^1$ \rev{embedded in $\R^2$, and $0 \le t_X \le 1$ is the intrinsic coordinate (arclength)}.
(Top) \rev{The left two plots show a} typical realization of $\hat{\rho}$ by kNN \rev{defined in \eqref{eq:def-knn-hatR} compared with $\bar{\rho}$},
\rev{and the right plot shows the} relative error for varying values of $k_y$,
$N_y = 5000$,  averaged over 500 runs.
(Bottom) Same plots for $\hat{p}$ \rev{defined in \eqref{eq:def-hatp-fixeps} compared with $p$}, 
and relative error for varying values of $\epsilon$.
}
\label{fig:hatrho-hatp-1}
\end{figure}

\begin{figure}[b]
\centering{
\includegraphics[height=.225\linewidth]{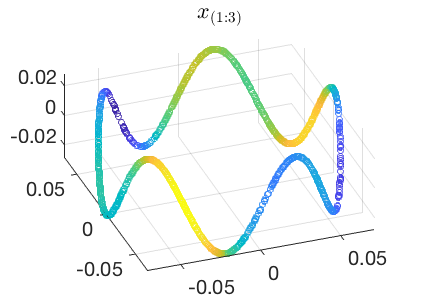} 
\includegraphics[height=.225\linewidth]{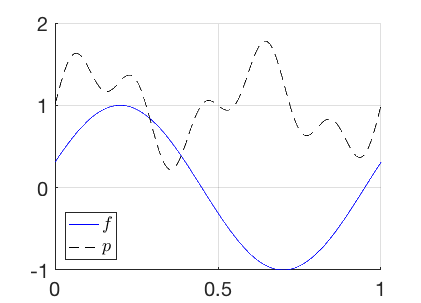}
\includegraphics[height=.225\linewidth]{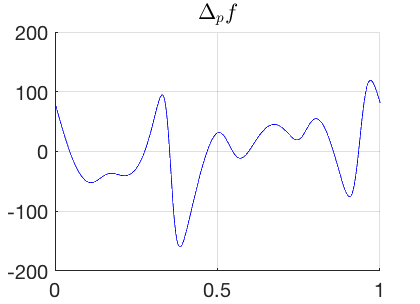}
}\vspace{-5pt}
\caption{
\small
Data in $\R^4$ lying on a 1D closed curve of length 1.
(Left) First 3 coordinates of 2000 samples {with color indicating the density function $p$}.
(Middle) The density {function} $p$ and a function $f$,
and (Right) $\Delta_p f$, all plotted v.s. the intrinsic coordinate (the arclength) on $[0,1]$. 
}
\label{fig:data-1d}
\end{figure}

\section{Numerical Experiments}\label{sec:experiment}

\rev{In this section, we denote the parameter $k$ as $k_y$ when the $k$NN  estimation is conducted on the dataset $Y$. 
In Subsection \ref{subsec:standaloneY}, we use the notation $k_x$ when computing the $k$NN  estimation on $X$.
}

\subsection{kNN Estimator $\hat{\rho}$ of $\bar{\rho}= p^{-1/d}$}\label{subsec:exp-hatrho}

We numerically examine the  kNN estimation of $\bar{\rho}$, \rev{namely $\hat{\rho}$ as defined in \eqref{eq:def-knn-hatR}},
and compare it with the fixed bandwidth KDE estimator $\hat{p}$ as in \eqref{eq:def-hatp-fixeps},
where $ h_{kde} (r) := e^{-r/(4/\pi)} $.
The dataset
is sampled from  a circle of length 1 
{isometrically embedded} in $\R^2$ i.i.d. according to \rev{a} density function $p$, which \rev{equals} 0.05 at $ \rev{t_X} =0.25$,
\rev{$0 \le t_X \le 1$} being the intrinsic coordinate (arclength),
as shown in Fig. \ref{fig:hatrho-hatp-1}. 
The plots on the right {hand side} show the difference of the relative error at place where $p$ is low.
As $k_y$ decreases ($\epsilon$ increases),
the variance error starts to dominate,
 and $\hat{\rho}$ gives the \rev{relative} error uniformly small across locations,
as predicted by Theorem \ref{thm:hatrho}.
\rev{In contrast}, $\hat{p}$ gives a larger \rev{relative} error near $\rev{t_X} = 0.25$. 
\rev{
The result empirically verifies Theorem \ref{thm:hatrho} and Remark \ref{rk:compare-fix-eps-kde}.
}

\subsection{Estimation of Dirichlet Form and $L_N f(x)$}
\label{subsec:experiment-form-and-operator}

\begin{figure}
\hspace{-15pt}
\includegraphics[height=.25\linewidth]{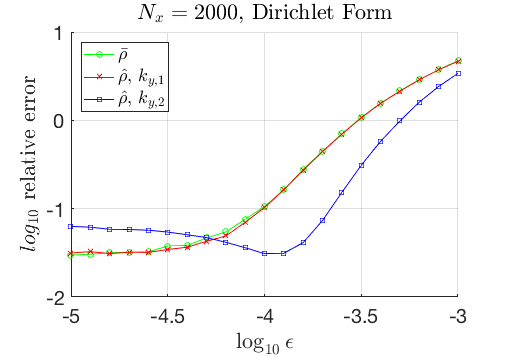}
\hspace{-13pt}
\includegraphics[height=.25\linewidth]{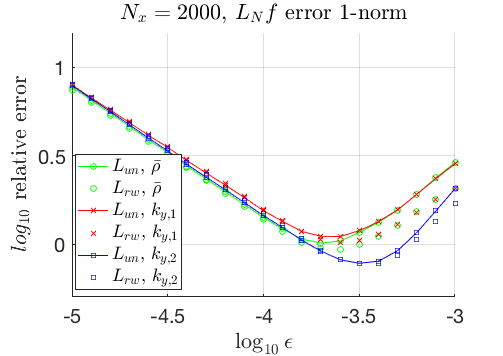}
\hspace{-13pt}
\includegraphics[height=.25\linewidth]{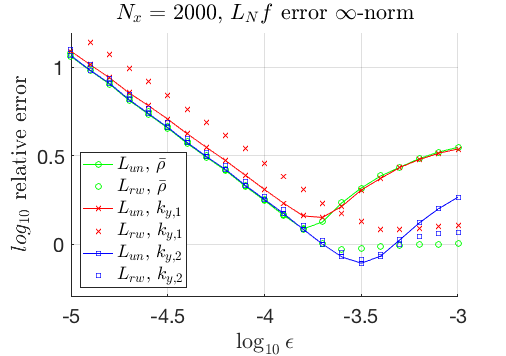}
\vspace{-5pt}
\caption{
\small
Relative error of 
(Left) \rev{Dirichlet} form computed using $L_{un}$
and 
(Middle)
$\text{Err}_1$  error in \eqref{eq:def-L1-err}  of $L_N f$ computed by various $L_N$ 
plotted v.s. a range of values of $\epsilon$,
$N_y=4000$, $k_y = 32$, 256,
averaged over 500 runs.
(Right) Same plot for $\text{Err}_\infty$ error.
}
\label{fig:Ln-error-1d}
\end{figure}

\begin{figure}
\begin{minipage}{1.1\linewidth}
\begin{minipage}{0.3 \linewidth}
\includegraphics[width=.8\linewidth]{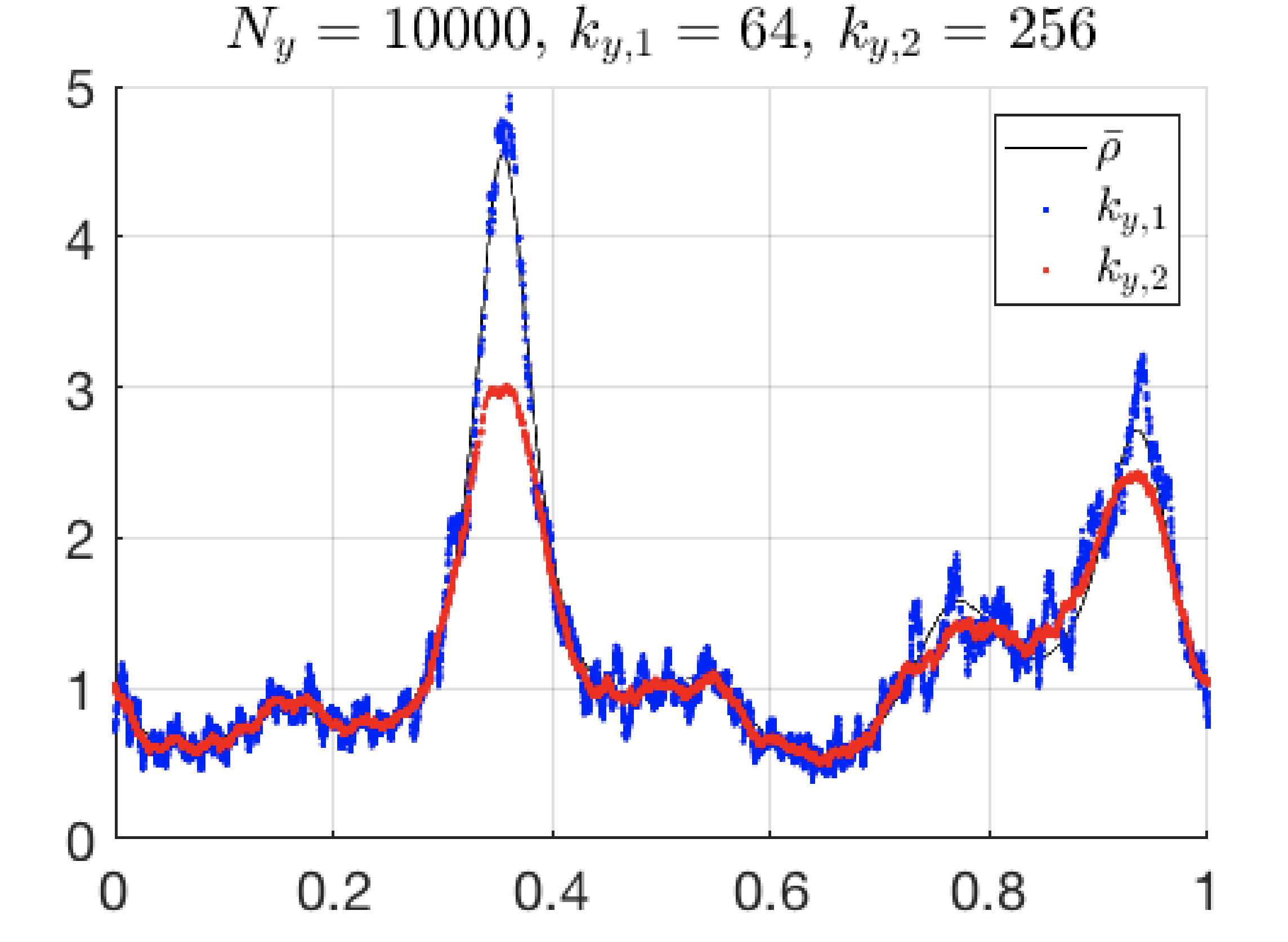}
\includegraphics[width=.8\linewidth]{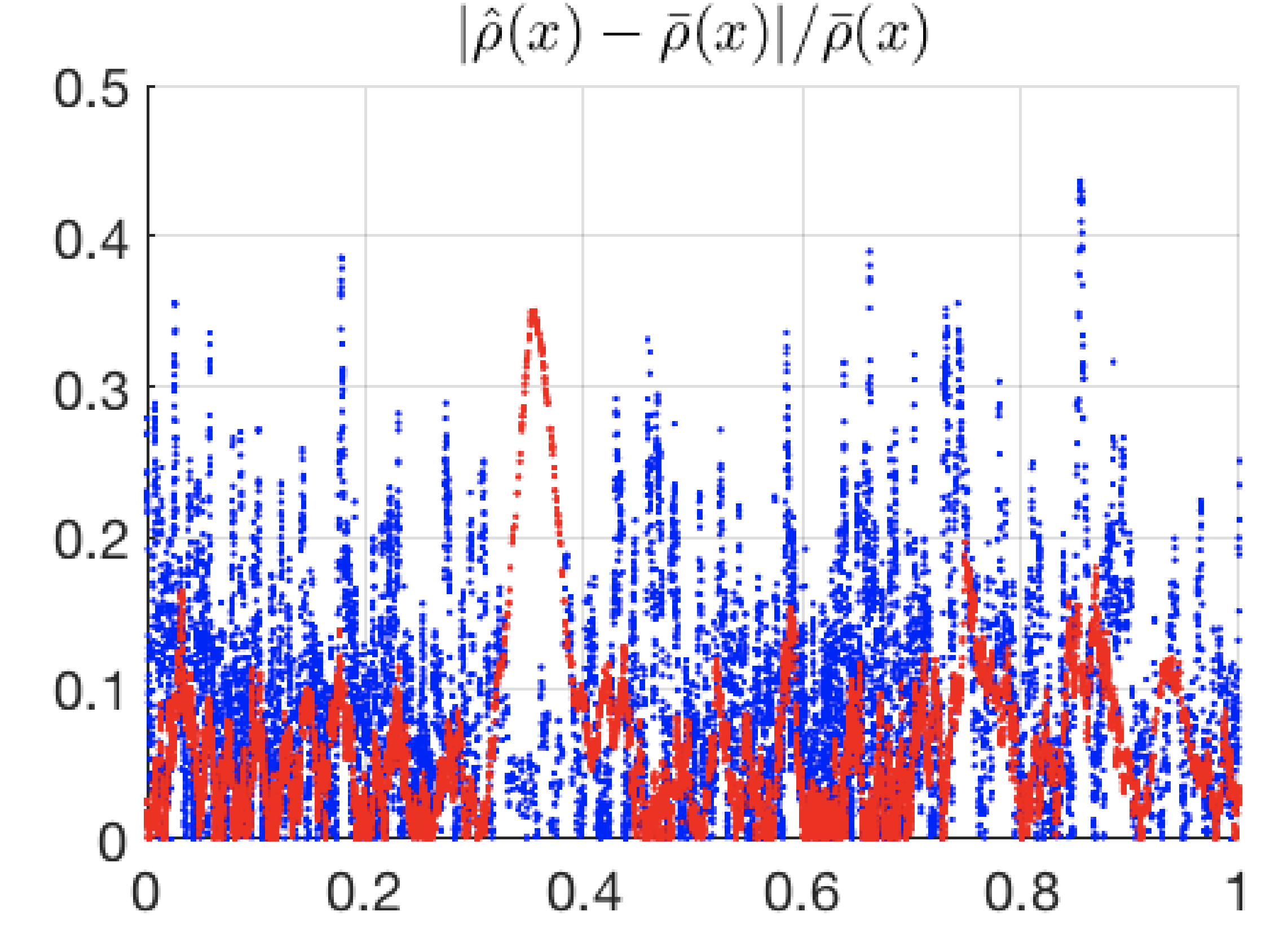}
\end{minipage}
\hspace{-20pt}
\begin{minipage}{0.7 \linewidth}
\includegraphics[width=.45 \linewidth]{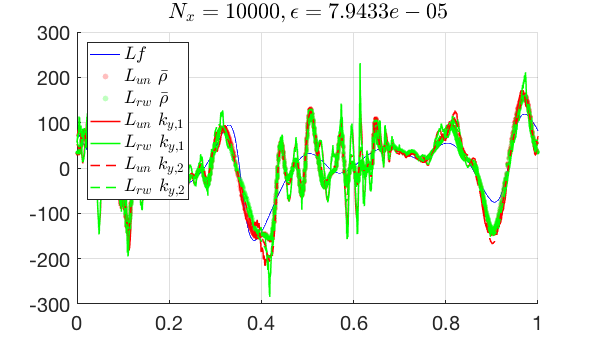}
\includegraphics[width=.45 \linewidth]{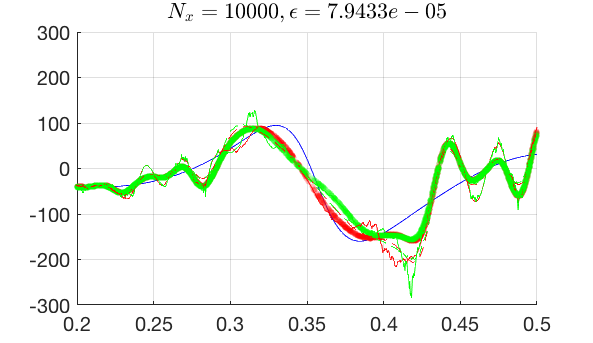}\\
\includegraphics[width=.45 \linewidth]{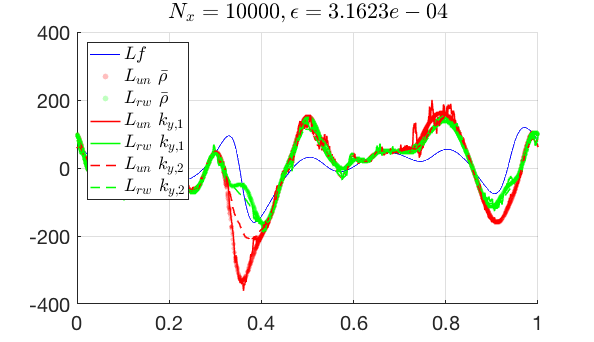}
\includegraphics[width=.45 \linewidth]{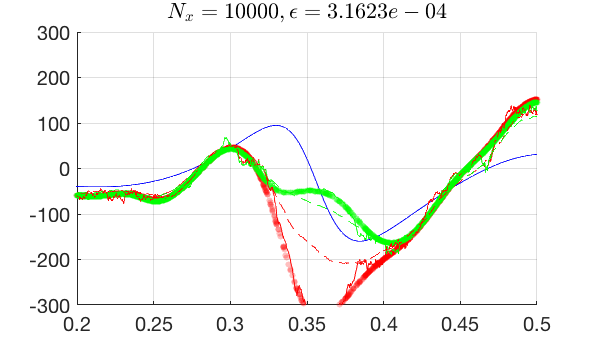}
\end{minipage}
\end{minipage}
\caption{
\small
$L_N f$ by estimated $\hat{\rho}$ \rev{compared with true $\Delta_p f$, denoted as $Lf$ (blue curve)}, with two values of $k_y$.
The data is as in Fig. \ref{fig:data-1d}.
(Left) kNN-estimated $\hat{\rho}$ and relative errors.
(Right upper) Estimated $L_N f$ where $L_N$ equals $L_{un}$ and $L_{rw'}$, and using kNN-estimated $\hat{\rho}$,
compared with using population $\bar{\rho}$.
The right plot is the zoom in of the left plot on interval $[0.2, 0.5]$.
(Right bottom) Same plot as the \rev{right upper} panel at another value of $\epsilon$.
\rev{See more explanation in Section \ref{subsec:experiment-form-and-operator}.}
}
\label{fig:Ln-1d}
\end{figure}

On the simulated data lying on a 1D smooth manifold {embedded} in $\R^4$ with a non-uniform density $p$  (Fig. \ref{fig:data-1d}),
we compute self-tuned graph Laplacians on $N_x= 2000$ data samples using
(1)  $L_{un}$ \eqref{eq:def-Lalapha-un}
and 
(2) $L_{rw'}$ as in \eqref{eq:def-Lalapha-rw},
$\alpha =1$,
where $\hat{\rho}$ is estimated from a stand-alone dataset $Y$ with $N_y = 4000$, and $k_y=32$, $256$ respectively.
To evaluate the influence of $\bar{\rho}$ estimation,
we also compute $L_{un}$ and $L_{rw'}$ where $\hat{\rho}$ is replaced to be the true $\bar{\rho}$.
The relative errors of 
\begin{itemize}
\item
 The Dirichlet form $\langle f, \Delta_p f \rangle_p$,

\item
The  point-wise error measure by 
\begin{equation}\label{eq:def-L1-err}
\text{Err}_1:= \sum_{i=1}^{N_x} |L_N f(x_i) - \Delta_p f(x_i)|,
\quad
\text{Err}_\infty:= \max_{1 \le i \le N_x} |L_N f(x_i) - \Delta_p f(x_i)|,
\end{equation}
\end{itemize}
are given in Fig. \ref{fig:Ln-error-1d}.
The error of $L_N f$ shows a scale of \rev{about} $\epsilon^{-d/4-1/2}=\epsilon^{-0.75}$ \rev{in Fig. \ref{fig:Ln-error-1d} right two plots}, 
when the variance error dominates due to the small value of $\epsilon$. 
In comparison, the accuracy of Dirichlet form is less sensitive to the small value of $\epsilon$,
as shown in Fig. \ref{fig:Ln-error-1d}(Left),
which is consistent with the theoretical result in Theorems \ref{thm:limit-form} and \ref{thm:limit-pointwise-rw-3}.
\rev{
Note that the relative 1-norm error shown in the plot divides $\text{Err}_1$  by $\| \{ \Delta_p f(x_i) \}_{i=1}^{N_x} \|_1$, 
thus its magnitude (about or greater than  1 in Fig. \ref{fig:Ln-error-1d}) depends on the choice of the test function $f$. Same with the relative $\infty$-norm error.
When $N_x$ is increased to be 10,000, with the same $f$, the smallest relative error across $\epsilon$ is about 0.5 (averaged over 20 runs).
}

Taking $N_x = 10,000$, $N_y = 10,000$, with $k_y =64 $, 256, respectively, 
we visualize in Fig. \ref{fig:Ln-1d} 
snapshots of single realizations of $L_N f$.
With smaller value of $\epsilon$, the estimated $L_N f$ has more oscillation around the true value,
and when $\epsilon$ is larger, the oscillation is less 
but the function $L_N f$ is significantly biased at certain places on the manifold.
Note that when $k_y$ is larger, the estimated $\hat{\rho}$ is smoother but has a significant bias at places where $p$ is small,
and such bias is also reflected in the estimated $L_N f$.
The two Laplacians, $L_{un}$ and $L_{rw'}$, give comparable results.

\subsection{The Influence of Stand-alone $Y$}\label{subsec:standaloneY}

We compare with the results using $X$ to estimate $\hat{\rho}$.
The dataset is the same as that in Fig. \ref{fig:data-1d}.
 $N_x = 2000$ 
{and}  $k_x = 32$ 
{are} used to compute $\hat{\rho}_X$.
{Take} $N_y = \{ 2000, 4000, \cdots,  32000\}$ 
{and}  $k_y = \{37, 64, \cdots, 338 \}$, where
$k_y$ is chosen to scale as $N_y^{4/5}$, according to Theorem \ref{thm:hatrho} ($d=1$).
The result for one realization with the largest $N_y$ is in Fig. \ref{fig:Ln-hatrhoX},
where using a stand-alone $Y$ of a much larger size than $X$ reduces the error in the estimated $\hat{\rho}$
as well as the oscillation in the estimated $L_{N} f$ (plots for $L_{rw'} f$ are similar and not shown).
The relative errors of Dirichlet form and  $\text{Err}_\infty$ of $L_{N} f $
across $\epsilon$ are shown in Fig. \ref{fig:Ln-error-hatrhoX},
where using $\hat{\rho}_X$ and $\hat{\rho}_Y$ give comparable accuracy. 
Moreover, the {result with $\hat{\rho}_Y$} approaches  {$L_N f$} computed from $\bar{\rho}$ as $N_y$ and $k_y$ increase.
This suggests that when \rev{significantly} more data samples than $N_x$ are available,
using the rest as $Y$ to estimate the bandwidth function $\hat{\rho}$ may improve the estimation of the self-tuned graph Laplacian. 
\rev{
With limited number of data samples, splitting stand-alone $Y$ may worsen the performance (due to decreasing $N_x$)
than using  the whole dataset as $X$ and estimating the bandwidth on itself.
}

\begin{figure}[b]
\begin{minipage}{1.1\linewidth}
\hspace{-55pt}
\includegraphics[height=.16\linewidth]{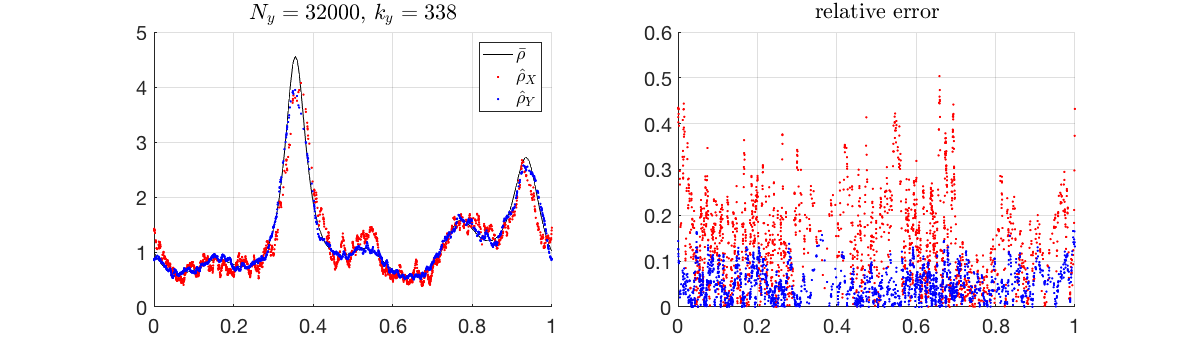}
\hspace{-20pt}
\includegraphics[height=.16\linewidth]{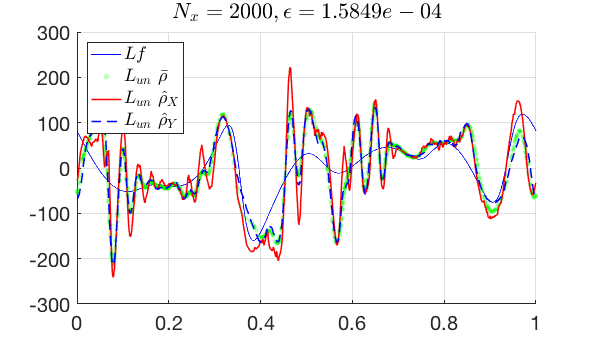}
\hspace{-10pt}
\includegraphics[height=.16\linewidth]{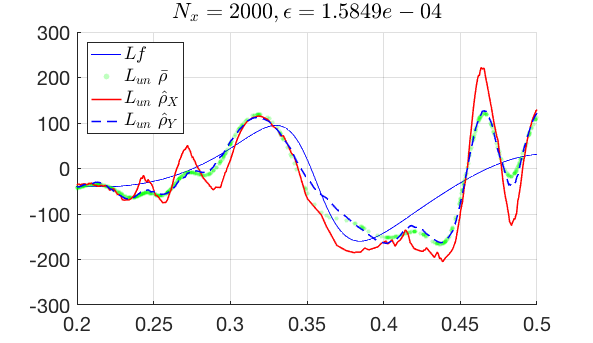}
\end{minipage}
\vspace{-5pt}
\caption{
\small
Same plots as Fig. \ref{fig:Ln-1d} for self-tune Laplacian computed with $\hat{\rho}_X$.
(Left two) kNN-estimated $\hat{\rho}$ and relative errors computed from $X$ and $Y$.
(Right two) Estimated $L_{un} f$,
with the zoom on the interval $[0.2, 0.5]$.
}
\label{fig:Ln-hatrhoX}
\vspace{5pt}
\end{figure}

\begin{figure}
\hspace{-15pt}
\includegraphics[height=.25\linewidth]{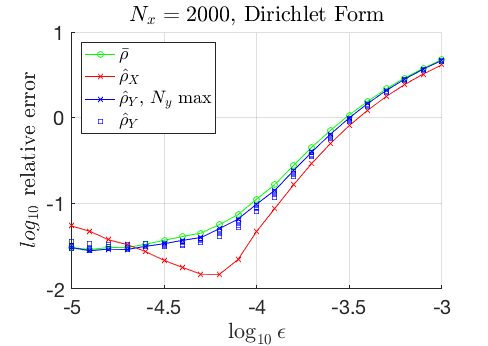}
\hspace{-13pt}
\includegraphics[height=.25\linewidth]{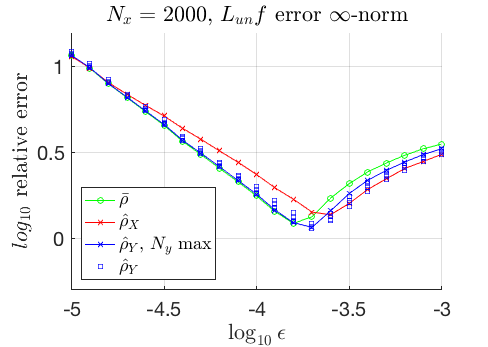}
\hspace{-13pt}
\includegraphics[height=.25\linewidth]{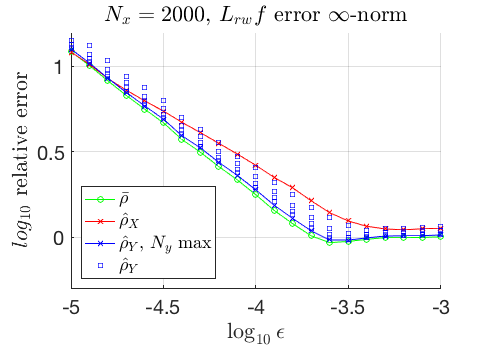}
\vspace{-10pt}
\caption{
\small
Relative error of 
(Left) \rev{Dirichlet} form computed using $L_{un}$,
and 
(Middle) $\text{Err}_\infty$ error of $L_{un}$ computed using $\bar{\rho}$,
$\hat{\rho}_X$, and $\hat{\rho}_Y$ over a range of $\epsilon$ and different $\{ N_y, k_y \}$ (blue squares),
averaged over 500 runs.
(Right) Same plot for $L_{rw'}$. 
}
\label{fig:Ln-error-hatrhoX}
\end{figure}

\begin{figure}[b]
\begin{minipage}{1.02\linewidth}
\hspace{-17pt}
\includegraphics[trim=30 5 10 3, clip, height=.181\linewidth]{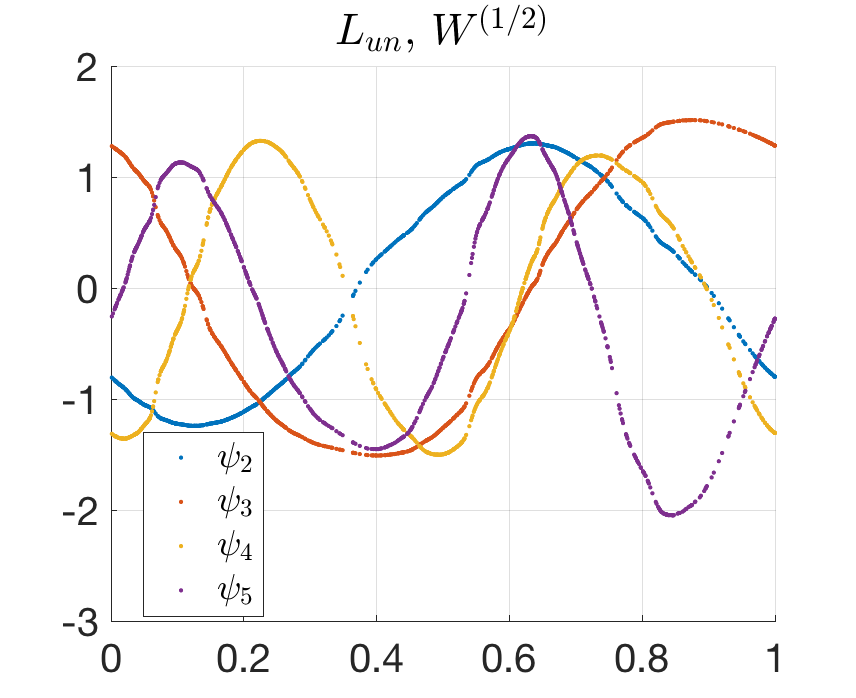}
\hspace{-10pt}
\includegraphics[trim=30 5 10 3, clip, height=.181\linewidth]{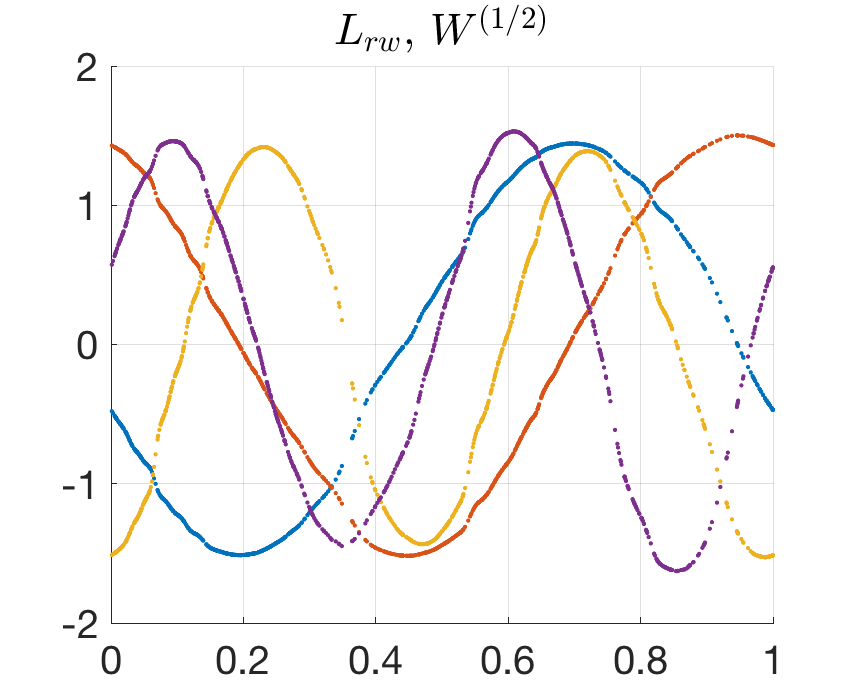}
\hspace{-10pt}
\includegraphics[trim=30 5 10 3, clip,height=.181\linewidth]{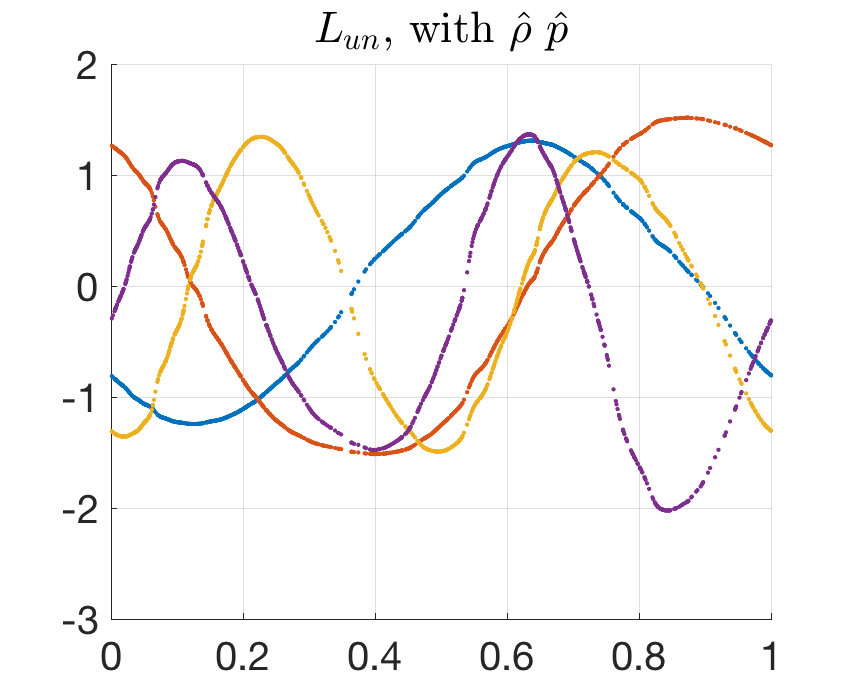}
\hspace{-10pt}
\includegraphics[trim=30 5 10 3, clip,height=.181\linewidth]{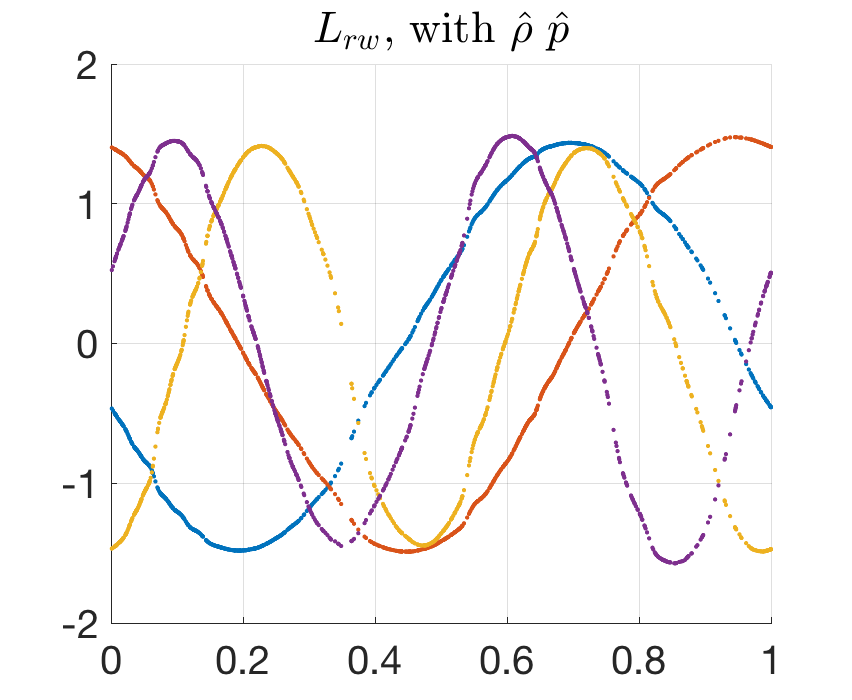}
\hspace{-10pt}
\includegraphics[trim=30 5 10 3, clip,height=.181\linewidth]{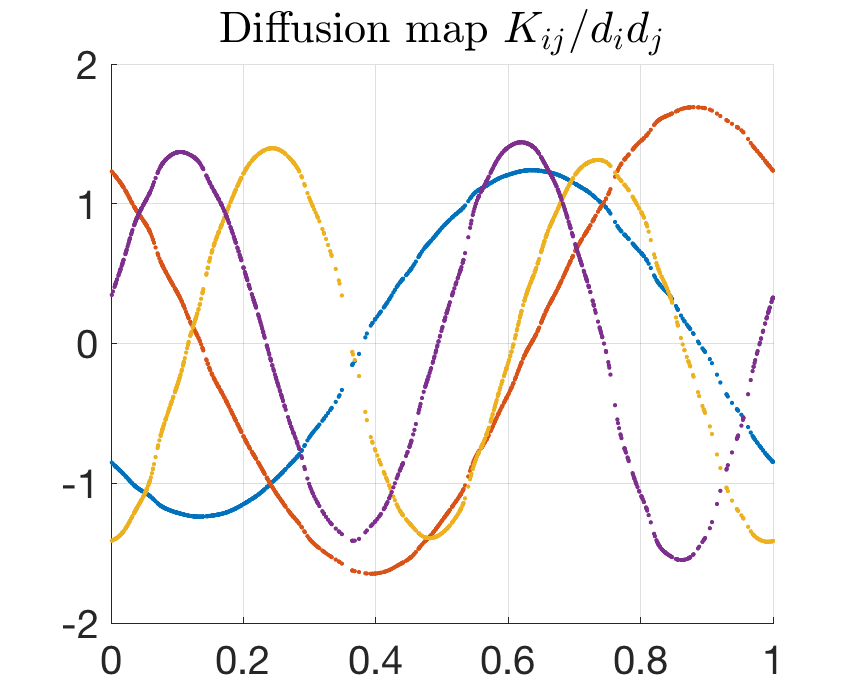}
\end{minipage}
\hspace{-10pt}
\caption{
\small
Plots of first 4 (non-trivial) eigenvectors of graph Laplacians which approximate eigenfunctions of $\Delta_{\calM}$ of $S^1$.
From left to right:
$L_{un}$,  $L_{rw'}$ using $W^{(1-\frac{d}{2})}$;
$L_{un}$,  $L_{rw'}$ using \eqref{eq:W-mixed-normalize};
degree $d_i^{-1}d_j^{-1}$normalized Diffusion Map \cite{coifman2006diffusion}.
Data as in Fig. \ref{fig:data-1d}, 
$N_x = 1000$,
the kNN self-tune bandwidth $\hat{\rho}$ computed from $X$ with $k_x = 21$,
$\epsilon = 1e-4$.
}
\label{fig:eig-S1}
\end{figure}

\subsection{Recovery of the Laplace-Beltrami Operator}\label{subsec:exp-d-not-known}

According to the theory, the limiting operator is $\Delta_{\calM}$ when $\alpha = 1-\frac{d}{2}$.
Here we examine two ways to recover $\Delta_{\calM}$:

(1) By the self-tuned kernel affinity $W^{(1-\frac{d}{2})}$.

(2) By a normalization with 
{a combination of $\hat{p}$ and $\hat{\rho}$}, 
\begin{equation}\label{eq:W-mixed-normalize}
W_{ij} = \frac{k_0 \left(  \frac{ \|x_i - x_j \|^2 }{\epsilon \hat{\rho}(x_i)  \hat{\rho}(x_j)  }\right)}{ (\hat{\rho}\hat{p}^{1/2}) (x_i)  (\hat{\rho}\hat{p}^{1/2}) (x_j)}
=  \frac{ W^{(1)}_{ij} }{ \hat{p}^{1/2} (x_i) \hat{p}^{1/2} (x_j) },
\end{equation}
because $\bar{\rho}^{-d/2} = p^{1/2}$. 
The second approach does not need prior knowledge or estimation of the intrinsic dimensionality $d$,
and thus can be applied in more general scenarios.

Consider the same
1D manifold data 
used in Fig. \ref{fig:data-1d}. Take
$N_x = 1000$, $\epsilon = 10^{-4}$,
and  compute $\hat{\rho}$ from $X$.
The embeddings by the first 4 (non-trivial) eigenvectors of 
various graph Laplacians are shown in Fig. \ref{fig:eig-S1}, 
where
the last column shows the 
result produced by 
affinity matrix $W_{ij} = \frac{K_{ij}}{d_i d_j}$, where
$K_{ij}$  is the fix-bandwidth kernel affinity $K_{ij} = k_0(\frac{ \|x_i - x_j\|^2}{\epsilon})$ and $d_i = \sum_j K_{ij}$, as in \cite{coifman2006diffusion}.
Recall that the eigenfunctions of the Laplace-Beltrami operator are 
{sine and cosine functions with different frequencies.}
In this example,
the random-walk graph Laplacian produces a visually
better eigenfunction approximation 
compared with the unnormalized graph Laplacian.
We postpone 
the study of the random-walk graph Laplacian with self-tuned kernel to future investigation.

\begin{figure}[t]
\includegraphics[trim= 360 0 310 0, clip,height=.18\linewidth]{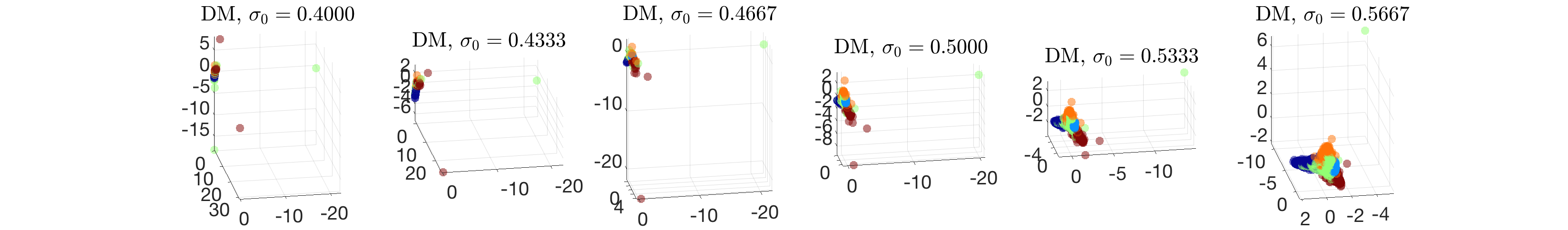} 
\includegraphics[trim= 420 0 300 0, clip,height=.18\linewidth]{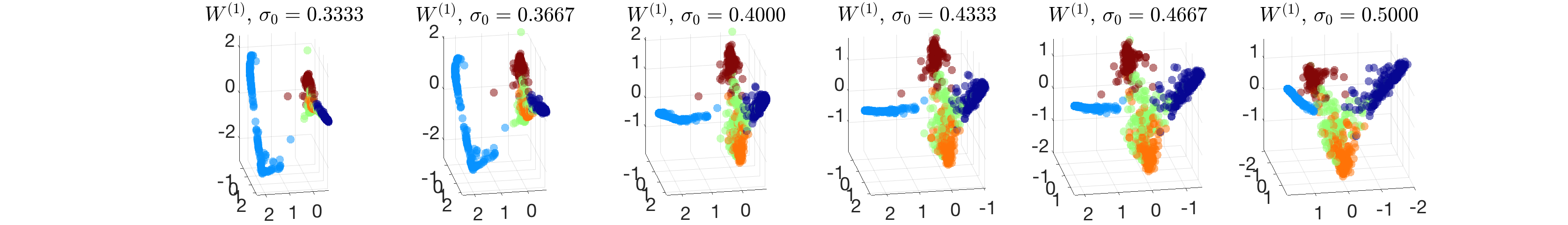}  
\includegraphics[trim= 420 0 300 0, clip,height=.18\linewidth]{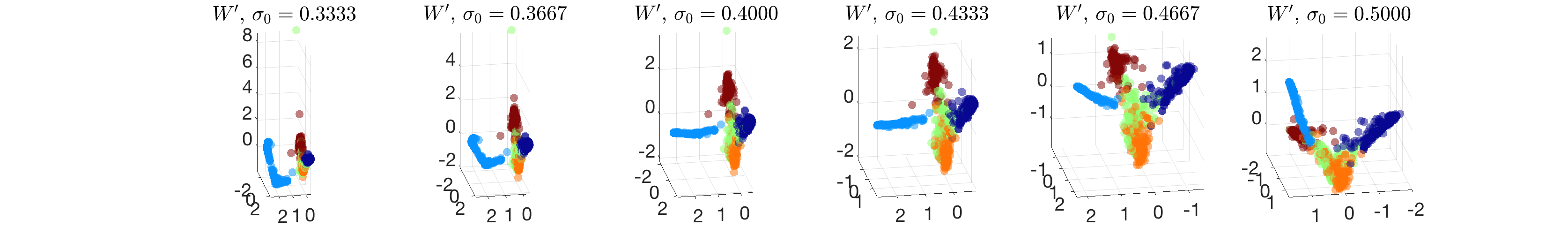}
\vspace{-5pt}
\caption{
\small
Eigenvector embedding of $N_x=$1000 MNIST hand-written digit images of 5 classes, $k_x=7$,
colored by digit class labels.
(Top) By $\frac{K_{ij}}{d_i d_j}$ with fixed-bandwidth kernel,
(Middle) by self-tuned kernel $W^{(1)}$,
(Bottom) by self-tuned kernel $W^{'}$,
as defined in \eqref{eq:W-mnist}. 
}
\label{fig:mnist}
\end{figure}

\begin{figure}[b]
\centering{
\includegraphics[height=.18\linewidth]{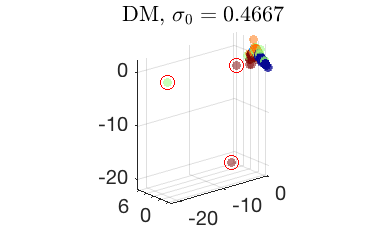}
\includegraphics[height=.18\linewidth]{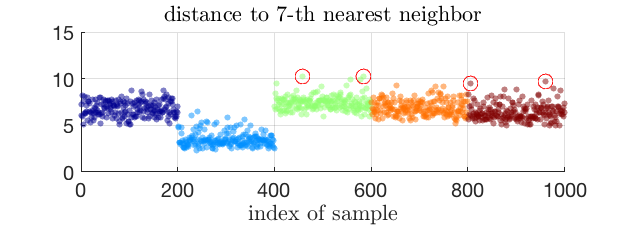}
}
\vspace{-5pt}
\caption{
\small
Outliers in the fixed-bandwidth kernel embedding.
(Left) 
One embedding in the top panel of Fig. \ref{fig:mnist} 
{with a proper rotation for visualization purpose},
where the outlier samples are marked with red circles.
(Right)
Values of $\hat{R}_i$,
i.e. the distance to the 7-th nearest neighbor,
 of the $N_x=$1000 samples.
 The outlier samples in the left plot are marked by red circles.
The plot is colored by digit class labels. 
}
\label{fig:mnist-low-density}
\end{figure}

\subsection{Embedding of Hand-written Digits Data}\label{subsec:mnist}

We implement the embedding on $N_x = 1000$ samples from the MNIST dataset, containing 5 classes (digits `0', `1', \ldots, `4') 
with 200 images in each class. 
The hand-written images can be viewed as lying near certain low-dimensional sub-manifolds in the $\R^{784}$ ambient space (each sample is a 28$\times$28 gray-scale image). 
We use $k_x = 7$,
and compute $\hat{R}_i $ as the $L^2$ distance 
between the $i$-th image to its $k_x$-th nearest neighbor.
Also compute $\hat{\mu}_i := \frac{1}{N} \sum_{j=1}^N h_{kde} ( \frac{ \|x_i -x_j\|^2}{\epsilon_{kde}} )$,
where $\epsilon_{kde}^{1/2}  = \text{Median}_{i} \{\hat{R}_i\}$. 
{Here,}
$\hat{\mu}$ is the (un-normalized) density estimator. 
{Consider two} self-tuned kernel affinities: 
\begin{equation}\label{eq:W-mnist}
W_{ij}^{(1)} = k_0\left( \frac{ \|x_i - x_j\|^2}{ \sigma_0^2 \hat{R}_i  \hat{R}_j} \right) \frac{1}{ \sigma_0^2 \hat{R}_i  \hat{R}_j},
\quad
W_{ij}^{'} = k_0\left( \frac{ \|x_i - x_j\|^2}{ \sigma_0^2 \hat{R}_i  \hat{R}_j} \right) \frac{1}{ \sigma_0^2 \hat{R}_i  \hat{R}_j \sqrt{ \hat{\mu}_i   \hat{\mu}_j}}\,.
\end{equation}
We use $L_{rw'} = D_{\hat{R}}^{-2}(D^{-1}W-I)$, {where} $W = W^{(1)}$ or $W^{'}$, $(D_{\hat{R}})_{ii} = \sigma_0 \hat{R}_i $,
and $D$ is the degree matrix of $W$.  
The parameter $\sigma_0^2$ serves as the dimension-less bandwidth ``$\epsilon$''.
We also compare with the fixed-bandwidth kernel affinity matrix, where $ \epsilon^{1/2} = \sigma_0 \text{Median}_i( \hat{R}_i)$,
called the Diffusion Map (DM) embedding.
The embeddings by the first 3 (non-trivial) eigenvectors over a range of values of $\sigma_0$ are shown in Fig. \ref{fig:mnist}.

We observe that
the DM embedding  
is disconnected at small value of $\sigma_0$
and
consists of points which are far away from the bulk (outlier points),
 due to sensitivity to data points which are relatively farther away from its neighbor samples.
As illustrated in Fig. \ref{fig:mnist-low-density},
the outlier points in the DM embedding
are those whose values of $\hat{R}_i$ are large.
In comparison, 
both the self-tuned kernels
provide informative embeddings of the dataset over the range of values of $\sigma_0$,
showing improved stability at small values of $\sigma_0$
to data samples lying at places where the data density is low.
The $W^{(1)}$ kernel affinity shows a better stability than the $W'$ kernel at the small value of $\sigma_0$
on this dataset,
due to that the $W'$ kernel still involves a fixed-bandwidth KDE $\hat{\mu}$ in the normalization.

\section{Proofs}\label{sec:proofs}

\subsection{Proofs  in Section \ref{sec:knn-hatrho}}

%
%
\begin{proof}[Proof of Lemma \ref{lemma:knn-rx}]
Given $Y$ and $k>1$ fixed, define
\[
{\calS}_Y := \left\{ x \in \R^D, \, s.t. \, \exists j \neq j', \, \|x-y_{j}\| = \|x - y_{j'}\| \right\}\,.
\]
Since $y_j$'s are distinct points, ${\calS}_Y$ is a collection of finitely many hyperplanes in $\R^D$ (finitely many points when $D=1$),
and  ${\calS}_Y \cap Y  = \emptyset$.
Whenever $x$ lies outside ${\calS}_Y$,
the set $ \{ \| x-y_j\| \}_{j=1}^N$ consists of distinct non-negative values.
The set $\R^D \backslash {\calS}_Y$ is open and consists of a finite union of polygons (the polygons can be  unbounded),
as illustrated in Fig. \ref{fig:diag-knn}.

We prove the lemma in three parts as below.

\vspace{5pt}
\noindent
~
\underline{Part 1}: To prove that $\hat{R}$ is piece-wise $C^\infty$ on $\R^D \backslash {\calS}_Y$,
and on each polygon $\mathbf{p}$ in $\R^D \backslash {\calS}_Y$,
$\hat{R}(x) = \| x - y_{\mathbf{p}}\|$ for a point $y_{\mathbf{p}} \in Y$ and outside $\mathbf{p}$.

First, for each (open) polygon $\mathbf{p}$ and any $x \in \mathbf{p}$,
 the $k$-th nearest neighbor (kNN) of $x$ in $Y$ is uniquely defined due to the fact that the distance list $ \{ \| x-y_j\| \}_{j=1}^N$ has distinct values.
Thus 
the function $\hat{R}(x)$ equals $\|x - y^{(k,x)}\|$, where $ y^{(k,x)}$ is the kNN of $x$ in $Y$.

Second, we claim that the point $ y^{(k,x)}$ is the same $y \in Y$ for all $x$ inside the polygon $\mathbf{p}$, 
because the ordered list of nearest neighbors is fixed for all $x$ within $\mathbf{p}$. 
Indeed, for the ordered list to cross, the distances of $\|x - y_{j}\|$ and $\|x - y_{j'}\|$ need to be equal at some $x$, and this $x$ lies on ${\calS}_Y$. 
We call this point $ y_{\mathbf{p}}$, and then $\hat{R}(x) = \|x - y_{\mathbf{p}}\|$ for $x \in \mathbf{p}$.

Third, we claim that $y_{\mathbf{p}} \notin \mathbf{p} $.
Note that each polygon $\mathbf{p}$ has at most one point $y_j $ inside it.
Because otherwise, suppose $y_j \neq y_{j'} $  are both inside $\mathbf{p}$, then so is the middle point $\frac{ y_j + y_{j'}}{2} $
due to that $\mathbf{p}$ is convex,
but $\frac{ y_j + y_{j'}}{2} $ is in ${\calS}_Y$ and cannot intersect with $\mathbf{p}$. 
Now if $y_{\mathbf{p}} \in \mathbf{p} $, then by definition $y_{\mathbf{p}}$ is the kNN of itself,
which means that $k = 1$. This contradicts with the condition that $k > 1$.

The above gives us that $ \hat{R}(x)=\| x - y_{\mathbf{p}}\|$ is  $C^\infty$  and hence $\| \bar{\nabla} \hat{R} \| =1$ inside $\mathbf{p}$,
by the fact that the mapping $x \mapsto \|x\|$ is $C^\infty$ on $\R^D \backslash \{0\}$.
These properties hold for all polygons $\mathbf{p}$,
 thus $\hat{R} (x)$ is $C^\infty$ on $\R^D \backslash {\calS}_Y$,
 and $\| \bar{\nabla} \hat{R} \| =1$ at point of differentiability.

\begin{figure}[t]
\centering{
\includegraphics[height=.35\linewidth]{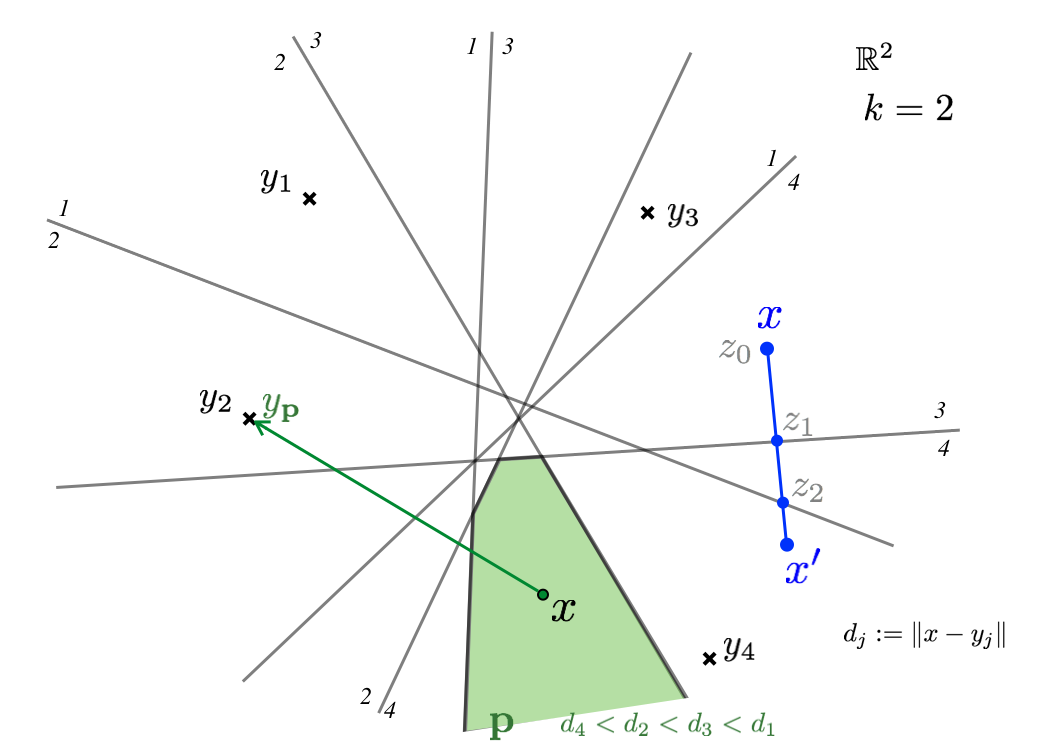}
}
\vspace{-5pt}
\caption{
\small
Illustration of the set ${\calS}_Y$ and example polygons in
the proof of Lemma \ref{lemma:knn-rx}, 
$D=2$, $Y = \{ y_1, \cdots ,y_4\}$, $k=2$, 
For each polygon $\mathbf{p}$, there is a point 
$y_{\mathbf{p}} \in Y$ such that 
$\hat{R}(x) = \| x - y_{\mathbf{p}}\|$ for $x \in \mathbf{p}$.
}
\label{fig:diag-knn}
\end{figure}

\noindent
~
\underline{Part 2}: To prove that $\text{Lip}_{\R^D} (\hat{R}) \le 1$. 

We assume that $\hat{R} $ is continuous $\R^D$, which will be proved in Part 3.
By Part 1, we have that $\hat{R}$ is Lipschitz 1 on each open polygon $\mathbf{p}$, 
and combined with the continuity of $\hat{R}$ at points on the boundary of $\mathbf{p}$,
we have that $\text{Lip}_{\bar{\mathbf{p}}} (\hat{R}) \le 1$, 
where $\bar{\mathbf{p}}$ is the closure of $\mathbf{p}$,
that is,
\begin{equation}\label{eq:Lip-p-closure}
| \hat{R}(z) - \hat{R}(z') |  \le \| z - z'\|,
\quad 
\forall z, z' \in \bar{\mathbf{p}}.
\end{equation}

For two points $x \neq x'$ in $\R^D$, we want to show that $ | \hat{R}(x) - \hat{R}(x') |  \le \| x - x'\|$. 
Consider the segment line $l$ connecting the two points. 
If $l$ is contained in some $\bar{\mathbf{p}}$,
then the claim is proved. 
Otherwise, 
there is a sub-segment $l_0$ connecting from $x$ and $z_1 \in {\calS}_Y $ such that $l_0$ is in some $\bar{\mathbf{p}}$.
Continue the process gives  finitely many distinct points $\{ z_1, \cdots, z_M\} \subset l$
such that the sub-segment $l_i$ connecting from $z_{i}$ to $z_{i+1}$ is contained in some $\bar{\mathbf{p}}$
for $i=0$ to $M$,
where $z_0 = x$
and 
$z_{M+1} = x'$.
Note that this decomposition of $l$ into the union of $l_i$'s holds even when one or both of $x$ and $x'$ are in ${\calS}_Y$,
as illustrated in Fig. \ref{fig:diag-knn}.

Now by construction,  $\| x - x'\| = \sum_{i=0}^{M} \| z_i -   z_{i+1} \|$. 
Meanwhile, applying \eqref{eq:Lip-p-closure} to each $l_i$ gives that 
$| \hat{R} (z_i)-\hat{R} (z_{i+1}) | \le \| z_i - z_{i+1} \|$. 
Thus 
 \[
 |\hat{R}(x) - \hat{R}(x')| \le  \sum_{i=0}^{M} |\hat{R}(z_i) - \hat{R}( z_{i+1})|
\le  \sum_{i=0}^{M} \| z_i -   z_{i+1} \| = \| x - x'\|.
\]

\vspace{5pt}
\noindent
~
\underline{Part 3}: To prove that $\hat{R}$ is continuous on $\R^D$.

To finish the proof,
it remains to prove the continuity of $\hat{R} $ on $\R^D$. For any $x_0 \in \R^D$, 
let $\hat{R} (x_0) = r_0$. Since $Y$ has distinct points by assumption, at most one point $y_j$ coincides with  $x_0$. Since $k > 1$, $r_0 > 0$.
We prove that when $x \to x_0$, $\hat{R}(x) \to r_0$.
Define
\[
F(x,r):= \sum_{j=1}^N {\bf 1}_{\{ \|x-y_j\| < r \}}\,.
\]
Recall that
\[
\hat{R} (x) = \inf \{ r > 0, \text{ s.t. }  F(x,r) \ge k \}\,.
\]
Since $F(x_0, r)$ is monotonically increasing as $r$ increases,
for any $r' := r_0 +\varepsilon > r_0$, 
$F(x_0,r') \ge k$. This means that $| Y \cap B_{r'}(x_0) | :=k' \ge k $.
Since $B_{r'}(x_0) $ is an open ball,  and there are $k'$ many $y_j$'s lying inside it,
they also all lie inside $B_{r''}(x_0) $ where $r_0<r'' < r'$.
Thus when $\|x - x_0\| < (r'-r'')/2:= r'''$,  these $k'$ points of $y_j$ also lie inside $B_{r'}(x)$,
then $F(x, r') \ge k' \ge k$. This gives that $ \hat{R} (x) \le r_0 +\varepsilon$, whenever $\|x - x_0\| < r'''$.

Meanwhile, for any $ 0 < r' := r_0 - \varepsilon < r_0$, by definition $F(x_0, r'+\frac{\varepsilon}{2}) < k$,
i.e., $| Y \cap B_{r'+\frac{\varepsilon}{2}}(x_0) | :=k' < k $.
This means that for any $y \in Y\backslash B_{r'+\frac{\varepsilon}{2}}(x_0)$, 
the distance $\|y - x_0\| \ge r'+\frac{\varepsilon}{2}$. 
Thus, when $\|x - x_0\| < \frac{\varepsilon}{4}$, 
the smallest distance $\| y - x \|$ for any $y \in Y\backslash B_{r'+\frac{\varepsilon}{2}}(x_0)$ 
is $\ge  r'+\frac{\varepsilon}{4}$,
and then $F(x,  r') \le k'  < k$.
This shows that $\hat{R}  (x) \ge r_0 - \varepsilon$, whenever  $\|x - x_0\| < \frac{\varepsilon}{4}$. 
Putting together, this proves the continuity of $\hat{R} (x)$ at $x_0$.
\end{proof}

%
%
%
\begin{proof}[Proof of Proposition \ref{prop:hatrho-one-point}]
Recall that $\bar{\rho}(x)=p(x)^{-1/d}$. Define
\begin{equation}\label{eq:def-barR}
\bar{R}(x) := \bar{\rho}(x) \left(  \frac{1}{m_0[h]}\frac{k}{N} \right)^{1/d}\,.
\end{equation}
Then, since we have $\hat{\rho}(x) = \hat{R}(x) \left(  \frac{1}{m_0[h]}\frac{k}{N} \right)^{-1/d}$
and $\bar{\rho}(x) = \bar{R}(x) \left(  \frac{1}{m_0[h]}\frac{k}{N} \right)^{-1/d}$,
the proposition can be equivalently proved 
by controlling 
$\frac{ |\hat{R}(x) - \bar{R}(x) |}{\bar{R}(x)}$.
For the given $s>0$,  
define
\begin{equation}\label{eq:def-deltar}
\delta_r := t_1 \left(\frac{k}{N}\right)^{2/d} + \frac{t_2}{d} \sqrt{\frac{s \log N}{k}}\,,
\end{equation}
where 
$t_1 =\Theta^{[p]} (1)$, $t_2 =\Theta^{[1]} (1)$, 
both will be determined later.
We will show that,
\rev{when $N$ exceeds a threshold depending on $(p,s)$,}
for any $x \in {\calM}$ fixed, 
w.p. greater than $1-2 N^{-s/4}$, 
\begin{equation}\label{Proposition2.2 Proof first bound}
\bar{R}(x)(1 - \delta_r) \le  \hat{R}(x)  \le \bar{R}(x)(1 + \delta_r).
\end{equation}
To \rev{prove \eqref{Proposition2.2 Proof first bound}}, 
we introduce some notations. Denote
\[
R_{-}(x) := \bar{R}(x)(1 - \delta_r),
\quad\mbox{and}\quad 
R_ +(x) = \bar{R}  {(x)}(1+ \delta_r)\,.
\]
\rev{Let $h = {\bf 1}_{[0,1)}$,}
and define, for any $x\in\calM$ and $r >0$,
\[
\hat{\mu}(x,r) := \frac{1}{N} \sum_{j=1}^N h \left(  \frac{ \| x- y_j\|^2 }{r^2}\right) 
=: \frac{1}{N} \sum_{j=1}^N H_j(x,r)\,,
\]
\rev{then, by \eqref{eq:def-knn-hatR},
$\hat{R} (x) = \inf_{r} \left\{ r > 0,\, \text{ s.t. }  \hat{\mu}(x,r) \ge \frac{k}{N} \right\}$}.
For fixed $x$ and $r$, $H_j$ are i.i.d. random variables, and
\begin{equation}
\E H_j(x,r) = \int_{\calM} h \left(  \frac{ \| x- y\|^2 }{r^2}\right)  p(y) dV(y) =: \mu(x,r).
\end{equation}
Below, to simplify notation, we omit the dependence on $x$ in  $\bar{R}$, $R_\pm$ and $H_j$ when there is no confusion.
The argument is for a fixed $x$, and we make sure that the constants $t_1$ and $t_2$ in $\delta_r$
\rev{as well as the large-$N$ threshold}
 are uniform for all $x$.

\begin{figure}[t]
\centering{
\includegraphics[height=.23\linewidth]{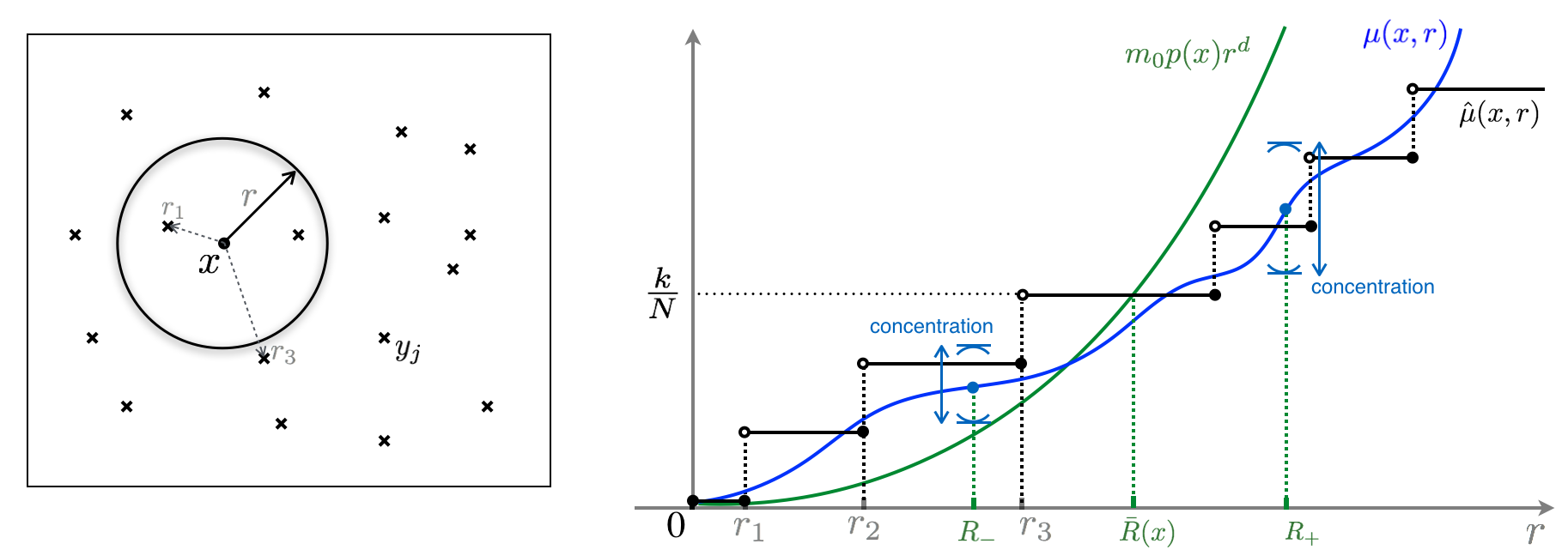}
}
\vspace{-5pt}
\caption{
\small
Given a dataset $Y$ and a fixed $x $,
plots of $\hat{\mu}(x,r)$, $\mu(x,r)$ and $m_0[h]p(x) r^d$ as  functions of $r$.
The values of $\hat{R}(x)$, $\bar{R}(x)$ and $R_{\pm}$ are marked. 
These quantities are used in the proof of Proposition \ref{prop:hatrho-one-point}.
}
\label{fig:diag-hatmu}
\end{figure}

We first address the lower bound in \eqref{Proposition2.2 Proof first bound}. By definition, $\hat{\mu}(x,r)$ is monotonically increasing on $(0, \infty)$.
 We claim that 
\begin{equation}\label{eq:lower-bound-relation-1}
\Pr [ \hat{R}(x) < R_- ]
\le 
\Pr \left[ \hat{\mu}( x, R_- ) \ge \frac{k}{N}  \right].
\end{equation}
Because 
$\hat{R}(x) = \inf \{ r > 0,  \hat{\mu}(x,r) \ge \frac{k}{N} \}$, if $\hat{R}(x) < R_- $, 
there is some $r'$, $ \hat{R}(x) < r'  < R_-$ such that $\hat{\mu}(x, r') \ge \frac{k}{N}$,
and by monotonicity $\hat{\mu}(x, R_-) \ge \hat{\mu}(x, r') \ge \frac{k}{N}$.

To bound the probability $\Pr \left[ \hat{\mu}(x,  R_- ) \ge \frac{k}{N}  \right]$,
we use that the expectation $\mu(x,R_-)$ would be smaller than $\frac{k}{N}$ under some conditions for $\delta_r$
defined in \eqref{eq:def-deltar}.

Note that by definition \eqref{eq:def-barR},
$\bar{R}(x) = \Theta^{[p]} ( (\frac{k}{N})^{1/d} ) =o^{[p]}(1)$,
and the implied constant is uniform for all $x$ by the uniform boundedness of $\bar{\rho}$.
Also, we have that
$\delta_r = o^{[p]}(1)$ 
under the asymptotic condition on $k$.
As a result, we have that
\begin{equation}\label{eq:Rminus(x)-asymp-uniform-x}
R_- 
= \Theta^{[p]} (
\bar{R}(x) )
= \Theta^{[p]} \left(
 ( \frac{k}{N} )^{1/d} 
 \right) 
= o^{[p]}(1).
\end{equation}
Then, \rev{ Lemma \ref{lemma:G-expansion-h-indicator}}
 gives that when $N$ is sufficiently large and then $R_-$ is small,   
\begin{align}
\mu( x,R_- ) 
&=  m_0[h] p(x) R_- ^d + O^{[p]} ( R_-^{d+2})  \nonumber \\
& = m _0[h] p(x) \bar{R}^d  (1- \delta_r)^d + O^{[p]}( \bar{R}^{d+2})   \nonumber \\
& = m _0[h] p(x) \bar{R}^d  \left( 1- d \delta_r + O^{\rev{[1]}}(\delta_r^2) + O^{[p]}( \bar{R}^{2})  \right)  \nonumber  \\
& \le 
\frac{k}{N} (1 - 0.9 d  \delta_r )
+ O^{[p]}(\bar{R}^{d+2})
=: \frac{k}{N} - \delta_{\mu_-}\,,
\rev{~~~ \text{(by that $m _0[h] p(x) \bar{R}^d = \frac{k}{N}$)}}
\label{eq:def-delta-mu-minus}
\end{align}
where
\rev{
the inequality in the last row is obtained by that $\delta_r = o^{[p]}(1)$,
and the large-$N$ threshold here only depends on $p$. 
}
Note that the implied constant of $O^{[p]}(\bar{R}^{d+2})$, denoted as  $c_p$, is uniform for all $x$.
Meanwhile, by \rev{uniform boundedness of $p$ from below}, we have
\[
\bar{R}(x) \le \max_{x \in {\calM}}\bar{\rho}(x) \left(\frac{1}{m_0[h]} \frac{k}{N}\right)^{1/d}
= \left( \frac{1}{ 
\rev{p_{min}}
m_0[h]} \right)^{1/d} \left(\frac{k}{N}\right)^{1/d}\,.
\]
Denote $c_{p,1}:=\left( \rev{p_{min}}  m_0[h] \right)^{-1/d}$,
and choose
\begin{equation}\label{eq:pick-t1}
t_1 := \frac{ c_p c_{p,1}^{d+2} }{0.8 d}
\rev{ = \Theta^{[p]}(1)},
\end{equation}
which 
is uniform for all $x$,
\rev{then}
$
t_1 \cdot 0.9 d   \left(\frac{k}{N}\right)^{1+2/d} > c_p   c_{p,1}^{d+2} \left(\frac{k}{N}\right)^{1+2/d}  \ge c_p \bar{R}^{d+2}$.
\rev{Thus,  when $N$ is sufficiently large and the threshold depends on $p$,
} 
we have
\begin{align}
\delta_{\mu_-} &= 0.9 d \frac{k}{N}
\left(  t_1 \left(\frac{k}{N}\right)^{2/d} + \frac{t_2}{d} \sqrt{\frac{s \log N}{k}} \right)
 + O^{[p]}(\bar{R}^{d+2}) \nonumber \\
& >  t_2 \cdot 0.9  \frac{k}{N}\sqrt{\frac{s \log N}{k}} 
= t_2 \cdot 0.9  \left( \frac{k}{N} \right)^{1/2} \sqrt{\frac{s \log N}{N}} 
=: \tilde{s}.
\label{eq:def-tilde-s}
\end{align}
To use the concentration of $\hat{\mu}( R_- )$ at $\mu(R_-)$,
we compute the boundedness and variance of $H_j(R_-)$.
Because $0 \le h \le 1$, so is $H_j$, and then $|H_j| \le L _H =1$. The variance
\[
\text{Var}(H_j)
\le \E H_j^2 
= \int_{\calM} h^2 \left(  \frac{ \| x- y\|^2 }{ R_-^2}\right)  p(y) dV(y)
= \mu(R_-),
\]
because the kernel function $h: \R \to \R$ satisfies $h^2 = h$. Thus, 
\rev{by that  $\mu(R_-)=  m_0[h] p(x) R_- ^d + O^{[p]} ( R_-^{d+2})$ with \eqref{eq:Rminus(x)-asymp-uniform-x},
and that $ \delta_r = o^{[p]}(1)$},
when $N$ is sufficiently large,
\[
\text{Var}(H_j) \le 1.1  m_0[h] p(x) R_-^d
\le 1.5 m_0[h] p(x) \bar{R}^d 
= 1.5 \frac{k}{N}
=: \bar{\nu}_H,
\]
\rev{and the two inequalities hold when $N$ exceeds a threshold depending on $p$ only.}
By the classical Bernstein {inequality}, 
as long as $\tilde{s} L_H < 3 \bar{\nu}_H$, 
then 
\[
\Pr [  \hat{\mu}( R_-)-  \mu( R_-) > \tilde{s}] < e^{- \frac{1}{4} \tilde{s}^2 \frac{N} {\bar{\nu}_H}}.
\]
To verify that $\tilde{s} L_H < 3 \bar{\nu}_H$:
note that it is equivalent to that $ t_2 \cdot 0.9   < 3 \cdot 1.5 ( \frac{k}{s \log N} )^{1/2}$,
and since  we have assumed $k = \Omega(\log N)$, 
if we have $t_2 = \Theta( 1)$,
then it holds when $N$ is sufficiently large \rev{where the threshold depends on $s$.}
This is fulfilled by
setting $t_2$ being an absolute constant such that
\begin{equation}\label{eq:pick-t2}
\frac{ ( t_2 0.9 )^2  }{4 \cdot 1.5} = 1,
\rev{ \quad  0 < t_2 < 3}.
\end{equation}
Thus, together with \eqref{eq:def-delta-mu-minus} and \eqref{eq:def-tilde-s}, 
{we have}
\[
\mu(R_-) 
\le \frac{k}{N} - {\delta_{\mu_-}} 
< \frac{k}{N} - \tilde{s}\,.
\]
As a result, \eqref{eq:lower-bound-relation-1} continues as
\begin{equation*} 
\Pr [ \hat{R}(x) < R_- ]
\le
\Pr \left[ \hat{\mu}( R_-) \ge \frac{k}{N} \right]
\le 
\Pr \left[ \hat{\mu}( R_-) > \mu(R_- ) + \tilde{s}  \right]
< e^{- \frac{1}{4} \tilde{s}^2 \frac{N} {\bar{\nu}_H}} = N^{-s},
\end{equation*}
which proves that { w.p. higher than} $1 - N^{-s}$,
the lower bound $ \hat{R}(x) \ge R_-  $ holds. 
We call the event $ [ \hat{R}(x) \ge R_- ]$ the good event $E_1$.
\rev{All the large-$N$ thresholds depend on $(p,s)$ and are uniform for all $x$}.

The upper bound is proved in a similar way. Specifically,
\begin{align*}
 \mu(R_+)
 & = \frac{k}{N} (1 + \delta r)^d + O^{[p]}( R_+^{d+2})
 \ge \frac{k}{N} (1 + 0.9 d \delta r ) + O^{[p]}( \bar{R}^{d+2}) \\
 & = \frac{k}{N}  + 0.9 d \frac{k}{N} 
 \left( t_1 (\frac{k}{N})^{2/d} + \frac{t_2}{d} \sqrt{\frac{s \log N}{k}} \right) + O^{[p]}( \bar{R}^{d+2}),
\end{align*}
and {the implied constant in} $O^{[p]}(  \bar{R}^{d+2} )$, 
{$c_{p}$,} is same as {the} above by Lemma \ref{lemma:G-expansion-h-indicator}.
Then, again by {the} uniform upper bound of $\bar{R}(x)$ by $c_{p,1} (\frac{k}{N})^{1/d}$,
{by setting} $t_1$ to { be that} in \eqref{eq:pick-t1}, {we have} 
\[
\mu(R_+) > \frac{k}{N} + t_2 0.9  \left(\frac{k}{N}\right)^{1/2} \sqrt{ \frac{s \log N}{N}} = \frac{k}{N} + \tilde{s}.
\]
Same as before, $H_j(R_+) $ is bounded by 1 and for {a sufficiently} large $N$,
\[
\text{Var}(H_j) \le \E H_j^2
= \mu( R_+) 
\le  1.5 \frac{k}{N} = \bar{\nu}_H\,.
\]
By letting $t_2$ as in \eqref{eq:pick-t2}, we have
\begin{equation*} 
\Pr[ \hat{R}(x) > R_+ ]
\le 
\Pr \left[ \hat{\mu}(R_+) < \frac{k}{N} \right]
\le 
\Pr \left[ \hat{\mu}(R_+) <  \mu(R_+) - \tilde{s} \right]
< e^{-\frac{1}{4} \tilde{s}^2 \frac{N}{ \bar{\nu}_H}} = N^{-s}.
\end{equation*}
This proves that w.p. higher than $1- N^{-s}$, the upper bound $\hat{R}(x) \le R_+$ holds.
We call the event  $[ \hat{R}(x) \le R_+ ]$ the good event $E_2$.

Putting {the above} together, under {the event $E_1\cap E_2$, which happens w.p. greater than} $1- 2N^{-s}$,
\[
\frac{|\hat{R}(x) - \bar{R}(x)|}{ \bar{R}(x) } \le {\delta_r} =
 \frac{ c_p c_{p,1}^{d+2} }{0.8 d} \left(\frac{k}{N}\right)^{2/d} 
 + \frac{\frac{2 \sqrt{ 1.5 } }{0.9}}{d} \sqrt{\frac{s \log N}{k}},
\]
which proves the claim of the proposition.
\end{proof}
%
%

%
%
\begin{proof}[Proof of Theorem \ref{thm:hatrho}]
We restrict to when $Y$ has distinct points,
which, under Assumption \ref{assump:M-p}, holds w.p. 1,
and then Lemma \ref{lemma:knn-rx} holds.

We cover ${\calM}$ using $r$-Euclidean balls, 
{where $r>0$ is a constant of order $(k/N)^{3/d}$ with the implied constant to be determined.} 
Suppose $N$ { is} large enough such that $r < \delta_0$ in Lemma \ref{lemma:M-delta0},
then by Lemma \ref{lemma:covering-2}, {we can find an} $r$-net 
$F :=\{ x_1, \cdots, x_n \}$
whose cardinal number is $n$, $n     { \le }  V({\calM}) r^{-d}$.
We ask for the bound in Proposition \ref{prop:hatrho-one-point} to hold at each $x_i$,
\rev{where $s > 0$ will be chosen later as an $\Theta^{[1]}(1)$ constant.}
Then, \rev{when $N$ exceeds a threshold depending on $p$ and uniform for all $x_i$,}
by a union bound, \rev{under a good event $E_{\hat{\rho},net}$ which happens} w.p. higher than $1- 2n N^{-s}$, 
we have
\begin{equation}\label{eq:def-epsilon-covering-bound}
\left| \frac{\hat{\rho}(x_i) }{\bar{\rho}(x_i) } - 1  \right| 
\le t_1 \left(\frac{k}{N}\right)^{2/d} +  \frac{t_2}{d} \sqrt{ \frac{s \log N }{k}} 
:= \varepsilon, 
\quad
\text{for all $i=1,\cdots, n$,}
\end{equation}
where $t_1 = \Theta^{[p]}( 1)$ and $t_2= \Theta^{[1]}( 1)$ are \rev{defined as} in the proof of Proposition \ref{prop:hatrho-one-point}.
Under the asymptotic condition on $k$,
$\varepsilon = o^{[p]}(1)$ as $N \to \infty$.

We {now} consider $\hat{\rho}/\bar{\rho}$ on each $\bar{B}_{r}(x_i) \cap {\calM}$.
Because $\bar{\rho} = p^{-1/d}$ is $C^\infty$ on ${\calM}$,
$\sup_{x \in {\calM}}| \nabla_{\calM} \bar{\rho} (x) | \le L_p$.
Then,  for each $x_i$,   by \eqref{eq:metric-geo-2}, 
\[
| \bar{\rho}(x) - \bar{\rho}(x_i) | \le L_p  d_{\calM}(x,x_i)
\le
1.1 L_p \| x - x_i \| 
\le 1.1 L_p r,
\quad 
\forall x \in \bar{B}_{r}(x_i) \cap {\calM}.
\]
Meanwhile, 
Lemma \ref{lemma:knn-rx} gives that 
\[
\text{Lip}_{\R^D}(\hat{\rho}) =  \left( \frac{1}{m_0[h]}\frac{k}{N} \right)^{-1/d}
\text{Lip}_{\R^D}( \hat{R})
\le  \left( \frac{1}{m_0[h]}\frac{k}{N} \right)^{-1/d},
\]
{so we have}
\[
| \hat{\rho}(x) - \hat{\rho}(x_i) | \le  \left( \frac{1}{m_0[h]}\frac{k}{N} \right)^{-1/d} r,
\quad
\text{$\forall x \in \bar{B}_{r}(x_i) \cap {\calM}$.}
\]
Together, we have that \rev{$\forall x \in \bar{B}_{r}(x_i) \cap {\calM}$,}
\begin{align}
\left| \frac{\hat{\rho}(x)}{\bar{\rho}(x)} - \frac{\hat{\rho}(x_i)}{\bar{\rho}(x_i)} \right|
&\le 
\frac{1}{\bar{\rho}(x)} \left| (\hat{\rho}(x)) - \hat{\rho}(x_i)) 
- \frac{\hat{\rho}(x_i)}{\bar{\rho}(x_i)} ( \bar{\rho}(x)-\bar{\rho}(x_i) ) \right| \nonumber \\
& \le 
{ \frac{1}{{\rho}_{min}}}
\left| 
\left( \frac{1}{m_0[h]}\frac{k}{N} \right)^{-1/d} 
+ (1+\varepsilon) \cdot 1.1 L_p 
\right| r 
=  \Theta^{[p]}\left( \left({k}/{N}\right)^{-1/d} \right) r\,.
\label{eq:bound-diff-rhoratio-from-xi}
\end{align}
Thus, \rev{one can choose $r>0$ to be $\Theta^{[p]}( (k/N)^{3/d} )$
so as to make \eqref{eq:bound-diff-rhoratio-from-xi} bounded by $ t_1  (\frac{k}{N})^{2/d}$
when $N$ is sufficiently large, where the threshold of $N$ depends on $p$ only.
This gives that 
$\left| \frac{\hat{\rho}(x)}{\bar{\rho}(x)} - \frac{\hat{\rho}(x_i)}{\bar{\rho}(x_i)} \right| \le  t_1  (\frac{k}{N})^{2/d}$.
Meanwhile, we already have \eqref{eq:def-epsilon-covering-bound} under $E_{\hat{\rho},net}$, and putting together,}
\[
\left| \frac{\hat{\rho}(x)}{\bar{\rho}(x)} -1 \right|
\le 
\left| \frac{\hat{\rho}(x)}{\bar{\rho}(x)} - \frac{\hat{\rho}(x_i)}{\bar{\rho}(x_i)} \right|
+ \left| \frac{\hat{\rho}(x_i)}{\bar{\rho}(x_i)} -1 \right|
\le t_1 \left(\frac{k}{N}\right)^{2/d} 
+ \varepsilon, 
\quad \forall x \in \bar{B}_{r}(x_i) \cap {\calM}.
\]
By that $\calM\subset \cup_i \bar B_r(x_i)$, \rev{the above bound holds for all $x \in \calM$.}
Recall the definition of $\varepsilon$ in \eqref{eq:def-epsilon-covering-bound},
we have that, \rev{under $E_{\hat{\rho},net}$,} 
\[
\sup_{x \in {\calM}}
\left| \frac{\hat{\rho}(x)}{\bar{\rho}(x)} -1 \right|
\le 2 t_1 \left(\frac{k}{N}\right)^{2/d} +  \frac{t_2}{d} \sqrt{ \frac{s \log N }{k}}.
\]
Finally, \rev{to show the high probability of $E_{\hat{\rho},net}$,} 
by that $n \le V({\calM}) r^{-d} $,
\begin{align*}
 2n N^{-s}
&\le 2V({\calM}) r^{-d} N^{-s}
\le c_p \left(\frac{k}{N}\right)^{-3} N^{-s} 
\quad \text{(constant $c_p$ depending on $p$)}
\\
& 
\le N^{-s+3}.
\quad \text{(with large $N$, because $k = \Omega(1)$)}
\end{align*}
so by setting $s = 13$, 
\rev{we have that $E_{\hat{\rho},net}$  happens w.p. higher than $>1- N^{-s+3} = 1-N^{-10}$}.
\end{proof}

\subsection{Proof of Proposition \ref{prop:hatE(f,f)}}

%
%
%
\begin{proof}[Proof of Proposition \ref{prop:hatE(f,f)}]
To simplify the notation, when there is no danger of confusion, we omit the dependence of $m_l[h]$ on $h$ and use the notation $m_l$, where $l=0,2$.

Under the condition that
\begin{equation}\label{eq:hatrho-varepsilon}
\sup_{x \in {\calM}}\frac{ | \hat{\rho}(x) - \bar{\rho}(x) | }{ | \bar{\rho}(x)|} 
< \varepsilon_\rho < 0.1,
\end{equation}
we have that
\begin{equation}\label{eq:bound-hatrho-0.1}
0.9 \bar{\rho}(x) < \hat{\rho}(x) < 1.1 \bar{\rho}(x),
\quad \forall x \in {\calM}.
\end{equation}
Recall that
\begin{align*}
&
{\calE}^{(\alpha)}(f,f)
 = \frac{\epsilon^{-\frac{d}{2}-1}}{ m_2}  
\int_{\calM} \int_{\calM}  ( f(x) - f(y) )^2 
k_0 \left(  \frac{\| x - y \|^2}{ \epsilon \hat{\rho}(x)  \hat{\rho}( y ) } \right)
\frac{p(x) p(y) }{ \hat{\rho}(x)^\alpha \hat{\rho}(y )^\alpha}
 dV(x) dV(y)
 =: \textcircled{1},
\end{align*}
and we consider the counterpart of $\textcircled{1}$ where $\hat{\rho}(x)$ is replaced with $\bar{\rho}(x)$, namely,
\[
\textcircled{2}
: = \frac{\epsilon^{-\frac{d}{2}-1}}{ m_2  }  
\int_{\calM} \int_{\calM}  ( f(x) - f(y) )^2 
k_0 \left(  \frac{\| x - y \|^2}{ \epsilon \bar{\rho}(x)  \hat{\rho}( y ) } \right)
\frac{p(x) p(y) }{ \bar{\rho}(x)^\alpha \hat{\rho}(y )^\alpha}
 dV(x) dV(y).
\]
With the operator $G^{(\rho)}_\epsilon$ defined as in \eqref{eq:def-G-R-rho-epsilon},
writing the integration over $dV(x)$ via $G^{(\bar{\rho})}_\epsilon$,
\begin{align}
\textcircled{2}
& = \frac{1}{ \epsilon m_2  }  
\left( 
\int_{\calM} (p \hat{\rho}^{d/2-\alpha}  )(y)
G_{\epsilon \hat{\rho}(y) }^{(\bar{\rho})} \frac{f^2 p}{ \bar{\rho}^{\alpha} }(y) dV(y)
- 2 \int_{\calM}  (f p \hat{\rho}^{d/2-\alpha}  )(y)
G_{\epsilon \hat{\rho}(y) }^{(\bar{\rho})}  \frac{f p}{ \bar{\rho}^{\alpha} } (y)dV(y) \right.
 \nonumber \\
& ~~~~~~~~~~
 \left.
+ \int_{\calM}  (f^2 p \hat{\rho}^{d/2-\alpha}  )(y)
G_{\epsilon \hat{\rho}(y) }^{(\bar{\rho})}  \frac{ p}{ \bar{\rho}^{\alpha} } (y)dV(y)
 \right).
\label{eq:circle2-2}
\end{align}
Recall that
$\bar{\rho} = p^{-1/d}$ is in $C^\infty({\calM})$ and uniformly bounded from below and above.
By Lemma \ref{lemma:right-operator-3},
 \begin{align*}
G_{\epsilon \hat{\rho}(y) }^{(\bar{\rho})}  \frac{f^2 p}{ \bar{\rho}^{\alpha} } 
&  = m_0 f^2 p \bar{\rho}^{\frac{d}{2}-\alpha} 
+ \epsilon \hat{\rho} \frac{m_2}{2} (   \omega  f^2 p \bar{\rho}^{1+\frac{d}{2}-\alpha}  + \Delta ( f^2 p \bar{\rho}^{1+\frac{d}{2}-\alpha}  ) ) 
+ \hat{\rho}^2 r_{1}^{(2)},  \\
G_{\epsilon \hat{\rho}(y) }^{(\bar{\rho})}  \frac{f p}{ \bar{\rho}^{\alpha} } 
&  = m_0 f p \bar{\rho}^{\frac{d}{2}-\alpha} 
+ \epsilon \hat{\rho} \frac{m_2}{2} (   \omega  f p \bar{\rho}^{1+\frac{d}{2}-\alpha}  + \Delta ( f p \bar{\rho}^{1+\frac{d}{2}-\alpha}  ) ) 
+ \hat{\rho}^2 r_{2}^{(2)}, 
\\
G_{\epsilon \hat{\rho}(y) }^{(\bar{\rho})}  \frac{ p}{ \bar{\rho}^{\alpha} } 
&  = m_0  p \bar{\rho}^{\frac{d}{2}-\alpha} 
+ \epsilon \hat{\rho} \frac{m_2}{2} (   \omega  p \bar{\rho}^{1+\frac{d}{2}-\alpha}  + \Delta ( p \bar{\rho}^{1+\frac{d}{2}-\alpha}  ) ) 
+ \hat{\rho}^2 r_{3}^{(2)}\,, 
 \end{align*}
where $\| r_1^{(2)} \|_\infty = O^{[f,p]}(\epsilon^2)$, $\| r_2^{(2)} \|_\infty = O^{[f,p]}(\epsilon^2)$ and $\| r_3^{(2)} \|_\infty = O^{[p]}(\epsilon^2)$ and we omit the evaluation of all functions at $y$ in the notation.
Then, \eqref{eq:circle2-2} becomes 
\begin{align*}
 \textcircled{2} 
 & = \frac{1}{ \epsilon m_2 }  
\left( 
\int_{\calM} (p \hat{\rho}^{d/2-\alpha}  )
\left\{
 m_0 f^2 p \bar{\rho}^{\frac{d}{2}-\alpha} 
+ \epsilon \hat{\rho} \frac{m_2}{2} (   \omega  f^2 p \bar{\rho}^{1+\frac{d}{2}-\alpha}  + \Delta ( f^2 p \bar{\rho}^{1+\frac{d}{2}-\alpha}  ) ) 
+ \hat{\rho}^2 r_{1}^{(2)} 
\right\}
 \right.
\\
&  ~~~~~
- 2 \int_{\calM} 
(f p \hat{\rho}^{d/2-\alpha}  )
\left\{
m_0 f p \bar{\rho}^{\frac{d}{2}-\alpha} 
+ \epsilon \hat{\rho} \frac{m_2}{2} (   \omega  f p \bar{\rho}^{1+\frac{d}{2}-\alpha}  + \Delta ( f p \bar{\rho}^{1+\frac{d}{2}-\alpha}  ) ) 
+ \hat{\rho}^2 r_{2}^{(2)}  
\right\}
\\
&  ~~~~~
\left.
+ \int_{\calM} 
(f^2 p \hat{\rho}^{d/2-\alpha}  )
\left\{
m_0  p \bar{\rho}^{\frac{d}{2}-\alpha} 
+ \epsilon \hat{\rho} \frac{m_2}{2} (   \omega   p \bar{\rho}^{1+\frac{d}{2}-\alpha}  + \Delta ( p \bar{\rho}^{1+\frac{d}{2}-\alpha}  ) ) 
+ \hat{\rho}^2 r_{3}^{(2)}
\right \}
\right)
 \\
& = 
\frac{1}{2} \int_{\calM} 
p \hat{\rho}^{\frac{d}{2}-\alpha+1} 
\left\{
  \Delta ( f^2 p \bar{\rho}^{1+\frac{d}{2}-\alpha}  ) 
-  2 f   \Delta ( f p \bar{\rho}^{1+\frac{d}{2}-\alpha}  ) 
+ f^2 \Delta (  p \bar{\rho}^{1+\frac{d}{2}-\alpha}  )
\right\} \\
&  ~~~~~
+ \frac{1}{ \epsilon m_2  }  \int_{\calM}  
 p \hat{\rho}^{ \frac{d}{2} -\alpha  +2 }  
 ( r_{1}^{(2)} -  2 f r_{2}^{(2)} + f^2  r_{3}^{(2)}) \\
& =:  \textcircled{2}_1  +  \textcircled{2}_2\,,
\end{align*}
 where again, we omit the evaluation of all functions at $y$ and the integration over \rev{$dV(y)$} in the notation, \rev{and same in below.}

We first consider $\textcircled{2}_1$.
Note that, by defining $g: = p \bar{\rho}^{1+d/2-\alpha}$, 
the bracket inside the integrand becomes
\begin{align*}
 \left\{  \cdots \right\}
=  \Delta( f^2 g) - 2f \Delta( f g) + f^2 \Delta g 
 = 2 g |\nabla f|^2 
 = 2 p \bar{\rho}^{1+d/2-\alpha} |\nabla f|^2\,.
\end{align*}
We define $\textcircled{2}_1'$ by substituting $\hat{\rho}$ with $\bar{\rho}$ in $\textcircled{2}_1$, namely, 
\[
\textcircled{2}_1' := 
\frac{1}{2} \int_{\calM} 
p \bar{\rho}^{\frac{d}{2}-\alpha+1} 
\left\{
\cdots \right\},
\quad
\textcircled{2}_1  - \textcircled{2}_1'
=  \frac{1}{2} \int_{\calM} 
p  (  \hat{\rho}^{\frac{d}{2}-\alpha+1} -  \bar{\rho}^{\frac{d}{2}-\alpha+1}  )
\left\{
\cdots \right\},
\]
Inserting the expression of $\{\cdots \}$,
by the definition of $\bar{\rho} = p^{-1/d}$ and $p_\alpha$, we have
\[
\textcircled{2}_1' 
= \int_{\calM} 
p^2  \bar{\rho}^{2+ d -2 \alpha} |\nabla f|^2
  =
  \int_{\calM} 
p_\alpha
|\nabla f|^2  
= { \calE}_{p_\alpha}(f,f),
\]
and also 
\[
\textcircled{2}_1  - \textcircled{2}_1'
=  
\int_{\calM} 
p^2    \bar{\rho}^{1+d/2-\alpha} |\nabla f|^2 (  \hat{\rho}^{\frac{d}{2}-\alpha+1} -  \bar{\rho}^{\frac{d}{2}-\alpha+1}  ).
\]
 Since $p$ and $\bar{\rho}$ are positive, we have
\[
| \textcircled{2}_1  - \textcircled{2}_1' |
\le   
\int_{\calM} 
p^2    \bar{\rho}^{1+d/2-\alpha} |\nabla f|^2  |  \hat{\rho}^{\frac{d}{2}-\alpha+1} -  \bar{\rho}^{\frac{d}{2}-\alpha+1}  |.
\]
Let $\gamma := d/2 - \alpha + 1 \in \R$. 
By the Mean Value Theorem and \eqref{eq:bound-hatrho-0.1}, 
for any $y$, there exists $\xi > 0$ between $\hat{\rho}(y)$ and $\bar{\rho}(y)$, 
$\xi^{\gamma-1} \le  \max \{ 1.1 ^{\gamma-1},   0.9^{\gamma-1} \} \bar{\rho}(y)^{\gamma-1}$
such that
\begin{equation}\label{eq:bound-anypower-hatrho-barho}
| \hat{\rho}(y)^\gamma -  \bar{\rho}(y)^\gamma  |
= \gamma \xi^{\gamma-1} | \hat{\rho}(y)-  \bar{\rho}(y) |
\le \gamma  \max \{ 1.1 ^{\gamma-1},   0.9^{\gamma-1} \} \bar{\rho}(y)^{\gamma-1} | \hat{\rho}(y)-  \bar{\rho}(y) |
\le c_\gamma \bar{\rho}(y)^{\gamma} \varepsilon_\rho,
\end{equation}
where the last inequality comes from \eqref{eq:hatrho-varepsilon} and $c_\gamma:= \gamma  \max \{ 1.1 ^{\gamma-1},   0.9^{\gamma-1} \} $ is a constant determined by $\alpha$ and $d$. 
Then
\[
| \textcircled{2}_1  - \textcircled{2}_1' |
\le   
c_\gamma  \varepsilon_\rho
 \int_{\calM} 
p^2    \bar{\rho}^{1+d/2-\alpha} |\nabla f|^2 
 \bar{\rho}^{\gamma}
 =  c_\gamma  \varepsilon_\rho \cdot \textcircled{2}_1' ,
\]
which gives that 
\begin{equation}\label{eq:bound-circle2-1-relative-error}
 \textcircled{2}_1  =  \textcircled{2}_1' (1 + O(c_\gamma \varepsilon_\rho ) ) 
 = { \calE}_{p_\alpha}(f,f) (1 +   O^{[\alpha]}(  \varepsilon_\rho ) )\,,
\end{equation}

{
Next we bound  $\textcircled{2}_2$. By definition, 
\[
| \textcircled{2}_2 |
\le \frac{1}{ \epsilon m_2  }  
\int_{\calM}  
 p \hat{\rho}^{ \frac{d}{2} -\alpha  +2 }  
 ( | r_{1}^{(2)} | + 2 |f| |r_{2}^{(2)}| +  |f|^2 |r_{3}^{(2)}| )\,.
 \]
Again,
by \eqref{eq:bound-hatrho-0.1}, $\hat{\rho}(y)^{{d}/{2} -\alpha  +2} 
\le \max\{ 1.1^{{d}/{2} -\alpha  +2}, 0.9^{{d}/{2} -\alpha  +2} \} \bar{\rho}(y)^{{d}/{2} -\alpha  +2}:= c_\alpha'  \bar{\rho}(y)^{{d}/{2} -\alpha  +2}$.
Then
\[
| \textcircled{2}_2 |
 \le O^{[f,p]}(\epsilon)
 \int_{\calM}   c_\alpha'  p \bar{\rho}^{ \frac{d}{2} -\alpha  +2 } 
 = O^{[f,p]}(\epsilon).
 \]
 Together with \eqref{eq:bound-circle2-1-relative-error}, we have that 
 \begin{equation}\label{eq:bound-circle2-relative-error}
  \textcircled{2}  
 = { \calE}_{p_\alpha}(f,f) (1 +   O^{[\alpha]}(  \varepsilon_\rho ) )
 + O^{[f,p]}(\epsilon).
 \end{equation}
}

It remains to bound $|\textcircled{2} - \textcircled{1}|$ to prove the same bound for $ \textcircled{1}$.
Define
\[
\textcircled{3}
: = 
\frac{\epsilon^{-\frac{d}{2}-1}}{ m_2  }  
\int_{\calM} \int_{\calM} ( f(x) - f(y) )^2
k_0 \left(  \frac{\| x - y \|^2}{ \epsilon \bar{\rho}(x)  \hat{\rho}( y ) } \right) 
 \frac{  p(x) p(y) }{\hat{\rho}(x)^\alpha \hat{\rho}(y )^\alpha}  
 dV(x) dV(y)\,.
\]
Then,
\begin{equation}\label{eq:circle3-circle2}
 \textcircled{3} - \textcircled{2} 
= \frac{\epsilon^{-\frac{d}{2}-1}}{ m_2  }  
\int_{\calM} \int_{\calM} ( f(x) - f(y) )^2
 \left(   \frac{ \bar{\rho}(x)^\alpha }{\hat{\rho}(x)^\alpha } - 1  \right)
k_0 \left(  \frac{\| x - y \|^2}{ \epsilon \bar{\rho}(x)  \hat{\rho}( y ) } \right) 
 \frac{  p(x) p(y) }{ \bar{\rho}(x)^\alpha \hat{\rho}(y )^\alpha} 
 dV(x) dV(y).
\end{equation}
{
By  \eqref{eq:hatrho-varepsilon} and \eqref{eq:bound-hatrho-0.1}, we have that 
\begin{equation}\label{eq:bound-uniform-rel-error-Oalpha}
\sup_{x \in {\calM}} \left| \frac{ \bar{\rho}(x)^\alpha }{\hat{\rho}(x)^\alpha } - 1 \right| = O^{[\alpha]}(\varepsilon_\rho),
\quad
\sup_{x \in {\calM}} \left| \frac{ \bar{\rho}(x)}{\hat{\rho}(x) } - 1 \right| = O^{[1]}(\varepsilon_\rho)\,.
\end{equation}
Note that by definition, $\textcircled{2} \ge 0$. Therefore, by that $k_0 \ge 0$  and $p$, $\bar{\rho}$, $\hat{\rho} > 0$,
\begin{align*}
| \textcircled{3} - \textcircled{2} |
& \le 
O^{[\alpha]}(\varepsilon_\rho) \cdot 
 \frac{\epsilon^{-\frac{d}{2}-1}}{ m_2  }  
\int_{\calM} \int_{\calM} ( f(x) - f(y) )^2
k_0 \left(  \frac{\| x - y \|^2}{ \epsilon \bar{\rho}(x)  \hat{\rho}( y ) } \right) 
 \frac{  p(x) p(y) }{ \bar{\rho}(x)^\alpha \hat{\rho}(y )^\alpha} 
 dV(x) dV(y) \\
& = O^{[\alpha]}(\varepsilon_\rho) \textcircled{2}.
\end{align*}
Together with \eqref{eq:bound-circle2-relative-error}, this gives that 
\begin{equation}\label{eq:bound-circle3-relative-error}
\textcircled{3} = \textcircled{2} (1 +  O^{[\alpha]}(\varepsilon_\rho))
= 
 { \calE}_{p_\alpha}(f,f) (1 +   O^{[\alpha]}(  \varepsilon_\rho ) ) 
 +  O^{[f,p]}(\epsilon).
\end{equation}
}

Meanwhile,
\[
\textcircled{3} - \textcircled{1} 
= 
\frac{\epsilon^{-\frac{d}{2}-1}}{ m_2  }  
\int_{\calM} \int_{\calM} ( f(x) - f(y) )^2
\left( 
k_0 \left(  \frac{\| x - y \|^2}{ \epsilon \bar{\rho}(x)  \hat{\rho}( y ) } \right) -
k_0 \left(  \frac{\| x - y \|^2}{ \epsilon \hat{\rho}(x)  \hat{\rho}( y ) } \right) 
\right) 
 \frac{  p(x) p(y) }{\hat{\rho}(x)^\alpha \hat{\rho}(y )^\alpha}  
 dV(x) dV(y),
\]
and
\[
k_0 \left(  \frac{\| x - y \|^2}{ \epsilon \bar{\rho}(x)  \hat{\rho}( y ) } \right) -
k_0 \left(  \frac{\| x - y \|^2}{ \epsilon \hat{\rho}(x)  \hat{\rho}( y ) } \right) 
= k_0'(\xi) \frac{\| x - y \|^2}{ \epsilon \bar{\rho}(x)  \hat{\rho}( y ) }  \left( 1 - \frac{\bar{\rho}(x)  }{ \hat{\rho}(x) } \right),
\]
where $\xi$ is between $  \frac{\| x - y \|^2}{ \epsilon \bar{\rho}(x)  \hat{\rho}( y ) } $ and $\frac{\| x - y \|^2}{ \epsilon \hat{\rho}(x)  \hat{\rho}( y ) } $.
By \eqref{eq:bound-hatrho-0.1},
$\xi \ge \frac{\| x - y \|^2}{ \epsilon 1.1 \bar{\rho}(x)  \hat{\rho}( y ) }$, 
{ and then, }
\[
|k_0'(\xi) | \le a_1 e^{-a \xi } 
\le  a_1 e^{-a  \frac{\| x - y \|^2}{ \epsilon 1.1 \bar{\rho}(x)  \hat{\rho}( y ) } }.
\]
By \eqref{eq:bound-uniform-rel-error-Oalpha}, we have that
\begin{align}
| \textcircled{3} - \textcircled{1}  |
&\le 
\frac{\epsilon^{-\frac{d}{2}-1}}{ m_2  }  
\int_{\calM} \int_{\calM} ( f(x) - f(y) )^2 
|k_0'(\xi) | \frac{\| x - y \|^2}{ \epsilon \bar{\rho}(x)  \hat{\rho}( y ) }  \left| 1 - \frac{\bar{\rho}(x)  }{ \hat{\rho}(x) } \right|
 \frac{  p(x) p(y) }{\hat{\rho}(x)^\alpha \hat{\rho}(y )^\alpha}  
 dV(x) dV(y) \nonumber \\
& \le O(\varepsilon_\rho)
\frac{\epsilon^{-\frac{d}{2}-1}}{ m_2  }  
\int_{\calM} \int_{\calM} ( f(x) - f(y) )^2 
a_1 e^{- \frac{a}{1.1}  \frac{\| x - y \|^2}{ \epsilon  \bar{\rho}(x)  \hat{\rho}( y ) } }
\frac{\| x - y \|^2}{ \epsilon \bar{\rho}(x)  \hat{\rho}( y ) } 
 \frac{  p(x) p(y) }{\hat{\rho}(x)^\alpha \hat{\rho}(y )^\alpha}  
 dV(x) dV(y)\nonumber \\
& = O(\varepsilon_\rho)
\frac{\epsilon^{-\frac{d}{2}-1}}{ m_2  }  
\int_{\calM} \int_{\calM} ( f(x) - f(y) )^2 
k_1 \left( \frac{\| x - y \|^2}{ \epsilon \bar{\rho}(x)  \hat{\rho}( y ) }  \right)
 \frac{  p(x) p(y) }{\hat{\rho}(x)^\alpha \hat{\rho}(y )^\alpha}  
 dV(x) dV(y) \nonumber \\
& = O(\varepsilon_\rho) \frac{m_2[k_1]}{m_2[k_0]} \cdot \text{$\textcircled{3}'$}\,,
\label{eq:bound-circle3-circle1-withk1}
\end{align}
where 
$\textcircled{3}'$ is defined by replacing $k_0$ to be $k_1$ in $\textcircled{3}$, where
\[
k_1( r) := 
a_1 r e^{- \frac{a}{1.1} r }\,\, \text{ and }\,\, r \ge 0.
\]
Since $k_1$ satisfies Assumption \ref{assump:h-selftune}, 
our analysis of $\textcircled{2}$ and $\textcircled{3}$  with $k_1$ so far applies.
{
Thus, by \eqref{eq:bound-circle2-relative-error} and \eqref{eq:bound-circle3-relative-error}, 
we have
\[
\text{$\textcircled{3}'$} =  { \calE}_{p_\alpha}(f,f) (1 +   O^{[\alpha]}(  \varepsilon_\rho ) ) + O^{[f,p]}(\epsilon),
\]
and then
\[
| \textcircled{3} - \textcircled{1}  | = O(\varepsilon_\rho) 
( { \calE}_{p_\alpha}(f,f) (1 +   O^{[\alpha]}(  \varepsilon_\rho ) ) + O^{[f,p]}(\epsilon)) 
= 
 { \calE}_{p_\alpha}(f,f)  O^{[\alpha]}(  \varepsilon_\rho )  + O^{[f,p]}(\epsilon \varepsilon_\rho).
\]
Inserting \eqref{eq:bound-circle3-relative-error} gives that
\[
 \textcircled{1}  =
 \textcircled{3} + { \calE}_{p_\alpha}(f,f)  O^{[\alpha]}(  \varepsilon_\rho )  + O^{[f,p]}(\epsilon \varepsilon_\rho)
 = { \calE}_{p_\alpha}(f,f) (1 +   O^{[\alpha]}(  \varepsilon_\rho ) ) + O^{[f,p]}(\epsilon).
\]
This finishes the proof since ${\calE}^{(\alpha)}(f,f) =  \textcircled{1}$.
}
\end{proof}

\subsection{Proofs in Section \ref{subsec:dir-form-graph} (Theorems \ref{thm:limit-form} and \ref{thm:limit-form-fixeps})}

\begin{proof}[Proof of Theorem \ref{thm:limit-form}]
Suppose $f$ is not a constant function, because otherwise $E_N(f,f) = {\calE}_{p_\alpha}(f,f)= 0$ and the theorem holds.
By definition,
\begin{equation}\label{eq:EN-1}
 E_N(f,f) = 
\frac{1}{N^2} \sum_{i,j=1}^N \frac{\epsilon^{-\frac{d}{2}-1}}{m_2}  (f(x_i) - f(x_j))^2 W_{ij}^{(\alpha)} 
= \frac{1}{N^2} \sum_{i \neq j, \, i,j = 1}^N V_{ij},
\end{equation}
where
\[
V_{ij}:= \frac{1 }{\epsilon m_2}  (f(x_i) - f(x_j))^2  \hat{K} ( x_i, x_j).
\]
As \eqref{eq:EN-1} is a V-statistic, we study $\E  E_N(f,f) $ and its variation away from $\E  E_N(f,f) $ respectively.

$\bullet$ \underline{Calculation of  $\E  E_N(f,f) $.}
By definition,
\[
\E   E_N(f,f) =   \frac{N-1}{N} \E V_{1,2}
\,\,\mbox{ and }\,\,
\E V_{1,2} = {\calE}^{(\alpha)}(f,f).
\]
Applying Proposition \ref{prop:hatE(f,f)} gives 
\begin{equation}\label{eq:EV12-formula}
\E V_{1,2}  =
{\calE}_{p_\alpha}(f,f)(1 + O^{[\alpha]} (\varepsilon_\rho) )
 + O^{[f,p]}(\epsilon ).
\end{equation}

$\bullet$ \underline{Bound the deviation of $ E_N(f,f) - \E E_N(f,f)$.}
We use the decoupling trick to bound the deviation of 
a V-statistic by that of an independent sum over $\frac{N}{2}$ terms.
Specifically, define $\tilde{V}_{ij} = V_{ij} - \E V_{ij}$.
For any $ t > 0$, 
 the Markov inequality gives us
\begin{align}\label{Proof Theorem 3.4 Prob first}
& \Pr \left[ 
\frac{1}{N(N-1)} \sum_{i \neq j, i,j=1}^N \tilde{V}_{ij} > t 
\right]
\le e^{-st} \E \exp \left\{ s  \frac{1}{N(N-1)} \sum_{i \neq j, i,j=1}^N \tilde{V}_{ij}  \right\}  
\end{align}
where $s > 0$ will be determined later. By a direct expansion, and denote by ${\calS}_N$ the permutation group,
we have
\begin{align*}
&e^{-st} \E \exp \left\{ s  \frac{1}{N(N-1)} \sum_{i \neq j, i,j=1}^N \tilde{V}_{ij}  \right\}
= 
 e^{-st} \E \exp \left\{ s \frac{1}{N!} \sum_{ \sigma \in {\calS}_N}  \frac{1}{N(N-1)}  \sum_{i \neq j, i,j=1}^N \tilde{V}_{\sigma(i), \sigma(j)}  \right\} \\
=&
 e^{-st}  \E \exp \left\{ s \frac{1}{N!} \sum_{ \sigma \in {\calS}_N}  \frac{1}{N/2}  \sum_{l=1}^{N/2} \tilde{V}_{\sigma(2l-1), \sigma(2l)}  \right\} 
\le  e^{-st}  \E \frac{1}{N!} \sum_{ \sigma \in {\calS}_N}  \exp \left\{ s  \frac{1}{N/2}  \sum_{l=1}^{N/2} \tilde{V}_{\sigma(2l-1), \sigma(2l)}  \right\} \\
=& e^{-st}\E  \exp \left\{ s  \frac{1}{N/2}  \sum_{l=1}^{N/2} \tilde{V}_{2l-1, 2l}  \right\},
\end{align*}
where we apply the Jensen's inequality in the inequality.
Then, as in the derivation of the Classical Bernstein's inequality, 
one can bound the probability in \eqref{Proof Theorem 3.4 Prob first} by 
\begin{equation}\label{eq:bernstein-1}
\Pr [ \cdots ] \le \exp\left\{ - \frac{ \frac{N}{2} t^2}{2 \nu + \frac{2}{3} t L } \right\},
\quad 
\text{where}
\quad 
\nu := \E  \tilde{V}_{1,2}^2,
\quad |\tilde{V}_{1,2}| \le L.
\end{equation}
Below, we control $\nu$ and $L$.
We first show that we can make $L=\Theta^{[f,p]}( \epsilon^{-\frac{d}{2}} )$.
Recall that
\[
V_{1,2} = 
\frac{\epsilon^{-\frac{d}{2}-1}}{m_2}  
k_0 \left(  \frac{\| x_1 - x_2 \|^2}{ \epsilon \hat{\rho}(x_1)  \hat{\rho}(x_2) } \right)
\frac{(f(x_1) - f(x_2))^2  }{ \hat{\rho}(x_1)^\alpha \hat{\rho}(x_2)^\alpha}\,.
\]
By Assumption \ref{assump:h-diffusionmap}(C2')  for the kernel $k_0$,
\[
|V_{1,2}| 
\le
\frac{\epsilon^{-\frac{d}{2}-1}}{m_2}  
a_0 e^{ - a \left(  \frac{\| x_1 - x_2 \|^2}{ \epsilon \hat{\rho}(x_1)  \hat{\rho}(x_2) } \right) }
\frac{(f(x_1) - f(x_2))^2  }{ \hat{\rho}(x_1)^\alpha \hat{\rho}(x_2)^\alpha}.
\]
By the assumption and \eqref{eq:bound-hatrho-0.1}, 
$\hat{\rho}(x) \le 1.1 \bar{\rho}(x) < 1.1 \rho_{max} $ for any $x$. Then
when $\| x_1 - x_2 \| \ge \delta_\epsilon:= \sqrt{ \epsilon 1.1^2 
\rho_{max}^2
  \frac{5+d/2}{a}\log \frac{1}{\epsilon}}$, 
\[
e^{ - a \left(  \frac{\| x_1 - x_2 \|^2}{ \epsilon \hat{\rho}(x_1)  \hat{\rho}(x_2) } \right) }
\le e^{- (5+d/2) \log \frac{1}{\epsilon} } = \epsilon^{5+d/2}
\]
and then 
\[
|V_{1,2}| \le \frac{a_0 \epsilon^{4} }{m_2}   \frac{ ( 2\|f\|_\infty)^2}{ (0.9
\rho_{min}
)^{2\alpha}}
=O^{[f,p]}(\epsilon^{4}), 
\quad
\text{
when $\| x_1 - x_2 \| \ge \delta_\epsilon$.}
\]
Note that $\delta_\epsilon$ is of order $\sqrt{\epsilon\log(\epsilon^{-1})}$,
since $\epsilon =o(1)$, when $\epsilon $ is small enough
 such that $\delta_\epsilon < \delta_0$ in Lemma \ref{lemma:M-delta0},
\begin{equation}\label{eq:local-Lip-f}
| f(x_1)-f(x_2) | \le \| \nabla_{\calM} f \|_\infty 1.1 \|x_1 - x_2\| =: L_f \|x_1 - x_2\|,
\quad
\text{for all $x_2 \in B_{\delta_\epsilon}(x_1) \cap {\calM}$.}
\end{equation}
Then, when $\|x_1 - x_2 \| < \delta_\epsilon$,
\begin{align*}
|V_{1,2}| &\le \frac{\epsilon^{-\frac{d}{2}-1}}{m_2}  
a_0 e^{ - a \left(  \frac{\| x_1 - x_2 \|^2}{ \epsilon \hat{\rho}(x_1)  \hat{\rho}(x_2) } \right) } L_f^2
\frac{   \|x_1-x_2\|^2 }{ \hat{\rho}(x_1) \hat{\rho}(x_2) } \hat{\rho}(x_1)^{1-\alpha} \hat{\rho}(x_2)^{1-\alpha} \\
& \le  \epsilon^{-\frac{d}{2}} \frac{a_0 L_f^2}{m_2} a_1' 
\| \max\{ 0.9^{1-\alpha}, 1.1^{1-\alpha}\} \bar{\rho}^{1-\alpha}\|_\infty^2,
\end{align*}
where $a_1'$ equals an absolute constant times $\frac{a_0}{a}$.
Combining both cases, 
$|V_{1,2}|=O^{[f,p]}(  \epsilon^{-\frac{d}{2}})$, 
and  we denote
\[
|\tilde{V}_{1,2}| \le L =\Theta^{[f,p]}( \epsilon^{-\frac{d}{2}}).
\]

We now compute the variance $\nu$, and show that $\nu \le \epsilon^{-d/2} V_f$ where $V_f =\Theta^{[f,p]}( 1)$.
By definition,
\begin{align}
\E V_{1,2}^2
& = \int_{\calM} \int_{\calM}
 \frac{ \epsilon^{-d-2}}{m_2[k_0]^2}  
k_0^2 \left(  \frac{\| x - y \|^2}{ \epsilon \hat{\rho}(x)  \hat{\rho}( y) } \right)
\frac{(f(x) - f( y))^4  }{ \hat{\rho}(x)^{2\alpha} \hat{\rho}( y)^{2\alpha} }
p(x) p(y) dV(x)dV(y)  
= \frac{ \epsilon^{-d/2-1}}{ \frac{ m_2[k_0]^2}{m_2[k_0^2]}}  \cdot  \textcircled{4}\,,
\label{eq:EV12square-circle4} 
\end{align}
where 
\begin{align}
\textcircled{4}
& := \frac{ \epsilon^{-d/2-1}}{m_2[k_0^2]} 
 \int_{\calM} \int_{\calM}
 k_0^2 \left(  \frac{\| x - y \|^2}{ \epsilon \hat{\rho}(x)  \hat{\rho}( y) } \right)
\frac{(f(x) - f( y))^4  }{ \hat{\rho}(x)^{2\alpha} \hat{\rho}( y)^{2\alpha} }
p(x) p(y) dV(x)dV(y).
\nonumber 
\end{align}
Let $\delta_\epsilon$ be as above,
and we separate the integral within and outside $ \{ \| x - y\| < \delta_\epsilon \}$ and make $\textcircled{4} = \textcircled{4}_1 + \textcircled{4}_2$.
Specifically,
 we define
\begin{align*}
& \textcircled{4}_2:=
\frac{ \epsilon^{-d/2-1}}{m_2[k_0^2]} 
 \int_{\calM} \int_{\calM} 
 {\bf 1}_{\{ \| x-y\| \ge \delta_\epsilon \}}
 k_0^2 \left(  \frac{\| x - y \|^2}{ \epsilon \hat{\rho}(x)  \hat{\rho}( y) } \right)
\frac{(f(x) - f( y))^4  }{ \hat{\rho}(x)^{2\alpha} \hat{\rho}( y)^{2\alpha} }
p(x) p(y) dV(x)dV(y)\,.
\end{align*}
By a direct bound, we have
\begin{align*}
\textcircled{4}_2 & \le 
\frac{ \epsilon^{-d/2-1}}{m_2[k_0^2]} 
\int_{\calM} \int_{\calM} 
a_0^2 e^{ - 2a \left(  \frac{\| x - y \|^2}{ \epsilon \hat{\rho}(x)  \hat{\rho}( y) } \right) }
 {\bf 1}_{\{ \| x-y\| \ge \delta_\epsilon \}}
(f(x) - f( y))^4  \hat{\rho}(x)^{-2\alpha} \hat{\rho}( y)^{-2\alpha} 
p(x) p(y) dV(x)dV(y) \\
& \le 
\frac{ \epsilon^{-d/2-1}}{m_2[k_0^2]} 
a_0^2 \epsilon^{10+d} 
\int_{\calM} \int_{\calM} 
(f(x) - f( y))^4  \max\{ 0.9^{-2\alpha}, 1.1^{-2\alpha}\}^2 \bar{\rho}(x)^{-2\alpha} \bar{\rho}( y)^{-2\alpha} 
p(x) p(y) dV(x)dV(y) \\
& = O^{ [f, p]}( \epsilon^{\frac{d}{2}+9}).
\end{align*}
Define 
\[
\textcircled{4}_1:=
\frac{ \epsilon^{-d/2-1}}{m_2[k_0^2]} 
 \int_{\calM} \int_{\calM} 
 {\bf 1}_{\{ \| x-y\| < \delta_\epsilon \}}
 k_0^2 \left(  \frac{\| x - y \|^2}{ \epsilon \hat{\rho}(x)  \hat{\rho}( y) } \right)
\frac{(f(x) - f( y))^4  }{ \hat{\rho}(x)^{2\alpha} \hat{\rho}( y)^{2\alpha} }
p(x) p(y) dV(x)dV(y)\,.
\]
To control $\textcircled{4}_1$, which involves an integration over  the $\delta_\epsilon$ ball, 
note that for $y \in B_{\delta_\epsilon}(x) \cap {\calM}$,
by Lemma \ref{lemma:M-delta0},
\[
f(y) = f(x) + \nabla f(x) \cdot (y-x) +  O^{[f]} (\|x-y\|^2),
\]
and thus,
\[
(f(y) - f(x))^4
= (\nabla f(x) \cdot (y-x))^4 + O^{[f]}( \|x-y\|^5)
\le( |\nabla f(x)| \| y-x\|)^4 + O^{[f]}( \|x-y\|^5)\,.
\]
We then have
\begin{align*}
 \textcircled{4}_1&
 \le 
\frac{ \epsilon^{-d/2-1}}{m_2[k_0^2]} 
 \int_{\calM} \int_{\calM} 
 {\bf 1}_{\{ \| x-y\| < \delta_\epsilon \}}
 k_0^2 \left(  \frac{\| x - y \|^2}{ \epsilon \hat{\rho}(x)  \hat{\rho}( y) } \right)
\frac{  |\nabla f(x)|^4 \|x-y\|^4  + O^{[f]}( \|x-y\|^5) }
{  \hat{\rho}(x)^{2\alpha} \hat{\rho}( y)^{2\alpha} }
p(x) p(y) dV(x)dV(y) \\
& = \frac{ \epsilon^{-d/2-1}}{m_2[k_0^2]} 
 \int_{\calM} \int_{\calM} 
 {\bf 1}_{\{ \| x-y\| < \delta_\epsilon \}}
 k_0^2 \left(  \frac{\| x - y \|^2}{ \epsilon \hat{\rho}(x)  \hat{\rho}( y) } \right)
\frac{  |\nabla f(x)|^4 \|x-y\|^4   }
{  \hat{\rho}(x)^{2\alpha} \hat{\rho}( y)^{2\alpha} }
p(x) p(y) dV(x)dV(y) \\
& ~~~
+\frac{ \epsilon^{-d/2-1}}{m_2[k_0^2]} 
 \int_{\calM} \int_{\calM} 
 {\bf 1}_{\{ \| x-y\| < \delta_\epsilon \}}
 k_0^2 \left(  \frac{\| x - y \|^2}{ \epsilon \hat{\rho}(x)  \hat{\rho}( y) } \right)
\frac{  O^{[f]}( \|x-y\|^5) }
{  \hat{\rho}(x)^{2\alpha} \hat{\rho}( y)^{2\alpha} }
p(x) p(y) dV(x)dV(y)
=: \textcircled{5} + \textcircled{6}.
\end{align*}
We establish a lemma, which can be proved similarly as in deriving the limit of $\textcircled{1}$ above,
namely, by replacing $\hat{\rho}(x)$ with $\bar{\rho}(x)$ first and then putting back. 
The proof is postponed to
  Appendix \ref{subsec:app-more-lemmas}.

\begin{lemma}\label{lemma:double-int-f2}
Under  \eqref{eq:hatrho-varepsilon} and  \eqref{eq:bound-hatrho-0.1},
suppose $k_0$ satisfies Assumption \ref{assump:h-selftune},
$f \in C^\infty( {\calM}) $ and $\alpha \in \R$, then 
\[
 \epsilon^{-\frac{d}{2}}
\int_{\calM} \int_{\calM}  f(x)^2 
k_0 \left(  \frac{\| x - y \|^2}{ \epsilon \hat{\rho}(x)  \hat{\rho}( y ) } \right)
\frac{p(x) p(y) }{ \hat{\rho}(x)^\alpha \hat{\rho}(y )^\alpha}
 dV(x) dV(y) 
  =  m_0[k_0] \int p^2  f^2 \bar{\rho}^{d-2\alpha}   
 + O^{[f,p]}( \epsilon, \varepsilon_\rho).
\]
\end{lemma}
\vspace{5pt}

We bound $\textcircled{5}$ and  $\textcircled{6}$ respectively, where $\textcircled{5}$ will dominate. 
By a direct bound, we have
\[
|\textcircled{6}|
\le
O^{[f]}( \epsilon^{\frac{3}{2}})
\frac{ \epsilon^{-d/2}}{m_2[k_0^2]} 
 \int_{\calM} \int_{\calM} 
 k_0^2 \left(  \frac{\| x - y \|^2}{ \epsilon \hat{\rho}(x)  \hat{\rho}( y) } \right)
\left( \frac{   \|x-y\|^2}
{ \epsilon \hat{\rho}(x) \hat{\rho}( y) } \right)^{\frac{5}{2}}
  \hat{\rho}(x)^{\frac{5}{2}-2\alpha} \hat{\rho}( y)^{\frac{5}{2}-2\alpha}
p(x) p(y) dV(x)dV(y)\,.
\]
Note that there is $b_6 > 0$, determined by $a_0$ and $a$,  such that 
\[
k_0( r)^2 r^{5/2} \le a_0^2 e^{- 2a r} r^{5/2} \le b_6 e^{-a r},
\quad
\forall r > 0.
\]
Thus,
\[
k_0^2 \left(  \frac{\| x - y \|^2}{ \epsilon \hat{\rho}(x)  \hat{\rho}( y) } \right)
\left( \frac{   \|x-y\|^2}
{ \epsilon \hat{\rho}(x) \hat{\rho}( y) } \right)^{\frac{5}{2}}
\le b_6 e^{-a  \frac{\| x - y \|^2}{ \epsilon \hat{\rho}(x)  \hat{\rho}( y) }},
\]
and then
\[
|\textcircled{6}|
\le
O^{[f]}( \epsilon^{\frac{3}{2}})
\frac{ \epsilon^{-d/2}}{m_2[k_0^2]} 
 \int_{\calM} \int_{\calM} 
 b_6 e^{-a  \frac{\| x - y \|^2}{ \epsilon \hat{\rho}(x)  \hat{\rho}( y) }}
 \frac{p(x) p(y) }{   \hat{\rho}(x)^{2\alpha - \frac{5}{2}} \hat{\rho}( y)^{2\alpha - \frac{5}{2}} }
dV(x)dV(y).
\]
Since the kernel $k_6(r) := b_6 e^{-ar}$ satisfies Assumption \ref{assump:h-selftune},
applying Lemma \ref{lemma:double-int-f2}
with $f$ replaced by 1 
and 
$\alpha$ replaced by $2\alpha - \frac{5}{2}$
gives that 
\[
|\textcircled{6}|
= 
 O^{[f]}( \epsilon^{\frac{3}{2}})
\frac{1}{m_2[k_0^2]} \left( m_0[k_6] \int p^2 \bar{\rho}^{d - 2( 2\alpha - \frac{5}{2} )} \rev{ + O^{[p]}( \epsilon, \varepsilon_\rho) }\right)
= O^{[f,p]} (\epsilon^{\frac{3}{2}}).
\]

Write $G(x) =    |\nabla f(x)|^2$. 
Clearly, since $f\in C^\infty(\calM)$, $G$ is in $C^\infty( {\calM})$.
Define 
\[
k_5 (r) := k_0^2(r) r^2\,.
\]
Then, $k_5$ satisfies Assumption \ref{assump:h-selftune} and $m_0[k_5] = d m_2[k_0^2]$.
As a result,
\begin{align*}
 \textcircled{5}
&\le 
\frac{ \epsilon^{-d/2-1}}{m_2[k_0^2]} 
 \int_{\calM} \int_{\calM} 
 G(x)^2
 k_0^2 \left(  \frac{\| x - y \|^2}{ \epsilon \hat{\rho}(x)  \hat{\rho}( y) } \right)
\frac{  \|x-y\|^4   }
{  \hat{\rho}(x)^{2\alpha} \hat{\rho}( y)^{2\alpha} }
p(x) p(y) dV(x)dV(y) \\
& =
\frac{  \epsilon }{m_2[k_0^2]}
\epsilon^{-d/2} \int_{\calM} \int_{\calM} 
 G(x)^2
 k_0^2 \left(  \frac{\| x - y \|^2}{ \epsilon \hat{\rho}(x)  \hat{\rho}( y) } \right)
\left( \frac{  \|x-y\|^2  }
{  \epsilon \hat{\rho}(x)   \hat{\rho}(y)} \right)^2
\frac{p(x) p(y)}{\hat{\rho}(x)^{2\alpha-2} \hat{\rho}( y)^{2\alpha-2}}
 dV(x)dV(y) \\
 &
 = 
 \frac{  \epsilon }{m_2[k_0^2]}
\epsilon^{-d/2} \int_{\calM} \int_{\calM} 
 G(x)^2
 k_5 \left(  \frac{\| x - y \|^2}{ \epsilon \hat{\rho}(x)  \hat{\rho}( y) } \right)
\frac{p(x) p(y)}{\hat{\rho}(x)^{2\alpha-2} \hat{\rho}( y)^{2\alpha-2}}
 dV(x)dV(y) \\
 & =
 \frac{  \epsilon }{m_2[k_0^2]}
 \cdot \left( m_0[k_5]  \int p^2 G^2 \bar{\rho}^{d-2(2\alpha-2)}  + O^{[f,p]}(\epsilon, \varepsilon_\rho) \right) \\
 & = \epsilon d \int p^2 |\nabla f|^4 \bar{\rho}^{d-2(2\alpha-2)} + O^{[f,p]}(\epsilon^2,  \epsilon \varepsilon_\rho)  \\
& = \epsilon d \int  |\nabla f|^4 p^{1+ \frac{4(\alpha-1)}{d}} + O^{[f,p]}(\epsilon^2,  \epsilon \varepsilon_\rho)  \,,
\end{align*}
where the third equality holds by applying Lemma \ref{lemma:double-int-f2} with $f$ replaced by $G$ and $\alpha$ replaced by $2\alpha-2$.
Putting together,
we have that
\begin{align*}
\textcircled{4}
& \le  \textcircled{5} + \textcircled{6} +  \textcircled{4}_2 \\
 & \le
 \epsilon d \int  |\nabla f|^4 p^{1+ \frac{4(\alpha-1)}{d}} {+ O^{[f,p]}(\epsilon^2,  \epsilon \varepsilon_\rho)} 
 +  O^{[f,p]} (\epsilon^{\frac{3}{2}}) 
+  O^{[f,p]}(\epsilon^{\frac{d}{2}+9}) \\
&= 
 \epsilon d \int  |\nabla f|^4 p^{1+ \frac{4(\alpha-1)}{d}}
+ O^{[f,p]} (\epsilon^{\frac{3}{2}}, \, \epsilon \varepsilon_\rho)\,.
\end{align*}
Plugging the above bounds
back to \eqref{eq:EV12square-circle4}, this gives that
\begin{align*}
\text{Var}(V_{1,2})
&   \le \E V_{1,2}^2 \le
 \frac{ \epsilon^{-d/2-1}}{ \frac{ m_2[k_0]^2}{m_2[k_0^2]}} 
 \left(
 \epsilon d \int  |\nabla f|^4 p^{1+ \frac{4(\alpha-1)}{d}}
+ O^{[f,p]} (\epsilon^{\frac{3}{2}}    {, \, \epsilon \varepsilon_\rho})
\right) \\
& =\frac{m_2[k_0^2]}{ m_2[k_0]^2} 
\epsilon^{-d/2}
\left( d \int  |\nabla f|^4 p^{1+ \frac{4(\alpha-1)}{d}}
+ O^{[f,p]} (\epsilon^{\frac{1}{2}}   {, \,  \varepsilon_\rho})
\right).
\end{align*}
Since  $f$ is not constant valued,
 $ \int  |\nabla f|^4 p^{1+ \frac{4(\alpha-1)}{d}} > 0$,
 \rev{and by that $\epsilon^{1/2}, \, \varepsilon_\rho = o(1)$,}
we have that with \rev{sufficiently} large $N$, 
\[
\nu = \text{Var}(V_{1,2})
\le \Theta^{[1]}( \epsilon^{-d/2}) \int  |\nabla f|^4 p^{1+ \frac{4(\alpha-1)}{d}}
= \epsilon^{-d/2} V_f,
\]
where
\[ 
V_f :=\Theta^{[1]}\left( \int  |\nabla f|^4 p^{1+ \frac{4(\alpha-1)}{d}}\right).
\]
Back to \eqref{eq:bernstein-1}, 
by that $\nu \le  \epsilon^{-d/2} V_f$ and then
\begin{equation}\label{eq:bound-bern-2}
\exp\{ - \frac{ \frac{N}{2} t^2}{2 \nu + \frac{2}{3} t L } \} 
\le 
\exp\{ - \frac{ \frac{N}{2} t^2}{2 \epsilon^{-d/2} V_f + \frac{2}{3} t L } \}  ,
\end{equation}
we now control the r.h.s.
Let $s=\Theta(1)$ to be determined, and we set
\[
t = \sqrt{  \epsilon^{-d/2}V_f  \frac{s  \log N}{N}}\,. 
\]
Since $L  =\Theta^{[f,p]}( \epsilon^{-d/2})$, and by the condition that
$\epsilon^{d/2 } N = \Omega( \log N)$,
$\frac{\log N}{ N \epsilon^{d/2}} = o(1)$, 
and hence $t=o^{[f,p]}(1)$. 
With large $N$, 
\[
\frac{tL}{\epsilon^{-d/2} V_f} 
 =\Theta^{[f,p]} \left( 
\sqrt{ \frac{ \epsilon^{-d/2}}{V_f} \frac{s  \log N}{N} } \right)
= o^{[f,p]}(1) \,. 
\]
Therefore, when $N$ is sufficiently large, we have $\frac{tL}{\epsilon^{-d/2} V_f}<3$.
Then \eqref{eq:bound-bern-2} bounds  the tail probability 
in \eqref{eq:bernstein-1}
to be less than
$ \exp\{ - \frac{N t^2}{8 \epsilon^{-d/2}V_f  } \} = N^{-s/8}$.
Let $s = 80$, 
and use the same argument to bound $\Pr[ \frac{1}{N(N-1)} \sum_{i \neq j} \tilde{V}_{ij}  < -t ]$.
We have that 
w.p. greater than $ 1- 2N^{-10}$
\begin{equation}\label{Proof of Theorem 3.4 Final bound}
\left|  \frac{1}{N(N-1)} \sum_{i \neq j, i,j=1}^N \tilde{V}_{ij}  \right|
\le 
 \sqrt{  \epsilon^{-d/2}V_f  \frac{ 80  \log N}{N}}
 = O \left(
 \sqrt{ \frac{\log N}{ N \epsilon^{d/2}} \int  |\nabla f|^4 p^{1+ \frac{4(\alpha-1)}{d}} } \right).
\end{equation}
Call the event set that \eqref{Proof of Theorem 3.4 Final bound} holds the event $E_{Dir}$.

At last, 
\[
E_N (f,f) =  \left(1-\frac{1}{N} \right)\cdot \frac{1}{N(N-1)} \sum_{i \neq j} V_{ij},
\]
and with \eqref{eq:EV12-formula}
we have shown that under good event $E_{Dir}$, 
\[
\frac{1}{N(N-1)} \sum_{i \neq j} V_{ij} = 
{\calE}_{p_\alpha}(f,f)(1 + O^{[\alpha]}( \varepsilon_\rho) )
 + O^{[f,p]}(\epsilon )
+ O \left(
 \sqrt{ \frac{\log N}{ N \epsilon^{d/2}} \int  |\nabla f|^4 p^{1+ \frac{4(\alpha-1)}{d}} } \right),
\]
which is $ O^{[f,p]}(1)$,
thus $E_N (f,f)$ differs from it by $ O^{[f,p]}(\frac{1}{N}) = o^{[f,p]} ( \sqrt{ \frac{1}{ N}})$,
which is dominated by the variance error.
This finishes the proof of the theorem. 
\end{proof}

%
\begin{proof}[Proof of Theorem \ref{thm:limit-form-fixeps}]
The proof uses same techniques as that of Theorem \ref{thm:limit-form},
and is simplified due to fixed bandwidth kernel.
We track the influence of $\beta$ and $\varepsilon_p$ which differs from the proof of Theorem \ref{thm:limit-form}.
Inherit the notations in the proof of  Theorem \ref{thm:limit-form},
the random variable $V_{ij}$ is now
\[
V_{ij}  
 := \frac{\epsilon^{-\frac{d}{2}}}{ \epsilon m_2}  (f(x_i) - f(x_j))^2 k_0 \left( \frac{\|x_i - x_j\|^2}{\epsilon} \right)
 \frac{1}{ \hat{p}( x_i )^\beta \hat{p}( x_j )^\beta }.
\]
Similarly as before, we define
\[
\textcircled{1} 
: = \E V_{ij} = 
\frac{1}{\epsilon m_2}
\int \int
  (f(x) - f(y))^2 \epsilon^{-\frac{d}{2}} k_0 \left( \frac{\| x-y \|^2}{\epsilon} \right)
 \frac{p(x) p(y) }{ \hat{p}(x )^\beta \hat{p}( y )^\beta }
 dV(x)dV(y),
\]
and 
\begin{align*}
\textcircled{2} 
& : = 
\frac{1}{\epsilon m_2}
\int \int
  (f(x) - f(y))^2 \epsilon^{-\frac{d}{2}} k_0 \left( \frac{\| x-y \|^2}{\epsilon} \right)
 \frac{p(x)^{1-\beta} p(y) }{  \hat{p}( y )^\beta }
 dV(x)dV(y),  \\
 \textcircled{3} 
& : = 
\frac{1}{\epsilon m_2}
\int \int
  (f(x) - f(y))^2 \epsilon^{-\frac{d}{2}} k_0 \left( \frac{\| x-y \|^2}{\epsilon} \right)
 p(x)^{1-\beta} p(y)^{1-\beta} 
 dV(x)dV(y).
\end{align*}
Define $q := p^{1-\beta}$ \rev{which is a non-negative power of $p$ because $\beta \le 1$, thus $q \in C^\infty(\calM)$}. We then have
\begin{align*}
\textcircled{3} 
& =  
\frac{1}{\epsilon m_2}
\int \int
  (f(x) - f(y))^2 \epsilon^{-\frac{d}{2}} k_0 \left( \frac{\| x-y \|^2}{\epsilon} \right)
q(x)  q(y) dV(x)dV(y) \\
& =
\frac{2}{\epsilon m_2}
\int \int
  (f(x)^2 - f(x)f(y)) \epsilon^{-\frac{d}{2}} k_0 \left( \frac{\| x-y \|^2}{\epsilon} \right)
q(x)  q(y) dV(x)dV(y)  \\
& = 
\frac{2}{\epsilon m_2}
\left( 
\int (qf^2)G_\epsilon (q)
- \int (qf)G_\epsilon (fq)
\right).
\end{align*}
By Lemma \ref{lemma:h-integral-diffusionmap},
\[
G_\epsilon (q) = m_0 q + \epsilon \frac{m_2}{2} (wq + \Delta q) + O^{ [q^{(\le 4)} ]}(\epsilon^2),
\]
\[
G_\epsilon (fq) = m_0 fq + \epsilon \frac{m_2}{2} (wfq + \Delta (fq)) + O^{ [(fq)^{(\le 4)}] }(\epsilon^2),
\]
and thus, 
\begin{align}
\textcircled{3}
& = 
\int qf ( f   \Delta q -   \Delta (fq) )+ O^{[f, \rev{q} ]}(\epsilon)  
 = 
- \int q^2 f   ( \Delta f + 2 \frac{\nabla q}{q} \cdot \nabla f )
+ O^{[f, \rev{q}]}(\epsilon) \nonumber \\
& = 
 - \langle f, \Delta_{p_\beta} f \rangle_{p_\beta} + O^{[f, \rev{q}]}(\epsilon) 
\quad \text{($p_\beta : = q^2 = p^{2-2\beta}$)} \nonumber \\
 & = {\calE}_{p_\beta}(f,f) + O^{[f,\rev{q}]}(\epsilon).
 \label{eq:bound-fixeps-circle3}
\end{align}

To bound $|\textcircled{3} - \textcircled{2} |$: 
Because
\[
\textcircled{3} - \textcircled{2}  
= \frac{1}{\epsilon m_2}
\int \int
  (f(x) - f(y))^2 \epsilon^{-\frac{d}{2}} k_0 \left( \frac{\| x-y \|^2}{\epsilon} \right)
 p(x)^{1-\beta} p(y) ( p(y)^{-\beta} - \hat{p}(y)^{-\beta})
 dV(x)dV(y),
\]
by the positivity of $p$ and $k_0$,
\[
| \textcircled{3} - \textcircled{2}   |
\le \frac{1}{\epsilon m_2}
\int \int
  (f(x) - f(y))^2 \epsilon^{-\frac{d}{2}} k_0 \left( \frac{\| x-y \|^2}{\epsilon} \right)
 p(x)^{1-\beta} p(y) | p(y)^{-\beta} - \hat{p}(y)^{-\beta} |
 dV(x)dV(y).
\]
Note that for any $y$, by the Mean Value Theorem and that $ | \hat{p}(y) - p(y) |/ p(y) < \varepsilon_p < 0.1$, 
 $\exists \xi$ between $\hat{p}(y)$ and $p(y)$ and thus $\xi^{-\beta-1} \le  \max\{ 1.1^{-\beta-1}, 0.9^{-\beta-1}\}  p(y)^{-\beta-1}$,
 and then we have
\begin{align}
 |\hat{p}(y)^{-\beta} - p(y)^{-\beta} | 
& =  \rev{|\beta|} \xi^{-\beta-1} | \hat{p}(y) - p(y) |
\le  \rev{|\beta|} \max\{ 1.1^{-\beta-1}, 0.9^{-\beta-1}\}  p(y)^{-\beta-1}  | \hat{p}(y) - p(y) |   \nonumber \\
& \le c_\beta   \rev{|\beta|}  p(y)^{-\beta}  \varepsilon_p,
\quad 
c_\beta:=  \max\{ 1.1^{-\beta-1}, 0.9^{-\beta-1}\}.
\label{eq:hatp-p-beta-bound}
\end{align}
 Thus,
\[
| \textcircled{3} - \textcircled{2}   |
\le 
   \rev{|\beta|}
c_\beta  \varepsilon_p
\frac{1}{\epsilon m_2}
\int \int
  (f(x) - f(y))^2 \epsilon^{-\frac{d}{2}} k_0 \left( \frac{\| x-y \|^2}{\epsilon} \right)
 p(x)^{1-\beta} 
  p(y)^{1-\beta}
 dV(x)dV(y)
 =  \rev{|\beta|}
  c_\beta  \varepsilon_p \textcircled{3}.
\]
Combined with \eqref{eq:bound-fixeps-circle3}, we have that 
\begin{align}
 \textcircled{2} 
& =  \textcircled{3} (1 + O^{[1]}(\beta c_\beta  \varepsilon_p ))
 = ({\calE}_{p_\beta}(f,f) + O^{[f, \rev{q}]}(\epsilon))  (1 + O^{[1]}(\beta c_\beta  \varepsilon_p ))  \nonumber \\
& = {\calE}_{p_\beta}(f,f) (1 + O^{[1]}( \rev{c_\beta} \beta  \varepsilon_p )) + O^{[f, \rev{q}]}(\epsilon). 
 \label{eq:bound-fixeps-circle2}
\end{align}

To bound $ |\textcircled{2} - \textcircled{1} |$:
 By \eqref{eq:hatp-p-beta-bound}
\begin{align*}
& | \textcircled{2} - \textcircled{1}  |
\le \frac{1}{\epsilon m_2}
\int \int
  (f(x) - f(y))^2 \epsilon^{-\frac{d}{2}} k_0 \left( \frac{\| x-y \|^2}{\epsilon} \right)
  p(y)\hat{p}(y)^{-\beta}
  p(x)  | p(x)^{-\beta}  - \hat{p}(x)^{-\beta} |
 dV(x)dV(y) \\
 & 
\le  \rev{|\beta|} c_\beta \varepsilon_p \frac{1}{\epsilon m_2}
\int \int
  (f(x) - f(y))^2 \epsilon^{-\frac{d}{2}} k_0 \left( \frac{\| x-y \|^2}{\epsilon} \right)
  p(y)\hat{p}(y)^{-\beta}
 p(x)^{1-\beta} 
 dV(x)dV(y) 
 =  \rev{|\beta|}
  c_\beta \varepsilon_p \textcircled{2} .
\end{align*}
Then, together with \eqref{eq:bound-fixeps-circle2}, we have that 
\begin{align*}
 \E V_{ij}
 & =  \textcircled{1} 
 =  \textcircled{2} (1+ O^{[1]}( \beta c_\beta \varepsilon_p))
 = ( {\calE}_{p_\beta}(f,f) (1 + O^{[1]}(\beta \rev{c_\beta}  \varepsilon_p )) + O^{[f,\rev{q}]}(\epsilon)) 
 (1+ O^{[1]}(\beta \rev{c_\beta}  \varepsilon_p)) \\
& =  {\calE}_{p_\beta}(f,f) (1 + O^{[1]}(\beta \rev{c_\beta}   \varepsilon_p )) 
 + O^{[f,\rev{q}]}(\epsilon),
 \rev{
 \quad \text{and } O^{[f,\rev{q}]}(\epsilon) =O^{[f,\rev{p, \beta}]}(\epsilon) \text{ by definition.}
}
\end{align*}
\rev{In the special case where $\beta=0$, we have $q = p$,
and $\textcircled{1} = \textcircled{2} = \textcircled{3} $.
Then \eqref{eq:bound-fixeps-circle3} gives that
\[
 \E V_{ij}
  =  \textcircled{1} 
 =  {\calE}_{p^2}(f,f)  + O^{[f,p]}(\epsilon).
\]}

The boundedness  of $|V_{ij}|$ follows by the same  argument of truncation on the $\delta_\epsilon$ ball as in the proof of Theorem \ref{thm:limit-form},  which gives 
\[
|V_{ij}| \le L 
= \Theta^{[f,p \rev{, \beta}]}( 
\epsilon^{-d/2} ).
\]
The variance 
\begin{align*}
\text{Var}(V_{ij}) 
& \le  \E V_{ij}^2
=\int \int  \frac{\epsilon^{-d-2}}{m_2^2}  (f(x) - f(y))^4 k_0^2 \left( \frac{\|x_i - x_j\|^2}{\epsilon} \right) 
\frac{p(x)p(y)}{\hat{p}(x)^{2 \beta}  \hat{p}(y)^{2 \beta}}
dV(x) dV(y), 
\end{align*}
and, similarly as in the proof of Theorem \ref{thm:limit-form}, we can show that 
\[
\E V_{ij}^2  \le  \epsilon^{-d/2} V_f,
 \quad
V_f  = \Theta^{[1]}(
\int |\nabla f|^4 p^{2-4\beta}).
\]
Thus, 
by the V-statistics decoupling argument,
w.p. higher than $ 1- 2N^{-10}$,
the variance error is 
\[
O^{[1]} 
\left( 
\sqrt{ \frac{\log N}{ N \epsilon^{d/2}}  \int |\nabla f|^4 p^{2- 4\beta}} 
\right).
\]
The normalization $\frac{1}{N^2}$ and  $\frac{1}{N(N-1)}$ in the V-statistics incurs higher order error,
and putting together bias and variance error proves the theorem.
\end{proof}

\subsection{Proofs in Section \ref{subsec:Lnf-convergence}
(Theorems \ref{thm:limit-pointwise-rw-3}, \ref{thm:limit-pointwise-un-3}, \ref{thm:limit-weak-un} and \ref{thm:limit-pointwise-rw-fix-epsilon}) }

%
%
%
%

\begin{proof}[Proof of Theorem \ref{thm:limit-pointwise-rw-3}]
{ Define $ \tilde{m}:= \frac{m_2}{2m_0}$,}
and 
rewrite \eqref{eq:def-Lalapha-rw} as
\begin{equation}
L^{(\alpha)}_{ rw'}
f(x) = 
\frac{1}{\epsilon \tilde{m} \hat{\rho}(x)^{2}}
\left(
\frac{ \frac{1}{N} \sum_{j=1}^{N}   F_j(x)}
{ \frac{1}{N} \sum_{j=1}^{N}  G_j(x) }
 -f(x) \right),
\end{equation}
where 
\begin{equation}\label{eq:def-Fj-Gj}
F_j  : =  \epsilon^{-d/2}  k_0 \left(  \frac{\| x - x_j\|^2}{  \epsilon  \hat{\rho}( x)  \hat{\rho}( x_j)}  \right) \frac{f(x_{j})}{\hat{\rho}(x_j)^{\alpha} } ,
\quad
G_j  :=  \epsilon^{-d/2} {  k_0 \left(  \frac{\| x - x_j\|^2}{  \epsilon  \hat{\rho}( x)  \hat{\rho}( x_j) } \right)} \frac{1}{\hat{\rho}(x_j)^{\alpha} }.
\end{equation}
Note that since $x_j$  $\sim p$ i.i.d., 
$\{F_j\}_{j=1}^N$ are i.i.d. rv's, so are $\{G_j\}_{j=1}^N$,
while $F_j$'s and $G_j$'s are dependent. 
The expectations are
\begin{align*}
\E F (x)
& = \epsilon^{-d/2} \int_{\calM} {  k_0 \left(  \frac{\| x - y\|^2}{  \epsilon  \hat{\rho}( x)  \hat{\rho}( y) } \right)} \frac{f( y )}{\hat{\rho}(y)^{\alpha} } p(y) dV(y), \\
\E G (x)
& = \epsilon^{-d/2} \int_{\calM} {  k_0 \left(  \frac{\| x - y\|^2}{  \epsilon  \hat{\rho}( x)  \hat{\rho}( y) } \right)} \frac{1}{\hat{\rho}(y)^{\alpha} } p(y) dV(y).
\end{align*}
Following the strategy in \cite{singer2006graph,berry2016variable}
to analyze the bias and variance errors respectively,
 we will show that 
\begin{itemize}

\item
The bias: 
\begin{equation}\label{eq:bias-error-tildeLN-goal-3}
\frac{1}{\epsilon \tilde{m} \hat{\rho}(x)^{2}}
\left(
\frac{ \E F(x)}{ \E G(x) } -f(x) \right)
\overset{?}{=}
{\calL}^{(\alpha)} f (x)
+ O^{[f,p]} \left(   \epsilon, \, \frac{\varepsilon_\rho}{ \epsilon} \right)
\end{equation}

\item
The variance:  
\begin{equation}\label{eq:variance-error-tildeLN-goal-3}
\frac{1}{\epsilon \tilde{m} \hat{\rho}(x)^{2 }}
\left(
\frac{ \frac{1}{N} \sum_{j=1}^{N}   F_j(x)}
{ \frac{1}{N} \sum_{j=1}^{N}  G_j(x) }
 - \frac{ \E F(x)}{ \E G(x) }
 \right)
\overset{?}{=} 
O^{[1]} \left(  \rev{ \| \nabla f \|_\infty}  p(x)^{1/d}   
\sqrt{     \frac{\log N}{   N  \epsilon^{d/2+1}}   }  \right)
\end{equation}
\end{itemize}

\underline{Proof of \eqref{eq:bias-error-tildeLN-goal-3}}:
By the definition of $G^{({\rho})}_{\epsilon}$ in \eqref{eq:def-G-R-rho-epsilon},
\[
\E F(x) 
 = \hat{\rho}(x)^{d/2} G^{(\hat{\rho})}_{\epsilon \hat{\rho}(x)} ( \frac{fp}{\hat{\rho}^{\alpha}} ) (x), 
 \]
 yet  $\frac{fp}{\hat{\rho}^{\alpha}}$ is not $C^4$.
 In order to apply Lemma \ref{lemma:right-operator-3} and \ref{lemma:right-operator-repalce-rho},
 we compare to replacing it with $\frac{fp}{\bar{\rho}^{\alpha}}$:

 \begin{lemma}\label{lemma:replace-hatrhoalpha-GR-hatrho}
Suppose notation and condition {are the} same as 
those in Lemma \ref{lemma:right-operator-3} and \ref{lemma:right-operator-repalce-rho}.
{When} $\epsilon$ is sufficiently small, 
for some constant $c_\rho''$ determined by $({\calM}, k_0, \rho_{max}, \rho_{min})$,
 \[
\sup_{x \in {\calM}} | G_\epsilon^{ (\tilde{\rho})} (\frac{f }{ \tilde{\rho}^\alpha} )(x) -
 G_\epsilon^{ (\tilde{\rho})} (\frac{f }{ {\rho}^\alpha} ) (x)|
 \le c_\rho'' \|f\|_\infty \varepsilon = O^{ [f,\rho] }(\varepsilon).
 \]
 \end{lemma}
{The} proof of Lemma {\ref{lemma:replace-hatrhoalpha-GR-hatrho}} is postponed to Appendix \ref{subsec:app-more-lemmas}.
 Applying Lemmas \ref{lemma:right-operator-3}, \ref{lemma:right-operator-repalce-rho}
 and \ref{lemma:replace-hatrhoalpha-GR-hatrho},
 where  $\rho = \bar{\rho}$, $\tilde{\rho} = \hat{\rho}$ and 
 ``$f$'' in the lemmas is replaced with $ fp$, we have
\begin{align*}
& \E F(x) 
 = \hat{\rho}(x)^{d/2} 
 \left( 
 G^{(\hat{\rho})}_{\epsilon \hat{\rho}(x)} ( \frac{fp}{\bar{\rho}^{\alpha}} ) (x) 
 + O^{[f,p]}(\varepsilon_\rho) 
 \right)  \quad \text{(by Lemma \ref{lemma:replace-hatrhoalpha-GR-hatrho})}
 \\
& 
= \hat{\rho}(x)^{d/2} 
 \left( 
 G^{(\bar{\rho})}_{\epsilon \hat{\rho}(x)} ( \frac{fp}{\bar{\rho}^{\alpha}} ) (x) 
 + O^{[f,p]}(\varepsilon_\rho)
 + O^{[f,p]}(\varepsilon_\rho) 
 \right) \quad \text{(by Lemma \ref{lemma:right-operator-repalce-rho})}
 \\
 &=
 \hat{\rho}(x)^{d/2} 
 \left( 
m_0 f p \bar{\rho}^{\frac{d}{2}-\alpha}(x) 
+ \epsilon \hat{\rho} \frac{m_2}{2} (   \omega  f p \bar{\rho}^{1+\frac{d}{2}-\alpha}  + \Delta ( f p \bar{\rho}^{1+\frac{d}{2}-\alpha}  ) )(x) 
+ \hat{\rho}^2(x) O^{[f,p]}(\epsilon^2)   
 + O^{[f,p]}(\varepsilon_\rho)
  \right),
\end{align*}
where the residual terms in big-$O$ are bounded uniformly for all $x \in {\calM}$.
{Below}, we omit the variable $x$ in the notation when there is no confusion.

Because $ 0.9 \bar{\rho} \le \hat{\rho} \le 1.1 \bar{\rho}$ for all $x \in {\calM}$, 
for any power $\gamma \in \R$,
$   \hat{\rho}^\gamma  $ lies between $\bar{\rho}^\gamma 1.1^\gamma$
and $\bar{\rho}^\gamma 0.9^\gamma$ (the order depending on the sign of $\gamma$),
and then uniformly bounded between
$ (0.9\rho_{min})^\gamma$ and $(1.1 \rho_{max})^\gamma$,
both of which are $\Theta^{[p]}(1)$ constants. 
We can also bound $| \hat{\rho}^\gamma - \bar{\rho}^\gamma |$ as in \eqref{eq:bound-anypower-hatrho-barho},
and in summary we have
\begin{equation}\label{eq:diff-any-power}
\sup_{x \in {\calM}} | \hat{\rho}^\gamma (x)|  = O^{[p]}(1),
\quad
\sup_{x \in {\calM}}| \hat{\rho}^{\gamma}(x) - \bar{\rho}^{\gamma}(x) | \le O^{[p]}( \varepsilon_\rho).
\end{equation}

We proceed with these bounds. 
We have shown, omitting the evaluation of $x$ in the notation, that
\[
\E F 
= \hat{\rho}^{d/2} 
 \left( 
m_0 f p \bar{\rho}^{\frac{d}{2}-\alpha} 
+ \epsilon \hat{\rho} \frac{m_2}{2} (   \omega  f p \bar{\rho}^{1+\frac{d}{2}-\alpha}  + \Delta ( f p \bar{\rho}^{1+\frac{d}{2}-\alpha}  ) ) 
  \right)
+  O^{[f,p]}(\epsilon^2, \, \varepsilon_\rho), 
\]
and by \eqref{eq:diff-any-power} with $\gamma = \frac{d}{2}$ and $\frac{d}{2}+1$,
\begin{equation}\label{eq:EF-3}
\E F 
= \bar{\rho}^{d/2} 
 \left( 
m_0 f p \bar{\rho}^{\frac{d}{2}-\alpha} 
+ \epsilon \bar{\rho} \frac{m_2}{2} (   \omega  f p \bar{\rho}^{1+\frac{d}{2}-\alpha}  + \Delta ( f p \bar{\rho}^{1+\frac{d}{2}-\alpha}  ) ) 
  \right)
+  O^{[f,p]}(\epsilon^2, \, \varepsilon_\rho).
\end{equation}
Similarly, we have
 \begin{align}
 \E G (x)
 &  =  \hat{\rho}(x)^{d/2} G^{(\hat{\rho})}_{\epsilon \hat{\rho}(x)} ( \frac{p}{\hat{\rho}^{\alpha}} ) (x) \nonumber \\
& =
\bar{\rho}^{d/2} 
 \left( 
m_0 p \bar{\rho}^{\frac{d}{2}-\alpha} 
+ \epsilon \bar{\rho} \frac{m_2}{2} (   \omega  p \bar{\rho}^{1+\frac{d}{2}-\alpha}  + \Delta ( p \bar{\rho}^{1+\frac{d}{2}-\alpha}  ) ) 
  \right)
+  O^{[p]}(\epsilon^2, \, \varepsilon_\rho),
\label{eq:EG-3}
\end{align}
and expanding $\E G $ to the $O(\epsilon)$ term only gives that
\[
\E G(x) 
= 
m_0 p \bar{\rho}^{ d -\alpha}(x) 
+ r_G(x), 
\quad
 \| r_G \|_\infty =  O^{[p]}(\epsilon, \, \varepsilon_\rho).
\]
Because $\inf_{x \in {\calM}} m_0 p \bar{\rho}^{ d -\alpha}(x)$ is a strictly positive constant depending on $p$,
and  $\epsilon$ and $\varepsilon_\rho$  are $o(1)$,
thus  \rev{when $N$ is large and the threshold depends on $(\calM, p, \alpha)$,}
$\| r_G\|_\infty < \inf_{x \in {\calM}} m_0 p \bar{\rho}^{ d -\alpha}(x)$.
Then for any $x \in {\calM}$,
$|r_G(x)| < m_0 p \bar{\rho}^{ d -\alpha}(x)$,
and we have
\begin{equation}\label{eq:1overEG-1}
\frac{1}{\mathbb E G (x)} 
= 
\frac{1}{ m_0 p \bar{\rho}^{ d -\alpha} (x)} \sum_{l=0}^\infty \left( - \frac{ r_G(x)}{m_0 p \bar{\rho}^{ d -\alpha} (x) } \right)^l
= \frac{1}{ m_0 p \bar{\rho}^{ d -\alpha} (x)} + O^{[p]}( \epsilon, \, \varepsilon_\rho ).
\end{equation}
Meanwhile, 
\[
\E F(x) - f(x) \E G(x) 
=
\epsilon \frac{m_2}{2}  \bar{\rho}^{d/2+1}  
\left[   \Delta ( fp \bar{\rho}^{1+\frac{d}{2} -\alpha} )  - f \Delta ( p \bar{\rho}^{1+\frac{d}{2} -\alpha} )  \right] 
+ O^{[f,p]}(  \epsilon^2, \, \varepsilon_\rho).
\]
Note that the quantity in the square brackets
 \[
\left[ \cdots \right]
= p \bar{\rho}^{1+\frac{d}{2} -\alpha}
\left(  \Delta f 
+    2 \frac{\nabla p}{p}  \cdot \nabla f  +    (2 + d -2\alpha) \frac{\nabla \bar{\rho}}{\bar{\rho}}  \cdot \nabla f  \right),
\]
and then by the definition of ${\calL}_\rho^{(\alpha)}$ in \eqref{eq:def-Lrho-alpha},
and that $ {\calL}_{\bar{\rho}}^{(\alpha)} =  {\calL}^{(\alpha)}$,
we have
\begin{equation}\label{eq:EF-fEG-3}
\E F(x) - f(x) \E G(x) 
=
\epsilon \frac{m_2}{2}  
p \bar{\rho}^{d+2 -\alpha}   {\calL}^{(\alpha)}  f 
+ O^{[f,p]}(  \epsilon^2, \, \varepsilon_\rho).
\end{equation}
Putting together, we have
\begin{align}
\frac{\E F(x)}{\E G(x)} - f(x)
& = \frac{ \E F(x) - f(x) \E G(x) }{\E G(x) } \nonumber \\
& =
\left( 
\epsilon \frac{m_2}{2}  
p \bar{\rho}^{d+2 -\alpha}   {\calL}^{(\alpha)}  f 
+ O^{[f,p]}(  \epsilon^2, \, \varepsilon_\rho)
 \right)
\left( 
\frac{1}{ m_0 p \bar{\rho}^{ d -\alpha} (x)} + O^{[p]}( \epsilon, \, \varepsilon_\rho )
 \right) 	\nonumber
\\
& =
\epsilon \frac{m_2}{2 m_0} \bar{\rho}^2 {\calL}^{(\alpha)} f (x)
+ O^{[f,p]} \left(  \epsilon^2, \, \varepsilon_\rho \right)
+ O^{[f,p]} \left(  \epsilon^2, \, \epsilon \varepsilon_\rho \right)  \nonumber \\
& = \epsilon \tilde{m} \bar{\rho}^2 {\calL}^{(\alpha)} f (x)
+ O^{[f,p]} \left(  \epsilon^2, \, \varepsilon_\rho \right),
\label{eq:bound-bias-error-term1}
\end{align}
and then
\[
\frac{1}{\epsilon \tilde{m} \bar{\rho}^2 }\left(\frac{\E F(x)}{\E G(x)} - f(x) \right)
= {\calL}^{(\alpha)} f (x)
+ O^{[f,p]} \left(  \epsilon, \, \frac{ \varepsilon_\rho}{\epsilon} \right)
= O^{[f,p]}(1).
\]
Finally,  by \eqref{eq:diff-any-power} with $\gamma = -2$,
\[
\left|
\frac{1}{\epsilon \tilde{m}}
\left( \frac{1}{ \hat{\rho}(x)^{2}} - \frac{1}{ \bar{\rho}(x)^{2}} \right) 
\left(\frac{\E F(x)}{\E G(x)} - f(x) \right) \right|
=
O^{[p]}(\varepsilon_\rho)
\frac{1}{\epsilon \tilde{m} \bar{\rho}^2 }
\left| \frac{\E F(x)}{\E G(x)} - f(x) \right|
=  O^{[p]}(\varepsilon_\rho) O^{[f,p]}(1),
\]
and, using triangle inequality, this together with \eqref{eq:bound-bias-error-term1} gives the following
\[
\frac{1}{\epsilon \tilde{m} \hat{\rho}^2 (x) }\left(\frac{\E F(x)}{\E G(x)} - f(x) \right)
= O^{[f,p]}(\varepsilon_\rho)
+  {\calL}^{(\alpha)} f (x)
+ O^{[f,p]} \left(  \epsilon, \, \frac{ \varepsilon_\rho}{\epsilon} \right),
\]
where the constants in $O^{[f,p]}(\cdot)$ are uniform for $x \in {\calM}$. This proves  \eqref{eq:bias-error-tildeLN-goal-3}.
\vspace{5pt}

\underline{Proof of \eqref{eq:variance-error-tildeLN-goal-3}}:
 By definition, we have
\begin{equation}\label{eq:def-Yj-3}
\frac{ \frac{1}{N} \sum_{j=1}^{N}   F_j(x)}
{ \frac{1}{N} \sum_{j=1}^{N}  G_j(x) }
 - \frac{ \E F(x)}{ \E G(x) }
= \frac{ \frac{1}{N} \sum_{j=1}^{N}   
( F_j(x) \E G(x) -  G_j(x) \E F(x) )}
{\E G(x) \cdot \frac{1}{N} \sum_{j=1}^{N}  G_j(x)}
=: \frac{ \frac{1}{N} \sum_{j=1}^{N}   Y_j(x) }
{\E G(x) \cdot \frac{1}{N} \sum_{j=1}^{N}  G_j(x)},
\end{equation}
where $Y_j(x) = F_j(x) \E G(x) -  G_j(x) \E F(x)$.
Below we omit  
$x$, which is fixed, in the notation.
We consider the concentration of 
$ \frac{1}{N} \sum_{j=1}^{N}  G_j$
and 
$\frac{1}{N} \sum_{j=1}^{N}  Y_j$  respectively.

We have shown in \eqref{eq:EG-3} that
\[
\E G
= m_0 p  {\bar{\rho}}^{d -\alpha} 
+ O^{[p]}(   \epsilon, \, \varepsilon_\rho )
= \Theta^{[1]}( 
m_0 p  {\bar{\rho}}^{d -\alpha} ),
\]
and by the boundedness of $k_0$ and uniform boundedness of $\hat{\rho}$,
$|G_j| \le L_G = \Theta^{[p]} (
\epsilon^{-d/2})$.
Using {the} same argument to analyze the operator $G^{(\hat{\rho})}_{\epsilon \hat{\rho}(x)}$
as {that} in {Lemmas} \ref{lemma:right-operator-3}, \ref{lemma:right-operator-repalce-rho}
 and \ref{lemma:replace-hatrhoalpha-GR-hatrho}, 
 the variance
\begin{align}
& \text{Var}(G_j)
 \le \E G_j^2 
~~~ = \int_{\calM}
\epsilon^{-d} {  k_0^2 \left(  \frac{\| x - y\|^2}{  \epsilon  \hat{\rho}( x)  \hat{\rho}( y) } \right)} 
\frac{p(y)}{\hat{\rho}(y)^{2\alpha} } dV(y)  \nonumber \\
&~~~ =
\epsilon^{-d/2} \hat{\rho}(x)^{d/2}
G^{(\hat{\rho})}_{\epsilon \hat{\rho}(x)}[k_0^2]( \frac{p}{\hat{\rho}^{2 \alpha}})(x) \nonumber  \\
&~~~ =
\epsilon^{-d/2}
\left\{  m_0[k_0^2] p {\bar{\rho}}^{d - 2\alpha}
+ \epsilon   \frac{m_2[k_0^2]}{2} \bar{\rho}^{1+\frac{d}{2}} (   \omega  p \bar{\rho}^{1+\frac{d}{2}-2 \alpha} 
+ \Delta ( p  \bar{\rho}^{1+\frac{d}{2} -2\alpha} ) ) 
+   O^{[p]} \left(  \epsilon^2, \, \varepsilon_\rho \right) 
\right\} 
\nonumber\\
& ~~~= \epsilon^{-d/2}
\left\{  m_0[k_0^2] p {\bar{\rho}}^{d - 2\alpha}
+ O^{[p]} \left(   \epsilon, \, \varepsilon_\rho \right) 
\right\} \label{eq:EGsquare-3} \\
&~~~ 
 \le \bar{\nu}_G \rev{(x)}  =\Theta( 
\epsilon^{-d/2} m_0[k_0^2] p {\bar{\rho}}^{d - 2\alpha}\rev{(x)} )
\rev{~\text{ with large $N$, since $p {\bar{\rho}}^{d - 2\alpha}(x) \ge \Theta^{[p]}(1)>0 $ for all $x$.}}
 \nonumber
\end{align}
By that $\frac{\log N}{N \epsilon^{d/2} } = o( 1)$, 
\rev{when $N$ is large,
$\sqrt{ 40 \frac{\log N}{N} \bar{\nu}_G(x)} <  \frac{ 3\bar{\nu}_G(x)}{L_G}$ for all $x$, and then}  w.p. higher than  $> 1- 2N^{-10}$,
\[
\left| \frac{1}{N} \sum_{j=1}^{N}  G_j - \E G 
\right|  \le  \sqrt{  \frac{40 \log N}{N}  \bar{\nu}_G  }  
= O^{[p]} \left( \sqrt{ \frac{\log N}{N \epsilon^{d/2} }  }  \right),
\]
which we define as {the} good event $E_2$.
\rev{
The threshold of large $N$ needed for $\text{Var}(G_j) \le  \bar{\nu}_G \rev{(x)} $ and for applying the sub-Gaussian tail in Bernstein inequality 
depends on $(\calM, p, \alpha)$.
}
Under $E_2$,
\[
\frac{1}{N} \sum_{j=1}^{N}  G_j 
= m_0 p  {\bar{\rho}}^{d -\alpha} 
+ O^{[p]}(   \epsilon, \, \varepsilon_\rho )
+  O^{[p]} \left( \sqrt{ \frac{\log N}{N \epsilon^{d/2} }  }  \right)
=\Theta( 
 m_0 p  {\bar{\rho}}^{d -\alpha}) \,,
\]
and then
\begin{equation}\label{eq:EG-sumG-O(1)^2-3}
\E G \cdot  \frac{1}{N} \sum_{j=1}^{N}  G_j 
= \Theta( 
(m_0 p  {\bar{\rho}}^{d -\alpha})^2 ) \,. 
\end{equation}

To analyze the independent sum
$\frac{1}{N} \sum_{j=1}^{N}  Y_j$, 
first note that $\E Y_j =0$.
For boundedness of $Y_j$, because
\begin{equation}\label{eq:boundedness-Gj-Fj}
\E G = O^{[p]}(1), 
\quad \E F = O^{[f,p]}(1),
\quad
|G_j| \le L_G = \Theta^{[p]}( \epsilon^{-d/2}),
\quad |F_j| \le L_F = \Theta^{[f,p]}( \epsilon^{-d/2}),
\end{equation}
we have that
\[
|Y_j| \le | F_j|  |\E G| +  |G_j|  |\E F | \le L_Y 
= \Theta^{[f,p]}( 
\epsilon^{-d/2}).
\]
For the variance of $Y_j$,
\[
\E Y_j^2 
= \E (F_j \E G -  G_j \E F)^2
=  \E F^2 (\E G)^2 + \E  G^2 (\E F)^2 - 2 \E (F G) \E F \E G,
\]
and  we have
\begin{align}
 \E F^2 
&= \int_{\calM}
\epsilon^{-d} {  k_0^2 \left(  \frac{\| x - y\|^2}{  \epsilon  \hat{\rho}( x)  \hat{\rho}( y) } \right)} 
\frac{f(y)^2 p(y)}{\hat{\rho}(y)^{2\alpha} } dV(y)  \nonumber \\
& =
\epsilon^{-d/2} \hat{\rho}(x)^{d/2}
G^{(\hat{\rho})}_{\epsilon \hat{\rho}(x)}[k_0^2]( \frac{ f^2 p}{\hat{\rho}^{2 \alpha}})(x) \nonumber  \\
& =
\epsilon^{-d/2}
\left\{  m_0[k_0^2] f^2 p {\bar{\rho}}^{d - 2\alpha}
+ \epsilon   \frac{m_2[k_0^2]}{2} \bar{\rho}^{1+\frac{d}{2}} (   \omega  f^2 p \bar{\rho}^{1+\frac{d}{2}-2 \alpha} 
+ \Delta ( f^2 p  \bar{\rho}^{1+\frac{d}{2} -2\alpha} ) ) 
+  O^{[f,p]} \left(   \epsilon^2, \, \varepsilon_\rho \right) 
\right\}, 
\nonumber  \\
 \E [FG] 
&= \int_{\calM}
\epsilon^{-d} {  k_0^2 \left(  \frac{\| x - y\|^2}{  \epsilon  \hat{\rho}( x)  \hat{\rho}( y) } \right)} 
\frac{f(y) p(y)}{\hat{\rho}(y)^{2\alpha} } dV(y)  \nonumber \\
& =
\epsilon^{-d/2} \hat{\rho}(x)^{d/2}
G^{(\hat{\rho})}_{\epsilon \hat{\rho}(x)}[k_0^2]( \frac{ f p}{\hat{\rho}^{2 \alpha}})(x) \nonumber  \\
& =
\epsilon^{-d/2}
\left\{  m_0[k_0^2] f p {\bar{\rho}}^{d - 2\alpha}
+ \epsilon   \frac{m_2[k_0^2]}{2} \bar{\rho}^{1+\frac{d}{2}} (   \omega  f p \bar{\rho}^{1+\frac{d}{2}-2 \alpha} 
+ \Delta ( f p  \bar{\rho}^{1+\frac{d}{2} -2\alpha} ) ) 
+   O^{[f,p]} \left(  \epsilon^2, \, \varepsilon_\rho \right) 
\right\}.
\nonumber 
\end{align}
Together with  \eqref{eq:EGsquare-3} and \eqref{eq:EF-3}\eqref{eq:EG-3},
and defining $m_0':=m_0[k_0^2] $ and  $m_2':=m_2[k_0^2] $,
we have that
{\allowdisplaybreaks
\begin{align*}
\E Y_j^2
&=
\epsilon^{-d/2}
\left\{  m_0' f^2 p {\bar{\rho}}^{d - 2\alpha}
+ \epsilon   \frac{m_2'}{2} \bar{\rho}^{1+\frac{d}{2}} (   \omega  f^2 p \bar{\rho}^{1+\frac{d}{2}-2 \alpha} 
+ \Delta ( f^2 p  \bar{\rho}^{1+\frac{d}{2} -2\alpha} ) ) 
+   O^{[f,p]} \left(  \epsilon^2, \, \varepsilon_\rho\right) 
\right\}\\
&~~~~~~
\cdot \left\{
m_0 p  {\bar{\rho}}^{d -\alpha} 
+ \epsilon \frac{m_2}{2}  \bar{\rho}^{d/2+1}  
(   \omega  p \bar{\rho}^{1+\frac{d}{2} -\alpha } 
+ \Delta ( p \bar{\rho}^{1+\frac{d}{2} -\alpha} ) ) 
+  O^{[p]} \left(  \epsilon^2, \, \varepsilon_\rho\right)
\right\}^2 \\
& ~~~
 +\epsilon^{-d/2}
\left\{  m_0' p {\bar{\rho}}^{d - 2\alpha}
+ \epsilon   \frac{m_2'}{2} \bar{\rho}^{1+\frac{d}{2}} (   \omega  p \bar{\rho}^{1+\frac{d}{2}-2 \alpha} 
+ \Delta ( p  \bar{\rho}^{1+\frac{d}{2} -2\alpha} ) ) 
+   O^{[p]} \left(  \epsilon^2, \, \varepsilon_\rho\right) 
\right\}  \\
& ~~~~~~
\cdot \left\{
m_0 fp  {\bar{\rho}}^{d -\alpha} 
+ \epsilon \frac{m_2}{2}  \bar{\rho}^{d/2+1}  
(   \omega  fp \bar{\rho}^{1+\frac{d}{2} -\alpha } 
+ \Delta ( fp \bar{\rho}^{1+\frac{d}{2} -\alpha} ) ) 
+O^{[f,p]} \left(  \epsilon^2, \, \varepsilon_\rho\right)
\right\}^2 \\
&~~~
 - 2 \epsilon^{-d/2}
\left\{  m_0' f p {\bar{\rho}}^{d - 2\alpha}
+ \epsilon   \frac{m_2'}{2} \bar{\rho}^{1+\frac{d}{2}} (   \omega  f p \bar{\rho}^{1+\frac{d}{2}-2 \alpha} 
+ \Delta ( f p  \bar{\rho}^{1+\frac{d}{2} -2\alpha} ) ) 
+  O^{[f,p]} \left(  \epsilon^2, \, \varepsilon_\rho\right) 
\right\} \\
&~~~~~~
\cdot \left\{
m_0 fp  {\bar{\rho}}^{d -\alpha} 
+ \epsilon \frac{m_2}{2}  \bar{\rho}^{d/2+1}  
(   \omega  fp \bar{\rho}^{1+\frac{d}{2} -\alpha } 
+ \Delta ( fp \bar{\rho}^{1+\frac{d}{2} -\alpha} ) ) 
+ O^{[f,p]} \left(  \epsilon^2, \, \varepsilon_\rho\right)
\right\} \\
&~~~~~~
\cdot \left\{
m_0 p  {\bar{\rho}}^{d -\alpha} 
+ \epsilon \frac{m_2}{2}  \bar{\rho}^{d/2+1}  
(   \omega  p \bar{\rho}^{1+\frac{d}{2} -\alpha } 
+ \Delta ( p \bar{\rho}^{1+\frac{d}{2} -\alpha} ) ) 
+  O^{[p]} \left(  \epsilon^2, \, \varepsilon_\rho\right)
\right\} \\
& = 
\epsilon^{-d/2}
\left\{
\epsilon
 \frac{m_2' m_0^2}{2}
  p^2  {\bar{\rho}}^{ \frac{5}{2}d -2\alpha + 1} 
 (  \Delta ( f^2 p  \bar{\rho}^{1+\frac{d}{2} -2\alpha} ) )  
  +  \epsilon m_2  m_0'  m_0
 p^2 {\bar{\rho}}^{ \frac{5}{2} d - 3\alpha + 1}     
(    f^2  \Delta ( p \bar{\rho}^{1+\frac{d}{2} -\alpha} ) ) 
  \right. \\
 & ~~~
 +  \epsilon
 \frac{m_2' m_0^2}{2}
f^2  p^2  {\bar{\rho}}^{ \frac{5}{2}d -2\alpha + 1} 
 (  \Delta (  p  \bar{\rho}^{1+\frac{d}{2} -2\alpha} ) )  
 + \epsilon m_2  m_0'  m_0
 p^2 {\bar{\rho}}^{ \frac{5}{2} d - 3\alpha + 1}     
(    f  \Delta ( f p \bar{\rho}^{1+\frac{d}{2} -\alpha} ) ) 
  \\
  & ~~~
  - 2    \epsilon
 \frac{m_2' m_0^2}{2}
f  p^2  {\bar{\rho}}^{ \frac{5}{2}d -2\alpha + 1} 
 (  \Delta (  f p  \bar{\rho}^{1+\frac{d}{2} -2\alpha} ) )  
 -  \epsilon m_2  m_0'  m_0
 p^2 {\bar{\rho}}^{ \frac{5}{2} d - 3\alpha + 1}     
(    f  \Delta ( f p \bar{\rho}^{1+\frac{d}{2} -\alpha} ) )  
\\
& ~~~
  \left. -  \epsilon m_2  m_0'  m_0
 p^2 {\bar{\rho}}^{ \frac{5}{2} d - 3\alpha + 1}     
(    f^2  \Delta (  p \bar{\rho}^{1+\frac{d}{2} -\alpha} ) ) 
+O^{[f,p]} \left(  \epsilon^2, \, \varepsilon_\rho
\right)
\right\} \\
& = \epsilon^{-d/2}
\left\{
\epsilon
 \frac{m_2' m_0^2}{2}
  p^2  {\bar{\rho}}^{ \frac{5}{2}d -2\alpha + 1} 
 \left[  
 \Delta ( f^2 p  \bar{\rho}^{1+\frac{d}{2} -2\alpha} )  
 + f^2 \Delta (  p  \bar{\rho}^{1+\frac{d}{2} -2\alpha} )  
 -2 f \Delta ( f p  \bar{\rho}^{1+\frac{d}{2} -2\alpha} )   
 \right]     
 \right. \\
& ~~~
+ \left.
O^{[f,p]} \left(  \epsilon^2, \, \varepsilon_\rho
\right)
\right\}.
\end{align*}
}
Note that the quantity in the square brackets
\begin{equation}\label{eq:simplify-three-deltas-into-|grad f|^2}
\left[ \cdots \right]
=  2 |\nabla f|^2  p  \bar{\rho}^{1+\frac{d}{2} -2\alpha}\,.
\end{equation}
Then, also by the assumption that $\epsilon, \, \frac{\varepsilon_\rho}{ \epsilon} = o(1)$,
we have \rev{with large enough $N$,}
\begin{align}
\E Y_j^2
& = \epsilon^{-d/2+1}
\left\{ 
{m_2' m_0^2}  p^3  {\bar{\rho}}^{ 3 d -4\alpha + 2}   |\nabla f|^2   
+ O^{[f,p]} \left(   \epsilon, \, \frac{\varepsilon_\rho}{ \epsilon} \right)
\right\} 
\label{eq:EYi2-barnuY-1} \\
& \le \bar{\nu}_Y \rev{(x)}
\sim 
\epsilon^{-d/2+1}   p^3  {\bar{\rho}}^{ 3 d -4\alpha + 2}  \rev{(x)  \| \nabla f \|_\infty^2 },
\nonumber
\end{align}
\rev{
where in obtaining the last row we assume that $ \| \nabla f \|_\infty > 0$ (because otherwise the theorem holds trivially),
and use that $p(x) > p_{min}$ for all $x$.
Since $ \bar{\nu}_Y(x) \ge c_{f,p,\alpha} \epsilon^{-d/2+1} $  for $c_{f,p,\alpha} > 0$,
the needed threshold of large $N$ for $\E Y_j^2 \le \bar{\nu}_Y(x)$ is determined by $(\calM,p,f,\alpha)$. 
Meanwhile, under the condition that $\frac{\log N}{N} = o(\epsilon^{d/2+1})$,
with sufficiently large $N$ and the threshold is determined by $(\calM,p,f,\alpha)$,
we have
$ 40 \frac{\log N}{N} <  \frac{ 9 c_{f,p,\alpha} \epsilon^{-d/2+1} }{L_Y^2}\le  \frac{9 \bar{\nu}_Y(x)}{L_Y^2}$,
i.e., $\sqrt{ 40 \frac{\log N}{N} \bar{\nu}_Y(x)} <  \frac{3 \bar{\nu}_Y(x)}{L_Y}$  for any $x \in \calM$.
}
Then, by the classical Bernstein,  w.p. higher than $1-2N^{-10}$,
\[
| \frac{1}{N} \sum_{j=1}^{N}  Y_j  |  \le  \sqrt{  \frac{40 \log N}{N}  \bar{\nu}_Y \rev{(x)}  }  
= O^{[1]} \left( \rev{ \|\nabla f \|_\infty}
p^{3/2} {\bar{\rho}}^{ \frac{3}{2} d -2\alpha + 1}\rev{(x)} 
\epsilon^{-d/4+1/2} 
\sqrt{       \frac{\log N}{N }   }  \right),
\]
and we call the event {the} good event $E_3$.

\rev{Note that in \eqref{eq:EYi2-barnuY-1}, when $ | \nabla f(x) | > 0$,
one can bound the variance of $Y_j$ at $x$ by 
\[ 
 \bar{\nu}_Y(x) \sim 
\epsilon^{-d/2+1}   p^3  {\bar{\rho}}^{ 3 d -4\alpha + 2}  (x) |\nabla f(x)|^2,
\]
and obtain the same large deviation bound where $ \|\nabla f \|_\infty$ is replaced with $ |\nabla f(x)|$,
allowing the large $N$ threshold to depend on $x$ (such that the $O^{[f,p]} \left(   \epsilon, \, \frac{\varepsilon_\rho}{ \epsilon} \right)$ term
in  \eqref{eq:EYi2-barnuY-1} is dominated by $\Theta^{[1]}(1)$ multiplied the first term,
and
  ${ 40 \frac{\log N}{N} } <  \frac{9 \bar{\nu}_Y(x)}{L_Y^2}$ ).
An alternative way to obtain an $x$-uniform threshold of large $N$  is by 
adding $0.1^2$ to $|\nabla f(x)|^2$  in setting $\bar{\nu}_Y(x)$, 
so that the $x$-dependent constant in front of $\epsilon^{-d/2+1}$ is uniformly bounded from below.
This leads to the same variance error bound where $ \|\nabla f \|_\infty$ is replaced with $ |\nabla f(x)|+0.1$.
The above verifies  Remark \ref{rk:nablaf(x)=0}.
}

Back to \eqref{eq:def-Yj-3},
with \eqref{eq:EG-sumG-O(1)^2-3},
 we have that under good events $E_2$ and $E_3$,
\begin{align*}
& \frac{1}{ \epsilon \tilde{m} \hat{\rho}^2}
\left| \frac{ \frac{1}{N} \sum_{j=1}^{N}   F_j }
{ \frac{1}{N} \sum_{j=1}^{N}  G_j }
 - \frac{ \E F }{ \E G } \right|
=  \frac{1}{ \epsilon \tilde{m} \hat{\rho}^2}
\frac{ | \frac{1}{N} \sum_{j=1}^{N}   Y_j | }
{ \E G \cdot  \frac{1}{N} \sum_{j=1}^{N}  G_j  }
{=}
\frac{ O^{[1]} \left(  \rev{ \|\nabla f \|_\infty}
 p^{3/2} {\bar{\rho}}^{ \frac{3}{2} d -2\alpha + 1}
\epsilon^{-d/4+1/2} 
\sqrt{       \frac{\log N}{N }   }  \right) }
{\epsilon \bar{\rho}^2  (m_0 p \bar{\rho}^{d-\alpha})^2 } \\
& ~~~
= O^{[1]} \left(  \rev{ \|\nabla f \|_\infty}
p^{-1/2} {\bar{\rho}}^{ - \frac{d}{2}  - 1}
\epsilon^{-d/4-1/2} 
\sqrt{       \frac{\log N}{N }   }  \right),
\end{align*}
\rev{where the location $x$ is omitted in the notation.}
By that 
$\bar{\rho}=p^{-1/d}$,
this proves \eqref{eq:variance-error-tildeLN-goal-3}.
 \end{proof}
 %
 %

%
%
%
\begin{proof}[Proof of Theorem \ref{thm:limit-pointwise-un-3}]
By the definition of $L^{(\alpha)}_{ un} f$ and that of $F_j$, $G_j$ in \eqref{eq:def-Fj-Gj},
\[
L^{(\alpha)}_{ un} f(x) = 
\frac{1}{N}
\sum_{j=1}^{N}  
\frac{1}{ {\epsilon \frac{m_2}{2} } }
\frac{1}{\hat{\rho}(x)^\alpha}
(F_j(x) - f(x) G_j(x))
=: 
\frac{1}{N}
\sum_{j=1}^{N}   H_j.
\]
We have computed $\E F - f(x) \E G$ in \eqref{eq:EF-fEG-3},
and \eqref{eq:diff-any-power} gives that $ \sup_{x \in {\calM}} | \hat{\rho}(x)^{-\alpha}  - \hat{\rho}(x)^{-\alpha} | = O^{[p]}(\varepsilon_\rho)$.
Then,
\begin{align*}
\E H_j 
& =  \frac{1}{ {\epsilon \frac{m_2}{2} } } \hat{\rho}(x)^{-\alpha} (\E F(x) - f(x) \E G(x) )  \\
&=
\frac{1}{ {\epsilon \frac{m_2}{2} } }
\hat{\rho}(x)^{-\alpha} 
\left( 
\epsilon \frac{m_2}{2}  
p \bar{\rho}^{d+2 -\alpha}   {\calL}^{(\alpha)}  f 
+ O^{[f,p]}(  \epsilon^2, \, \varepsilon_\rho)
\right) \\
& = 
p \bar{\rho}^{d+2 - 2\alpha}   {\calL}^{(\alpha)}  f (x)
+ O^{[f,p]}(  \epsilon, \, \frac{\varepsilon_\rho}{\epsilon}).
\end{align*}
By $\bar{\rho} = p^{-1/d}$, this proves that 
\begin{equation}\label{eq:bias-error-un}
\E L^{(\alpha)}_{ un} f(x) 
= 
p^{ \frac{ 2\alpha - 2}{d}}   {\calL}^{(\alpha)}  f (x)
+ O^{[f,p]}(  \epsilon, \, \frac{\varepsilon_\rho}{\epsilon}),
\end{equation}
where the constant in $O^{[f,p]}(\cdot)$ is uniform for all $x \in {\calM}$.

To analyze the variance,
first note the boundedness of $|H_j|$ as
\[
|H_j | \le L_H =\Theta^{[f,p]} (\epsilon^{-d/2-1}),
\]
which follows by the boundedness of $G_j$, $F_j$ in \eqref{eq:boundedness-Gj-Fj} and the uniform boundedness of $\hat{\rho}$.
For the variance of $H_j$, we have
\[
\E H_j^2
= ( \frac{1}{ {\epsilon \frac{m_2}{2} } } \hat{\rho}(x)^{-\alpha}  )^2
( \E F_j^2 + f(x)^2 \E G_j^2 - 2f(x) \E F_j G_j),
\]
and we have computed $\E F^2$, $\E G^2$ and $\E F G$ in the proof of Theorem \ref{thm:limit-pointwise-rw-3}.
Specifically, with notation the same as therein, we have
\begin{align}
&  \E F^2 =
\epsilon^{-d/2}
\left\{  m_0' f^2 p {\bar{\rho}}^{d - 2\alpha}
+ \epsilon   \frac{m_2'}{2} \bar{\rho}^{1+\frac{d}{2}} (   \omega  f^2 p \bar{\rho}^{1+\frac{d}{2}-2 \alpha} 
+ \Delta ( f^2 p  \bar{\rho}^{1+\frac{d}{2} -2\alpha} ) ) 
+  O^{[f,p]} \left(   \epsilon^2, \, \varepsilon_\rho \right) \right\}, 
\nonumber  \\
& \E G^2 =
\epsilon^{-d/2}
\left\{  m_0' p {\bar{\rho}}^{d - 2\alpha}
+ \epsilon   \frac{m_2'}{2} \bar{\rho}^{1+\frac{d}{2}} (   \omega  p \bar{\rho}^{1+\frac{d}{2}-2 \alpha} 
+ \Delta ( p  \bar{\rho}^{1+\frac{d}{2} -2\alpha} ) ) 
+   O^{[p]} \left(  \epsilon^2, \, \varepsilon_\rho \right) \right\},
  \nonumber\\
&  \E FG 
 =
\epsilon^{-d/2}
\left\{  m_0' f p {\bar{\rho}}^{d - 2\alpha}
+ \epsilon   \frac{m_2'}{2} \bar{\rho}^{1+\frac{d}{2}} (   \omega  f p \bar{\rho}^{1+\frac{d}{2}-2 \alpha} 
+ \Delta ( f p  \bar{\rho}^{1+\frac{d}{2} -2\alpha} ) ) 
+   O^{[f,p]} \left(  \epsilon^2, \, \varepsilon_\rho \right) \right\}.
\nonumber 
\end{align}
Also, 
by \eqref{eq:simplify-three-deltas-into-|grad f|^2}, where the square brackets denote the same quantity as before, we have
\begin{align*}
& \E F^2 + f^2 \E G^2 - 2f \E F G
= 
\epsilon^{-d/2}
\epsilon   \frac{m_2'}{2} \bar{\rho}^{1+\frac{d}{2}} [\cdots]  \\
&=
\epsilon^{-d/2}
\left\{
\epsilon  m_2' 
 |\nabla f|^2  p  \bar{\rho}^{2+ d -2\alpha}
 + O^{[f,p]} \left(  \epsilon^2, \, \varepsilon_\rho \right)
 \right\}.
\end{align*}
Then, also by \eqref{eq:diff-any-power},  we have
\begin{align*}
& \E H_j^2 = 
\frac{4}{ { m_2^2 } }  \epsilon^{-1 - d/2} 
\hat{\rho}^{-2\alpha}(x)
\left\{ m_2'  |\nabla f|^2  p  \bar{\rho}^{2+ d -2\alpha}(x)
 + O^{[f,p]} \left(  \epsilon, \,  \frac{\varepsilon_\rho}{\epsilon} \right)
 \right\} \\
 & =
 \frac{4}{ { m_2^2 } }  \epsilon^{-1 - d/2} 
\left\{ 
 m_2'  |\nabla f|^2  p  \bar{\rho}^{2+ d - 4\alpha}(x)
 + O^{[f,p]} \left(  \epsilon, \,  \frac{\varepsilon_\rho}{\epsilon} \right)
 \right\}
 \quad 
 \text{(by that $\hat{\rho}^{-2\alpha}(x) = \bar{\rho}^{-2\alpha}(x) + O^{[p]}(\varepsilon_\rho)$)} \\
 & \le \bar{\nu}_H  \rev{(x)}
 = \Theta^{[1]}(  \epsilon^{-1 - d/2} \rev{ \| \nabla f \|_\infty^2  }
 p^{ \frac{ 4\alpha-2}{d}}(x)), \rev{~\text{with large $N$, the threshold depending on $(\calM,p,f,\alpha)$.}}
\end{align*}
\rev{
In obtaining the last row, 
we assumed $\| \nabla f\|_\infty >0$ (when $\| \nabla f\|_\infty =0$, the theorem holds trivially),
and used that   
$O^{[f,p]} \left(  \epsilon, \,  \frac{\varepsilon_\rho}{\epsilon} \right) =o(1)$,
 $\bar{\rho} = p^{-1/d}$, 
 and  that $p$ is uniformly bounded from below.
Then same as in the proof of Theorem \ref{thm:limit-pointwise-rw-3},
 the threshold of large $N$ to achieve the sub-Gaussian tail in Bernstein inequality is determined by $(\calM,p,f,\alpha)$.
 }
As a result, when $N$ is large enough,
we have that w.p. higher than  $1-2N^{-10}$, 
\[
\left| \frac{1}{N} \sum_{j=1}^{N}  H_j - \E H_j  \right| 
 \le  \sqrt{  \frac{40 \log N}{N}  \bar{\nu}_H   \rev{(x)} }  
= O^{[1]} \left(  
 \rev{ \|\nabla f \|_\infty    }
 p^{ \frac{ 2\alpha-1}{d}}(x)
\sqrt{   \frac{\log N}{N  \epsilon^{ d/2+1} }   }  \right).
\]
\rev{To replace $ \|\nabla f \|_\infty  $ with $ |\nabla f(x)| $ when strictly positive, or with +0.1, as in Remark \ref{rk:nablaf(x)=0},   
the same argument by re-defining $\bar{\nu}_H(x) $ similarly as in the proof of  Theorem \ref{thm:limit-pointwise-rw-3} applies. 
 }
Combined with \eqref{eq:bias-error-un}, this finishes the proof.
\end{proof}
%
%

 %
%
\begin{proof}[Proof of Theorem \ref{thm:limit-weak-un}]
\rev{Suppose $\varphi \neq 0$ and $\| \nabla f \|_\infty> 0$, otherwise the theorem trivially holds.}
By definition \eqref{eq:def-Lalapha-un}, we have that
\[
\langle \varphi, L^{(\alpha)}_{ un} f \rangle_p  = 
\frac{1}{N} \sum_{j=1}^N  H_j,
\]
where
\[
H_j := 
\frac{2 \epsilon^{-1}}{ {m_2 } } 
\int_{\calM}
\hat{K}(x, x_j)
 (f(x_j) - f(x)) \varphi(x) p(x) dV(x),
\]
and $\hat{K}(x,y)$ is defined as in \eqref{eq:def-hatK}.
We define
\[
B( g, f) := 
\frac{2 \epsilon^{-1}}{ {m_2 } } 
\int_{\calM} \int_{\calM}
\hat{K}(x, y)
 (f(y) - f(x)) g(x) p(x) p(y) dV(x) dV(y).
\]
By $\hat{K}(x,y)= \hat{K}(y,x)$, $B(g,f) = B(f,g)$
i.e., $B(g,f)$ is a symmetric bilinear form.
Meanwhile,
\[
B(f,f) = - {\calE}^{(\alpha)}(f,f),
\]
where, by  Proposition \ref{prop:hatE(f,f)}, 
${\calE}^{(\alpha)} (f,f) = - \langle f, \Delta_{p_\alpha} f \rangle_{p_\alpha} + O^{[f,p]}( \epsilon, \, \varepsilon_\rho)$.
Thus,
\begin{align*}
\E H_j 
& =  B(\varphi, f)
= \frac{1}{4} (  B(\varphi+f, \varphi+f) -   B(\varphi-f, \varphi-f)) \\
&= \frac{1}{4} (  - {\calE}^{(\alpha)}(\varphi+f, \varphi+f) +  {\calE}^{(\alpha)}(\varphi-f, \varphi-f)) \\
& =\frac{1}{4}\left( 
 \langle \varphi+f, \Delta_{p_\alpha} ( \varphi+f) \rangle_{p_\alpha}
 -  \langle \varphi- f, \Delta_{p_\alpha} ( \varphi-f) \rangle_{p_\alpha}
 + O^{ [\varphi,f,p] }( \epsilon, \, \varepsilon_\rho)
 \right) \\
 & =
  \langle \varphi, \Delta_{p_\alpha}  f \rangle_{p_\alpha}
 + O^{ [\varphi,f,p] }( \epsilon, \, \varepsilon_\rho).
\end{align*}

To analyze the variance, we compute the boundedness and variance of $H_j$.
To avoid obtaining $\frac{\varepsilon_\rho}{\epsilon}$, we cannot directly apply 
Lemma \ref{lemma:right-operator-3} 
{and}
Lemma \ref{lemma:replace-hatrhoalpha-GR-hatrho}
as in the proof of Theorem \ref{thm:limit-pointwise-un-3}.
By Cauchy--Schwartz inequality and that $ \hat{K}(x,y) \ge 0$, we have
\[
|H_j|
\le 
\frac{2 \epsilon^{-1}}{ {m_2 } } 
\left( \int_{\calM} K(x, x_j)
 (f(x_j) - f(x))^2  p(x) dV(x)  \right)^{1/2}
\left( \int_{\calM} K(x, x_j)
 \varphi(x)^2 p(x) dV(x) \right)^{1/2}.
\]
We define $\textcircled{1}(y)$ and $\textcircled{2}(y)$ as below and claim the following:  For any $y \in {\calM}$, 
\begin{equation}\label{eq:intKphi2p=O(1)}
\textcircled{1}(y) := \int_{\calM} \hat{K}(x, y)  \varphi(x)^2 p(x) dV(x) 
\le c_1 \|\varphi  p^{ \alpha / d }\|_\infty^2  ,
\quad c_1  = \Theta^{[1]}(  1 ),
\end{equation}
\begin{equation}\label{eq:intKdifff2p=O(1)}
\textcircled{2}(y) : = \int_{\calM} \hat{K}(x, y)  ( f(y)-f(x))^2 p(x) dV(x) = O^{ [ f, p] }( \epsilon).
\end{equation}
If true,  then we have
\[
|H_j| 
{=} 
\epsilon^{-1} O^{p,f,\varphi}(\sqrt{\epsilon}) = O^{[p,f,\varphi]}(\epsilon^{-1/2}),
\]
and at the same time, using the upper bound  \eqref{eq:intKphi2p=O(1)}, we have
\begin{align*}
\E H_j^2 
&\le 
\left( \frac{ 4 \epsilon^{-1}}{ {m_2 } }  \right)
c_1 \|\varphi  p^{ \alpha / d }\|_\infty^2
\left(
\frac{1 }{  \epsilon m_2  }  
 \int_{\calM} \int_{\calM} \hat{K}(x, y)
 (f( y) - f(x))^2  p(x) dV(x)  p(y) dV(y) \right) \\
 & = 
 \epsilon^{-1} 
c_1' \|\varphi  p^{ \alpha / d }\|_\infty^2 {\calE}^{(\alpha)}(f,f),
\quad
c_1'   = \Theta^{[1]}( 1).
\end{align*}
Again, 
${\calE}^{(\alpha)}(f,f) =- \langle f, \Delta_{p_\alpha} f \rangle_{p_\alpha} + O^{[f,p]}( \epsilon, \, \varepsilon_\rho)$,
where
$- \langle f, \Delta_{p_\alpha} f \rangle_{p_\alpha} = \int p_\alpha |\nabla f|^2$ \rev{$>0$}.
We then have that 
\[
\text{Var}(H_j) \le 
\E H_j^2 \le 
\bar{\nu}_H 
= \Theta^{[1]} \left( \epsilon^{-1} 
 \|\varphi  p^{ \alpha / d }\|_\infty^2   \int p_\alpha |\nabla f|^2 
 \right).
\]
Thus, when $N$ is large enough, w.p. higher than $1-2N^{-10}$, we have
\[
\left| \frac{1}{N} \sum_{j=1}^{N}  H_j - \E H_j  \right|  
\le  \sqrt{  \frac{40 \log N}{N}  \bar{\nu}_H  }  
= O^{[1]} \left( 
 \|\varphi \|_\infty \| p^{ \alpha / d }\|_\infty
\sqrt{       \frac{\log N}{   N  \epsilon}   \int p_\alpha |\nabla f|^2 }  \right).
\]
It remains to show \eqref{eq:intKphi2p=O(1)}\eqref{eq:intKdifff2p=O(1)} to finish the proof of the theorem.
\vspace{5pt}

\underline{Proof of \eqref{eq:intKphi2p=O(1)}}: 
By definition,
\[
\textcircled{1}(y)
=  \frac{1}{\hat{\rho}( y)^\alpha}
\epsilon^{-\frac{d}{2}}
\int_{\calM}
 k_0 \left(  \frac{\| x - y \|^2}{ \epsilon \hat{\rho}(x)  \hat{\rho}( y ) } \right)
\frac{  \varphi(x)^2  p(x) }{ \hat{\rho}(x)^\alpha } dV(x),
\]
where by that $\sup_{x \in {\calM}}\frac{ | \hat{\rho}(x) - \bar{\rho}(x) | }{ | \bar{\rho}(x)|} < \varepsilon_\rho < 0.1$,
we have
$\hat{\rho}(x) \hat{\rho}(y) \le 1.1^2 \bar{\rho}(x) \bar{\rho}(y)$,
and then by Assumption \eqref{assump:h-selftune}(C2) we have
\[
k_0 \left (  \frac{\| x - y \|^2}{ \epsilon \hat{\rho}(x)  \hat{\rho}( y ) } \right)
\le 
a_0 e^{-a  \frac{\| x - y \|^2}{ \epsilon \hat{\rho}(x)  \hat{\rho}( y ) }}
\le 
a_0 e^{- \frac{a}{1.1^2}  \frac{\| x - y \|^2}{ \epsilon \bar{\rho}(x)  \bar{\rho}( y ) }}
= \bar{k}_1 \left(  \frac{\| x - y \|^2}{ \epsilon \bar{\rho}(x)  \bar{\rho}( y ) } \right),
\]
where $\bar{k}_1 (r) : = a_0 e^{- \frac{a}{1.1^2} r}$ and satisfies Assumption \ref{assump:h-selftune}.
We introduce $G_{\epsilon}^{(\rho)}[ h]$ 
when in the definition \eqref{eq:def-G-R-rho-epsilon}
 the kernel function $k_0$ is replaced with some $h$ that satisfies Assumption \ref{assump:h-selftune}.
That is, $G_{\epsilon}^{(\rho)} = G_{\epsilon}^{(\rho)}[ k_0]$,
and the notation $[ k_0]$  is to declare the kernel function being used.
To proceed, by that 
$\hat{\rho}(x)^{-\alpha} \le \max\{ 0.9^\alpha, 1.1^\alpha \} \bar{\rho}(x)^{-\alpha} := c_{2} \bar{\rho}(x)^{-\alpha}$,
we have that for any $x \in {\calM}$,
\[
\textcircled{1}(y) 
\le  \frac{c_2^2}{\bar{\rho}( y)^\alpha}
\epsilon^{-\frac{d}{2}}
\int_{\calM}
\bar{k}_1 \left(  \frac{\| x - y \|^2}{ \epsilon \bar{\rho}(x)  \bar{\rho}( y ) } \right)
\frac{  \varphi(x)^2  p(x) }{ \bar{\rho}(x)^\alpha } dV(x)
= \frac{c_2^2}{\bar{\rho}( y)^\alpha} \bar{\rho}(y)^{d/2}
G_{\epsilon \bar{\rho}(y)}^{(\bar{\rho})}[ \bar{k}_1] (\frac{\varphi^2 p}{ \bar{\rho}^\alpha }) (y).
\]
By Lemma \ref{lemma:right-operator-3}, 
\[
G_{\epsilon \bar{\rho}(y)}^{(\bar{\rho})}[ \bar{k}_1] (\frac{\varphi^2 p}{ \bar{\rho}^\alpha }) (y)
= m_0[  \bar{k}_1] \bar{\rho}^{d/2} (\frac{\varphi^2 p}{ \bar{\rho}^\alpha })(y) 
+ \bar{\rho}(y) O^{[ p,f,\varphi ]} (\epsilon),
\]
and then
\[
\textcircled{1}(y) 
\le 
c_2^2   m_0[  \bar{k}_1]
( p  \bar{\rho}^{d - 2 \alpha}  \varphi^2 ) (y) 
+ O^{[p,f,\varphi]} (\epsilon) 
= \Theta^{[1]}( ( p  \bar{\rho}^{d - 2 \alpha}  \varphi^2 ) (y) ).
\]
By that  $\bar{\rho}=p^{1/d}$,
we have shown that 
$ \textcircled{1}(y)  \le c_1 p(y)^{ 2\alpha /d } \varphi(y)^2$,
where $c_1   = \Theta^{[1]} (1)$,
and this proves  \eqref{eq:intKphi2p=O(1)}.
\vspace{5pt}

\underline{Proof of \eqref{eq:intKdifff2p=O(1)}}:
Similarly, we have
\begin{align*}
& \textcircled{2}(y) 
\le
\frac{ c_2^2 }{\bar{\rho}( y)^\alpha}
 \epsilon^{-\frac{d}{2}} \int_{\calM} \bar{k}_1 \left(  \frac{\| x - y \|^2}{ \epsilon \bar{\rho}(x)  \bar{\rho}( y ) } \right)
\frac{ ( f(y)-f(x))^2 }{ \bar{\rho}(x)^\alpha }
 p(x) dV(x) \\
& = 
\frac{ c_2^2 }{\bar{\rho}( y)^\alpha}
\bar{\rho}(y)^{d/2}
G_{\epsilon \bar{\rho}(y)}^{(\bar{\rho})}[ \bar{k}_1] ( g) (y),
\quad
g(x) := (f(y) - f(x))^2 ( \frac{  p}{ \bar{\rho}^\alpha })(x)\,.
\end{align*}
{Note that} 
$g \in C^{\infty}({\calM})$ and $g(y) = 0$.
By Lemma \ref{lemma:right-operator-3}, we have 
\[
G_{\epsilon \bar{\rho}(y)}^{(\bar{\rho})}[ \bar{k}_1] ( g) (y)
= 
\epsilon \bar{\rho}(y) 
\frac{m_2 [ \bar{k}_1]}{2}    \Delta g (y)
+ \bar{\rho}(y)^2 O^{[f,p]} (\epsilon^2),
\]
and then
\begin{align*}
& \textcircled{2}(y) 
\le
 \frac{ c_2^2 }{\bar{\rho}( y)^\alpha}
\bar{\rho}(y)^{d/2}
\left( 
\epsilon \bar{\rho}(y) 
\frac{m_2 [ \bar{k}_1]}{2}    \Delta g (y)
+ \bar{\rho}(y)^2 O^{[f,p]} (\epsilon^2)
\right) \\
& = 
 c_2^2  \frac{m_2 [ \bar{k}_1]}{2}   \epsilon
\bar{\rho}(y)^{d/2 - \alpha +1}
   \Delta g (y)
+  O^{[f,p]} (\epsilon^2)
= O^{[f,p]}(\epsilon ),
\end{align*}
which proves \eqref{eq:intKdifff2p=O(1)}.
\end{proof}
%
%

%
%
%
\begin{proof}[Proof of Theorem \ref{thm:limit-pointwise-rw-fix-epsilon}]
The proof combines the approach in the proof of Theorem \ref{thm:limit-pointwise-rw-3}
 and the computation in that of Theorem \ref{thm:limit-form-fixeps}.
Define 
 \[
 F_j := \epsilon^{-d/2} k_0 \left(  \frac{\| x - x_j\|^2}{  \epsilon   } \right)
\frac{f(x_{j})}{\hat{p}(x_j)^{\beta} },
\quad 
G_j : = \epsilon^{-d/2} k_0 \left(  \frac{\| x - x_j\|^2}{  \epsilon   } \right)
\frac{ 1 }{\hat{p}(x_j)^{\beta} },
 \]
 then we have
 \[
 \E F = G_{\epsilon} ( \frac{fp}{ \hat{p}^\beta}) 
 = \int_\calM \epsilon^{-d/2} k_0 \left(  \frac{\| x - y\|^2}{  \epsilon   } \right)
f( y ) \hat{p}(y)^{-\beta}  p(y) dV(y).
 \]
 By \eqref{eq:hatp-p-beta-bound} and the constant $c_\beta$ defined as therein,
\rev{ again defining $q = p^{1-\beta}$, we have}
 \begin{align*}
&  \left| G_{\epsilon} ( \frac{fp}{ \hat{p}^\beta}) -   G_{\epsilon} ( f {p}^{1-\beta}) \right|
 \le 
 \| f \|_\infty
\int \epsilon^{-d/2} k_0 \left(  \frac{\| x - y\|^2}{  \epsilon   } \right) p(y)
| \hat{p}(y)^{-\beta}  - {p}(y)^{-\beta} | dV(y) \\
& \le 
c_\beta \rev{ | \beta| } \varepsilon_p \| f \|_\infty
\int \epsilon^{-d/2} k_0 \left(  \frac{\| x - y\|^2}{  \epsilon   } \right) 
 p(y)^{1-\beta}
dV(y)  
= O^{[f, p, \beta ]}( \beta \varepsilon_p),
 \end{align*}
 where 
 \rev{
 we apply Lemma \ref{lemma:h-integral-diffusionmap} to obtain that 
 $\int \epsilon^{-d/2} k_0 \left(  \frac{\| x - y\|^2}{  \epsilon   } \right) 
 q(y) dV(y)  \le \|q\|_\infty O(1) $, and absorb 
 the constants $\|q\|_\infty$, $\|f\|_\infty$  and $c_\beta$ in to the notation $O^{[f,p,\beta]}(\cdot)$.
In the rest of the proof, we omit the dependence on $\beta$ in the superscript and write it as $O^{[f,p]}( \beta \varepsilon_p)$,
while we keep $\beta$ in $(\cdot)$ to indicate that the term vanishes when $\beta = 0$. 
 }
Then, using Lemma \ref{lemma:h-integral-diffusionmap} to expand $G_{\epsilon} ( fp^{1-\beta} )$, we have that
\[
\E F 
=  G_{\epsilon} ( fp^{1-\beta} )+ O^{[f, p ]}( \beta \varepsilon_p)
= m_0 fp^{1-\beta} + \epsilon \frac{m_2}{2} ( \omega  fp^{1-\beta} + \Delta ( fp^{1-\beta} ) ) + O^{[ f, p] } (\epsilon^2, \, \beta \varepsilon_p).
\]
Taking $f=1$ then gives
\[
\E G = m_0 p^{1-\beta} + \epsilon \frac{m_2}{2} ( \omega  p^{1-\beta} + \Delta ( p^{1-\beta} ) ) + O^{ [p ]} (\epsilon^2, \, \beta \varepsilon_p) 
= m_0 p^{1-\beta} + O^{[p]} (\epsilon, \, \beta \varepsilon_p).
\]
We can then compute and bound the bias error as
 \begin{align*}
&  \frac{1}{\epsilon \tilde{m}} \frac{\E F - f(x) \E G}{ \E G}
 = 
  \frac{1}{\epsilon \tilde{m}} \frac{  \epsilon \frac{m_2}{2} (  \Delta ( fp^{1-\beta} ) - f \Delta ( p^{1-\beta} )) 
  + O^{ [f, p] } (\epsilon^2, \, \beta \varepsilon_p)}{m_0 p^{1-\beta} + O^{[p]} (\epsilon, \, \beta \varepsilon_p)  } \\
&
= \Delta f + 2 \frac{\nabla p^{1-\beta}}{p^{1-\beta}} \cdot \nabla f + O^{ [f, p] } (\epsilon, \, \beta \frac{ \varepsilon_p}{\epsilon} ),
 \end{align*}
 \rev{
 which, similarly as in \eqref{eq:1overEG-1},
 holds when $N$ exceeds a threshold depending on $(\calM, p, \beta)$.}
 The variance analysis follows a similar computation as before, 
 specifically the computation of the quantities of 
 $\E F^2$, $\E G^2$ and  $\E (F G)$.
First, observe that 
 \[
 \E G^2 =  \epsilon^{-d/2} \{ m_0[k_0^2] p^{1- 2 \beta}  + O^{[p]}( \epsilon, \, \beta \varepsilon_p) \},
 \]
and then, \rev{using that $p^{1- 2\beta }(x) \ge \Theta^{[p,\beta]}(1) > 0$ for all $x$,}
one verifies that 
w.p. higher than  $ 1- 2N^{-10}$,
 \[
 \frac{1}{N} \sum_j G_j = \E G + O^{[p]} \left( \sqrt{ \frac{\log N }{ N \epsilon^{d/2} } } \right).
 \]
 Define $ Y_j := F_j \E G - G_j \E F$, then $\E Y = 0$.
Following the same method as before, one verifies that, with $m_2' := m_2 [k_0^2]$,
\[
\E Y^2  = \E F^2 (\E G)^2 
+ E G^2 (\E F)^2 - 2 \E (FG) \E F \E G
= \epsilon^{-d/2 + 1} \left\{  m_2' m_0^2  p^{3 -4 \beta}  |\nabla f|^2
+ O^{[f,p]}( \epsilon, \, \beta \frac{\varepsilon_p}{ \epsilon}) \right\} .
\]
\rev{Similarly as in the proof of Theorem \ref{thm:limit-pointwise-rw-3}, }
this gives that \rev{(assuming $ \|\nabla f\|_\infty >0$ otherwise the theorem holds trivially) when $N$ exceeds a threshold determined by $(\calM, p, f, \beta)$,}
w.p. higher than  $ 1- 2 N^{-10}$,
\[
\left| \frac{1}{N} \sum_{j=1}^{N}  Y_j  \right|  
= O \left( 
\rev{ \|\nabla f\|_\infty  }
 p^{3/2 -2 \beta}    
\epsilon^{-d/4 + 1/2}   \sqrt{       \frac{\log N}{N }   }  \right).
\]
\rev{
One can also replace $ \|\nabla f \|_\infty  $ with $ |\nabla f(x)| $ when strictly positive, or with +0.1, as in Remark \ref{rk:nablaf(x)=0}.  
 }

Putting together, we have that 
\begin{align*}
\frac{1}{ \epsilon \tilde{m}}
\left| \frac{ \frac{1}{N} \sum_{j=1}^{N}   F_j }
{ \frac{1}{N} \sum_{j=1}^{N}  G_j }
 - \frac{ \E F }{ \E G } \right|
&=  \frac{1}{ \epsilon \tilde{m} }
\frac{ | \frac{1}{N} \sum_{j=1}^{N}   Y_j | }
{ \E G \cdot  \frac{1}{N} \sum_{j=1}^{N}  G_j  }
\le
\frac{ O^{[1]} \left(  
\rev{ \|\nabla f\|_\infty  }
 p^{3/2 -2 \beta}    \epsilon^{-d/4 + 1/2}   \sqrt{       \frac{\log N}{N }   }
  \right) }
{ \epsilon ( m_0 p^{1-\beta} )^2 } \\
&
= O^{[1]} \left(  \rev{ \|\nabla f\|_\infty  }
 p^{-1/2} 
\epsilon^{-d/4-1/2} 
\sqrt{       \frac{\log N}{N }   }  \right).
\end{align*}
Combining the bias and variance error bounds proves the theorem.
\end{proof}

\section{Discussion}

Apart from what has been mentioned in the text,
the following {lists} a few possible future directions.
First, we use a stand-alone $Y$ to estimate the bandwidth function $\hat{\rho}$ for theoretical convenience.
Extending the result to the case where  $\hat{\rho}$ is computed from $X$ itself can be of 
{both theoretical and practical} interest,
especially when number of data samples are not large. 
Second,
one can continue to derive the spectral convergence, 
namely the convergence of eigenvalues and eigenvectors of the self-tuned graph Laplacian matrix to the eigenvalues and eigenfunctions of the associated limiting operators. For the purpose of statistical inference, it would be important to provide a convergence rate.
Our graph Dirichlet form convergence rate is better than the operator point-wise convergence rate by a factor of $\epsilon^{-1/2}$,
and since the Dirichlet form convergence largely implies the spectral convergence in the $L^2$ norm \cite{burago2013graph},
this suggests that the spectral convergence rate may also be better than the ponitwise convergence rate for the graph Laplacian operator  in a proper sense.
This theoretical speculation is supported by our empirical results. A uniform spectral convergence would also be important for various practical applications.
At last,  
the random-walk graph Laplacian in our experiments sometimes shows a better performance
{compared with} the unnormalized graph Laplacian,
especially in terms of eigenvector convergence. 
A theoretical justification then is needed, {which is}
possibly similar to that in \cite{von2008consistency}, 
and will be based on the spectral convergence result if can be established.

\section*{Acknowledgement}
The project was initiated as a {\it DoMath} Project titled ``Local affinity construction for dimension reduction methods'' 
for undergraduate summer research, 
and the authors thank the Duke Mathematics Department for organizing and hosting  the program. 
In 2018 summer,
Tyler Lian, Inchan Hwang, Joseph Saldutti and Ajay Dheeraj participated the project, 
 and contributed to initial experiments 
of the adaptive bandwidth kernel and the analysis of the proposed self-tuned kernel Laplacian when the bandwidth function is known. 
The authors thank Dr. Didong Li, who served as a graduate student mentor of the 2018 {\it DoMath} project, 
for guiding the undergraduate team work 
as well as  helpful discussion on NNDE and estimated bandwidth function.

\small
\bibliographystyle{plain}
\bibliography{kernel}

\normalsize

\appendix

\setcounter{figure}{0} \renewcommand{\thefigure}{A.\arabic{figure}}
\setcounter{table}{0} \renewcommand{\thetable}{A.\arabic{table}}
\setcounter{equation}{0} \renewcommand{\theequation}{A.\arabic{equation}}
\setcounter{remark}{0} \renewcommand{\theremark}{A.\arabic{remark}}

\section{Technical Lemmas of Differential Geometry}\label{app:diffgeo}

\subsection{Local Charting on ${\calM}$}

The following lemma is about manifold local charting, where we have metric and volume comparisons between the manifold and the 
ambient Euclidean space $\R^D$.  
\begin{lemma}[Lemmas 6 and 7 in \cite{coifman2006diffusion}]
\label{lemma:M-delta0}
Suppose ${\calM}$ is a $d$-dimensional $C^{3}$, boundaryless (thus {closed}) manifold
that is 
isometrically embedded in $\mathbb{R}^{D}$.
Then there exists some $\delta_0({\calM}) > 0$ such that 
for any $\delta < \delta_0$ and any $x \in {\calM}$,

(i) ${\calM} \cap B_{\delta}(x)$ is isomorphic to a ball in $\R^d$.

(ii) On the local chart at each $x$, let $\phi_x$ be the orthogonal projection to the tangent plane $T_x {\calM}$
embedded as an affine subspace of $\mathbb{R}^D$,
and call $u(y) := \phi_x(y)$ the tangent coordinate of $y$, then 
\begin{equation}\label{eq:local-metric-2}
0.9 \| y-x \|_{\R^D} < \| u(y)  \|_{\R^d} <  1.1 \| y- x \|_{\R^D},
\quad 
0.9 < \left| \det\left(\frac{dy}{du}\right) \right| < 1.1,
\quad \forall y  \in {\calM} \cap B_{\delta}(x).
\end{equation}

(iii)  Let $d_{\calM}$ denote the manifold geodesic distance, then 
\begin{equation}\label{eq:metric-geo-2}
 \| x-y  \|_{\R^D} \le d_{\calM}(x, y) \le 1.1 \| x-y  \|_{\R^D}, \quad \forall y\in B_{\delta}(x) \cap {\calM}.
\end{equation}
\end{lemma}

\begin{proof}
At every point $x$, (i) holds when $\delta < \delta_{x,1}$ for some $\delta_{x,1}>0$,
and then the local chart can be defined
where the  normal coordinates $s$ ($\exp_x (s) = y$)
and the tangent coordinates $u$  match up to $O( \| u\|^3)$
(Lemma 6 \cite{coifman2006diffusion}),
the squared metric of $\|y - x\|_{\R^D}^2$ and $\|u\|^2$ match up to $O( \|u\|^4)$,
and the Jacobian's  match via
\begin{equation}\label{eq:local-vol-1}
\left| \det\left(\frac{dy}{du}\right) \right| = 1 + b_{x}^{(v)}(u) + c_{x}^{(v)}(u) + O(\| u \|^4),
\end{equation}
where $b_x^{(v)}$ ($c_x^{(v)}$) is a homogeneous polynomial of degree 2 (3) of the variable $u = (u_1, \cdots, u_d)$
(Lemma 7 \cite{coifman2006diffusion}).
Thus \eqref{eq:local-metric-2}\eqref{eq:metric-geo-2}
hold on ${\calM} \cap B_{\delta}(x)$ when $\delta < \delta_{x,2}$ for some $0<\delta_{x,2} \le \delta_{x,1} $.
 The $\min_{x \in {\calM}} \delta_{x,2}$ exists due to the smoothness and compactness of ${\calM}$,
 and the minimum can be used as $\delta_0( {\calM})$.
\end{proof}

\subsection{Covering Number of ${\calM}$}

Introduce the definitions:

\begin{definition}\label{def:covering-number}
Let $(X,d)$ be a metric space and $Y\subset X$. 
Let  $\epsilon>0$, then $P \subset X$ is called a $\epsilon$-net of $Y$
 if $\forall x \in Y$, $\exists x_0\in P$, s.t. $d(x,x_0)  \le \epsilon$. 
The covering number of $Y$, denoted by $\mathcal{N}(Y,d,\epsilon)$, is defined to be the smallest cardinality of an $\epsilon$-net of $Y$. 
\end{definition}

\begin{definition}\label{def:packing-number}
	Let $(X,d)$ be a metric space. Let  $\epsilon>0$, 
	then $ P \subset X$ is said to be $\epsilon$-separated if  $d(x,y) > \epsilon$ for all distinct points $x,y \in P$. The packing number of $Y\subset X$ denoted by $\mathcal{P}(Y,d,\epsilon)$ is defined to be the largest cardinality of an $\epsilon$-separated subset of $Y$.
\end{definition}

The following lemma bounds the covering number of ${\calM}$ using Euclidean balls in $\R^D$, which has been established in literature.
We reproduce under our setting for completeness.

\begin{lemma}\label{lemma:covering-2}
For any $r < \delta_0$, 
where $\delta_0$ is defined in Lemma  \ref{lemma:M-delta0}, 
${\cal N}( {\calM},  \| \cdot \|_{\R^D}, r ) \le  V({\calM}) {r^{-d}}$,
where $V(\calM)$ equals an $O_d(1)$ constant times the Riemannian volume of ${\calM}$.
\end{lemma}

\begin{proof}[Proof of Lemma \ref{lemma:covering-2}]
The proof uses Lemma \ref{lemma:M-delta0} and standard arguments as in Section 4.2 \cite{vershynin2018high}.
Let $d_{E}$ be the Euclidean distance in $\R^D$.

Let $d_E$ denote the metric on ${\calM}$ induced by the Euclidean metric in $\R^D$,
that is, $d_E(x,y) =\|x-y\|_{\mathbb{R}^D}$, where $x,y\in \calM$. The packing number $\mathcal{P}( {\calM}, d_E, r) $ 
always upper bounds the covering number (see  e.g. Lemma 4.2.6 and Lemma 4.2.8 in \cite{vershynin2018high}),
thus it suffices to upper bound 
$\mathcal{P}( {\calM}, d_E, r) $.

Denote by $B_{r}(x, d_E)$ the open ball on $({\calM}, d_E)$ centered at $x$,
and $B_{r, \R^m}(x)$ the open Euclidean ball of radius $r$ centered at $x$ in $\R^m$.
Without declaring $m$, $B_r(x)$ means $B_{r, \R^D} (x)$.
By definition, $B_{r}(x, d_E) = B_r(x ) \cap {\calM}$.
Suppose $r < \delta_0$ in Lemma \ref{lemma:M-delta0}, 
we consider the manifold volume $\Vol$ of these Euclidean balls,
where for $Y \subset {\calM}$,  $\Vol(Y):=\int_{\calM} {\bf 1}_Y dV$ when integrable. 
By Lemma \ref{lemma:M-delta0}(ii), on $T_{x} ({\calM})$ which is viewed as $\R^d$,
\[
B_{0.9 r,\R^d}(0) \subset \phi_{x} (  B_{r }(x, d_E)    ),
\]
and $ \left| \det ( \frac{dy}{du} ) \right| > 0.9$ on  $B_{r}(x, d_E)$, then
\begin{align*}
 Vol (  B_{r }(x, d_E) ) 
& = \int_{\phi_x( 
  B_{r }(x, d_E)
    )} \left| \det ( \frac{dy}{du} ) \right| du 
 \ge 
\int_{ \{u, \, \|u\| < 0. 9 r\}} \left| \det ( \frac{dy}{du} ) \right| du \\
& \ge 0.9 \int_{ \{u, \, \|u\| < 0. 9 r\}}  du
= 0.9 v_d ( 0. 9 r)^d,
\end{align*}
where $v_d$ is the Euclidean volume of a unit $d$-sphere.

Now let $P$ be a maximal $r$-separated subset of ${\calM}$ (under $d_{E}$) such that Card($P$)$= n = \mathcal{P}( {\calM}, d_{E}, r) $,
and $P = \{x_1, \cdots, x_n\}$.
By definition of $r$-separateness,
$B_{\frac{r}{2} }( x_i, d_E)$ are disjoint, 
thus 
\[
Vol( {\calM}) 
\ge  \sum_{i=1}^n  Vol( B_{\frac{r}{2}}( x_i, d_E) ) 
\ge n  \cdot 0.9 v_d ( 0. 9 \frac{r}{2})^d,
\]
that is, for $V( {\calM})$ which is an $O_d(1)$ constant times the Riemannian volume of ${\calM}$,
\[
n \le \frac{V( {\calM})}{r^d}.
\]
This proves that $ \mathcal{N}( {\calM}, d_{E}, r) \le \mathcal{P}( {\calM}, d_{E}, r) \le  \frac{V( {\calM})}{r^d}$.
\end{proof}

 \subsection{Fixed-bandwidth Integral Operator}

\begin{assumption}[Assumption on $k_0$ in \cite{coifman2006diffusion}]
\label{assump:h-diffusionmap}

(C1') Regularity.
$k_0$ is continuous on $[0,\infty)$,  $C^2$ on $(0, \infty)$. 

(C2') Decay condition. 
$k_0$ and up to 
{its second} derivatives are bounded on $(0,\infty)$ and have sub-exponential tail, specifically,
$\exists a, a_k >0$, s.t., $ |h^{(k)}(\xi )| \leq a_k e^{-a \xi}$ for all $\xi > 0$, $k=0, 1,2 $.
To exclude the case that $k_0 \equiv 0$, suppose $\| k_0 \|_\infty > 0$.
 \end{assumption}

\begin{lemma}[Lemma 8 in \cite{coifman2006diffusion}]
\label{lemma:h-integral-diffusionmap}
Suppose $h$ satisfies Assumption \ref{assump:h-diffusionmap}.
For any $f \in C^\infty({\calM})$, define
\begin{equation}\label{eq:def-Ifh-lemma-diffusionmap}
G_\epsilon f (x) : = \int_{\calM} h( \frac{ \| x-y \|^2}{\epsilon}) f(y) dV(y).
\end{equation}
Then there is $\epsilon_0 ({\calM}, h)>0$ such that when $ 0 < \epsilon < \epsilon_0$, 
\[
G_\epsilon [h] f= \epsilon^{\frac{d}{2}} \left( m_0[h] f +  \epsilon \frac{m_2[h]}{2} ( \omega f + \Delta_{\calM}  f) 
+ O^{ [ f^{(\le 4)} ]}(\epsilon^2) \right) ,
\]
where $\omega(x)$ is determined by local derivatives of the extrinsic manifold coordinates at $x$, 
 the residual term denoted by big-$O$ with superscript $f^{(\le 4)}$
 means that the constant 
 involves up to the 4-th derivative of $f$ on ${\calM}$.
Specifically, if the residual term is denoted as $r_{f, \epsilon}(x)$, it satisfies 
$ \sup_{x \in {\calM}} |r_{f, \epsilon}(x)| \le C(f) \epsilon^2$, where
$C(f) = c({\calM}, h) (1+ \sum_{l=0}^4 \| D^{(l)} f \|_\infty)$.
\end{lemma}

For the sake of self-containedness and specifically quantifying the constant in the error term, we provide a proof of this lemma below. 

\vspace{5pt}
\begin{proof}[Proof of Lemma \ref{lemma:h-integral-diffusionmap}]
The original proof is in Appendix B of \cite{coifman2006diffusion}.
We made slightly more precise the truncation argument of the intergral,
as well as under the formal statement of assumptions on $k_0$ as in Assumption \ref{assump:h-diffusionmap}.

The proof uses the exponential decay of $h$ to truncate the integral of $dV(y)$ on ${\calM} \cap B_{\delta_\epsilon}(x)$,
where the Euclidean ball radius $\delta_\epsilon$ can be chosen to be  $\sqrt{ \alpha_0 \epsilon \log \frac{1}{\epsilon}}$ for some $O_d(1)$ constant $\alpha_0$.
E.g., 
let $\alpha_0 = \frac{d+10}{a}$,
where $a$ is the sub-exponential decay constant of $h$ in Assumption \ref{assump:h-diffusionmap}(C2), 
then the truncations of integrals used in the proof all incur an error of order $O(\epsilon^{10})$.
For the truncation tail bounds to hold,
the radius $\delta_\epsilon$ needs to be smaller than 
$\delta_0({\calM})$ in Lemma \ref{lemma:M-delta0}.
The requirement $\delta_\epsilon < \delta_0$
 gives rise to the condition that $\epsilon < \epsilon_0$ in the lemma.

Restricting on a local ball,
the integrals in the proof are computed via local projected coordinates on $T_x {\calM}$,
and using the volume and metric comparison lemmas, Lemma 6 and 7 in \cite{coifman2006diffusion},
as detailed in Appendix B in \cite{coifman2006diffusion}.
In particular, 
only the differentiability and sub-exponential decay of up to 2nd derivatives of $h$,
 and the isometry of the kernel ($h$ is a function of $\|x-y\|^2$)
are used, thus the lemma holds for any $h$ satisfying Assumption \ref{assump:h-diffusionmap}.
\end{proof}
%

The following lemma is the counterpart of  Lemma \ref{lemma:h-integral-diffusionmap} when $h$  is the indicator function 
(only to the ``$O(\epsilon)$'' term, $\epsilon=r^2$ here).
It can be implied by Lemma 4 in \cite{hein2005graphs} (without proof), 
 and was also given in a different setting for uniform $p$ in Lemma 7 in \cite{ting2011analysis}.
We include a proof  for completeness.

\begin{lemma}\label{lemma:G-expansion-h-indicator}
Under Assumption \ref{assump:M-p}, $h = {\bf 1}_{[0,1)}$, 
there is a constant $\delta_1 ({\calM})< \delta_0$ in Lemma \ref{lemma:M-delta0}
such that when $r < \delta_1$,
for any $x \in {\calM}$,
\[
r^{-d} \int_{\calM} h \left(  \frac{ \| x- y\|^2 }{r^2}\right)  p(y) dV(y) =  
r^{-d} \int_{\calM} {\bf 1}_{ \{ \| x- y\| < r \} }  p(y) dV(y) =  m_0[h] p(x) + O^{[p]}( r^2),
\]
and the constant in big-$O$ is uniform for all $x$.
\end{lemma}
\begin{proof}[Proof of Lemma \ref{lemma:G-expansion-h-indicator}]
The proof uses the same technique of that in Lemma \ref{lemma:h-integral-diffusionmap}.
Because $r < \delta_0$ in Lemma \ref{lemma:M-delta0}, using the local chart,
we have that 
\[
I_r := \int_{\calM} {\bf 1}_{ \{ \| x- y\| < r \} }   p(y) dV(y) 
= \int_{ B'} p(y(u)) \left| \det \left(\frac{dy}{du}\right) \right| du,
\quad 
B' :=  \phi_x( B_r(x) \cap {\calM} ) \subset \R^d.
\]
By that 
\[
\| y - x\|^2 = |u|^2 + \rev{O}( |u|^4),
\]
\rev{where the constant in big-$O$ depends on local derivatives of manifold extrinsic coordinates at $x$
 and  by compactness of ${\calM}$ is uniform for all $x$},
there is $\delta_1=\delta_1({\calM})$
and constant $c_M>0$ uniform for all $x$ such that 
when $r < \delta_1$, for any $x$,
\[
B_{r^-} \subset B' \subset B_{r^+}, 
\quad r^{\pm} = r (1 \pm c_M r^2),
\quad B_r:= \{ u \in \R^d, \, |u| < r\}.
\]
We consider upper and lower bounds of  $I_r $ respectively. By that $p > 0$,
\[
I_r \le 
\int_{ B_{r^+}} p(y(u)) \left| \det \left(\frac{dy}{du}\right) \right| du =: I_+.
\]
Similarly as in the proof of Lemma \ref{lemma:h-integral-diffusionmap},
\[
p(y(u)) = p(x) + \nabla_{\calM}p(x) \cdot u + O^{[p]}( |u|^2),
\]
and by \eqref{eq:local-vol-1},
$
\left| \det\left(\frac{dy}{du}\right) \right| = 1 + O(|u|^2)$,
where the constant in big-$O$ depends on local derivatives 
\rev{of manifold extrinsic coordinates}
at $x$ and by compactness of ${\calM}$ is uniform for all $x$.
This gives that
\[
I_+ = \int_{ B_{r^+}} \left( p(x) + \nabla_{\calM}p(x) \cdot u + O^{[p]}( |u|^2) \right) ( 1 + O(|u|^2) ) du
= Vol( B_{r^+}) (p(x)  + O^{[p]}(r^2)),
\]
where the odd-order term of $u$ does not contribute to integral because $B_{r^+}$ is a $d$-sphere,
and $Vol( B_{r^+})  = v_d r^d(1 + c_M r^2)^d = m_0[h] r^d(1 + O(r^2))$.
Thus
\[
I_r \le I_+ 
= m_0[h] r^d(1 + O(r^2)) (p(x) + O^{[p]}(r^2)) 
= m_0[h] r^d (p(x) + O^{[p]}(r^2)).
\]
Similarly,
\[
I_r \ge \int_{ B_{r^-}} p(y(u)) \left| \det (\frac{dy}{du}) \right| du
=  Vol( B_{r^-}) (p(x)  + O^{[p]}(r^2))
= m_0[h] r^d (p(x) + O^{[p]}(r^2)).
\]
Putting together upper and lower bounds proves the lemma.
\end{proof}

\section{Other Lemmas and Proofs}\label{subsec:app-more-lemmas}

\subsection{Proofs of Lemma \ref{lemma:right-operator-3} and \ref{lemma:right-operator-repalce-rho}}

\begin{remark}\label{remark:lemma-right-operator}
The expansion of $G_\epsilon^{(\rho)} f $ for differentiable $\rho$ was derived in Appendix A.3 in \cite{berry2016variable},
where  duality was to analyze the ``right operator'' $G^{(\rho)}_\epsilon $ by its ``left operator''
which were defined in \cite{berry2016variable}.
That bounds the error in the weak sense but not in the strong sense.
Here we give a direct proof of a more precise bound of the error in the  point-wise strong sense, 
which is important for analyzing the  point-wise convergence of $L_N f(x)$.
\end{remark}

\begin{proof}[Proof of Lemma \ref{lemma:right-operator-3}]
For a fixed $x \in {\calM}$, define $\delta_{r_\epsilon}(x,y)$ as the following:
\[
\frac{\| x - y\|^2}{  \epsilon  {\rho}( y)  }
=
\frac{\| x - y\|^2}{  \epsilon  {\rho}( x)  } 
+ \frac{\| x - y\|^2}{  \epsilon {\rho}( x)} \left( \frac{{\rho}( x)}{ \rho( y)  }  - 1 \right)
=: \frac{\| x - y\|^2}{  \epsilon  {\rho}( x)  } 
+ \delta_{r_\epsilon}(x,y).
\]
By that $k_0$ is $C^4$ on $(0,\infty)$, Taylor expansion up to the fourth order at $   \frac{\| x - y\|^2}{  \epsilon  {\rho}( x)  }  $ gives 
\begin{align*}
 k_0 \left(  \frac{\| x - y\|^2}{  \epsilon  {\rho}( y)  } \right)
 & = k_0 \left(  \frac{\| x - y\|^2}{  \epsilon  {\rho}( x)  } \right)
 + k_0' \left(  \frac{\| x - y\|^2}{  \epsilon  {\rho}( x)  } \right)  \delta_{ r_\epsilon}(x,y) 
 + \frac{1}{2}k_0'' \left(  \frac{\| x - y\|^2}{  \epsilon  {\rho}( x)  } \right)  \delta_{ r_\epsilon}(x,y)^2 \\
 & ~~~
 +\frac{1}{6}k_0^{(3)} \left(  \frac{\| x - y\|^2}{  \epsilon  {\rho}( x)  } \right)  \delta_{ r_\epsilon}(x,y)^3
+ \frac{1}{24}k_0^{(4)}(\xi (x,y))   \delta_{ r_\epsilon}(x,y)^4
\end{align*}
where $\xi (x,y) $ is between $ \frac{\| x - y\|^2}{  \epsilon  {\rho}( y)  }$ and $\frac{\| x - y\|^2}{  \epsilon  {\rho}( x)  }$.
Thus
\begin{align*}
G_\epsilon^{(\rho)} f 
& =
\epsilon^{-d/2} 
\left\{
\int_{\calM} 
k_0 \left(  \frac{\| x - y\|^2}{  \epsilon  {\rho}( x)  } \right)  f(y)  dV(y)
 + 
\int_{\calM}  k_0' \left(  \frac{\| x - y\|^2}{  \epsilon  {\rho}( x)  } \right)  \delta_{ r_\epsilon}(x,y)  f(y)  dV(y)
 \right.
 \\
&  ~~~
\left.
+ \cdots +
 \frac{1}{24} \int_{\calM}  k_0^{(4)}(\xi (x,y))  \delta_{ r_\epsilon}(x,y)^4  f(y)  dV(y)
 \right\} \\
& := \textcircled{1} + \textcircled{2} + \textcircled{3} + \textcircled{4} + \textcircled{5}  .
\end{align*}

We first bound $|\textcircled{5}|$. Because $\rho(x) < \rho_{max} $ uniformly on ${\calM}$,
\[
\frac{\| x - y\|^2}{  \epsilon  {\rho}( y)  }, \,
\frac{\| x - y\|^2}{  \epsilon  {\rho}( x)  } \ge
\frac{\| x - y\|^2}{  \epsilon  {\rho}_{max}  }
,.
\]
Thus,
\[
|k_0^{(4)}(\xi (x,y))| \le a_4 e^{-a \xi} \le a_4 e^{- \frac{a}{ {\rho}_{max} } \frac{\| x - y\|^2}{  \epsilon  } } 
= \bar{k}_4\left( \frac{\| x - y\|^2}{  \epsilon  } \right)\,,
\]
where we define
\[
\bar{k}_4(r ) : = a_4 e^{- \frac{a}{ {\rho}_{max} } r } \ge 0.
\]
Note that $\bar{k}_4$
satisfies Assumption \ref{assump:h-diffusionmap},
and $m_0 [\bar{k}_4] $ are constant {depending on} $\rho_{max}$.
Then
\begin{align*}
24 |\textcircled{5}|
& \le 
\epsilon^{-d/2}  
\int_{\calM} | k_0^{(4)}(\xi (x,y)) | | f(y)|  \delta r_\epsilon(x,y)^4  dV(y) \\
& \le
\epsilon^{-d/2}  
\int_{\calM} \bar{k}_4\left( \frac{\| x - y\|^2}{  \epsilon  } \right) | f(y)|
\delta r_\epsilon(x,y)^4  dV(y) \\
& \le
\|f\|_\infty \epsilon^{-d/2}  
\int_{\calM} \bar{k}_4\left( \frac{\| x - y\|^2}{  \epsilon  } \right)
\left( \frac{\| x - y\|^2}{  \epsilon } ( \frac{ 1 }{ \rho( y)  }  - \frac{1}{{\rho}( x)} ) \right)^4 dV(y),
\end{align*}
where we define $\tilde{k}(r) := \bar{k}_4(r)r^4$.
By Lemma \ref{lemma:h-integral-diffusionmap},
 \[
 \epsilon^{-d/2}  
\int_{\calM} \tilde{k}\left( \frac{\| x - y\|^2}{  \epsilon  } \right) 
\left(  \frac{ 1 }{ \rho( y)  }  - \frac{1}{{\rho}( x)} \right)^4 dV(y)
= O^{[  g^{(\le 4)} ]} (\epsilon^2),
 \]
 where  we denote 
 $g(y) = \left(  \frac{ 1 }{ \rho( y)  }  - \frac{1}{{\rho}( x)} \right)^4$,
and then there is $c_5^{\rho}=c_5^{\rho}(\rho_{min}, \rho_{max})$, such that 
$ \sum_{l=0}^4 \| g^{(l)} \|_\infty \le c_5^{\rho} (1+\sum_{l=1}^4 \| {\bf D}^{(l)} \rho^{-1} \|_\infty)$.
This proves that 
\begin{equation}\label{eq:bound-5}
 |\textcircled{5}| \le \epsilon^2  \|f\|_\infty c_ 5^{\rho} \left(1 +\sum_{l=1}^4 \| {\bf D}^{(l)} \rho^{-1} \|_\infty\right).
\end{equation}

The other four terms involve fixed bandwidth $\epsilon \rho(x)$ where $x$ is fixed. 
Applying Lemma \ref{lemma:h-integral-diffusionmap} gives the following.
First,
\begin{align*}
 \textcircled{1}
&= \rho(x)^{\frac{d}{2}} G_{\epsilon \rho(x)} f(x) 
=  \rho^{\frac{d}{2}} \left( m_0 f + \epsilon \rho \frac{m_2}{2} ( \omega f + \Delta f) + O^{ [ f^{(\le 4)} ] }(  \epsilon^2)  \rho^2\right) \\
& =:  \textcircled{1}_1 + O^{ [ f^{(\le 4)} ] }(  \epsilon^2)  \rho^{\frac{d}{2}+2}.
\end{align*}
Define $k_1(r) := k_0'(r)r$ and $g_1(y) := (\frac{\rho(x)}{\rho(y)} - 1) f(y)$.
We have  $g_1(x) = 0$,  and
\begin{align*}
 \textcircled{2}
&= \epsilon^{-d/2} \int_{\calM}  k_0' \left(  \frac{\| x - y\|^2}{  \epsilon  {\rho}( x)  } \right)  
\frac{\| x - y\|^2}{  \epsilon {\rho}( x) } \left( \frac{ {\rho}( x)}{ \rho( y)  }  - 1\right) 
f(y)  dV(y) \\
& = \rho(x)^{\frac{d}{2}} G_{\epsilon \rho(x)}[k_1] (g_1)(x) \\
&= \rho^{\frac{d}{2}}  \left( m_0[k_1] g_1 + \epsilon \rho \frac{m_2[k_1]}{2}( \omega g_1 + \Delta g_1 ) 
  + O^{ [ g_1^{(\le 4)} ] }(\epsilon^2) \rho^2 \right) \\
&= \rho^{\frac{d}{2}}  \left( \epsilon \rho \frac{m_2[k_1]}{2}(  \Delta g_1 ) + O^{ [ g_1^{(\le 4)} ]}(\epsilon^2) \rho^2 \right)  \\
&=: \textcircled{2}_1 + O^{ [ g_1^{(\le 4)} ]}(\epsilon^2) \rho^{\frac{d}{2}+2}.
\end{align*}
Define
$k_2(r) := k_0''(r)r^2$ and $g_2(y) :=   (\frac{\rho(x)}{\rho(y)} - 1)^2f(y)$. 
We have that $g_2(x) = 0$ and
\begin{align*}
 \textcircled{3} 
&= \frac{1}{2} \epsilon^{-d/2} \int_{\calM}  k_0'' \left(  \frac{\| x - y\|^2}{  \epsilon  {\rho}( x)  } \right)  
\left( \frac{\| x - y\|^2}{  \epsilon {\rho}( x) } ( \frac{ {\rho}( x)}{ \rho( y)  }  - 1) \right)^2
f(y)  dV(y)  \\
& =\frac{1}{2} \rho(x)^{\frac{d}{2}} G_{\epsilon \rho(x)}[k_2] (g_2)(x) \\
&=\frac{1}{2} \rho^{\frac{d}{2}}  \left( m_0[k_2] g_2 + \epsilon \rho \frac{m_2[k_2]}{2}( \omega g_2 + \Delta g_2 ) 
 + O^{ [ g_2^{(\le 4)} ]}(\epsilon^2) \rho^2 \right) \\
&= \frac{1}{2} \rho^{\frac{d}{2}}  \left(  \epsilon \rho \frac{m_2[k_2]}{2}(  \Delta g_2 ) 
 + O^{ [ g_2^{(\le 4)} ] }(\epsilon^2) \rho^2 \right) \\
&=: \textcircled{3}_1 
 + O^{ [ g_2^{(\le 4)} ]}(\epsilon^2) \rho^{\frac{d}{2}+2}.
\end{align*}
Define
$k_3(r) = k_0^{(3)}(r)r^3$ and $g_3(y) =   (\frac{\rho(x)}{\rho(y)} - 1)^3 f(y)$.
Then, we have $g_3(x) = 0$, $\Delta g_3(x) = 0$, and 
\begin{align*}
 \textcircled{4}
&=\epsilon^{-d/2} \int_{\calM}  k_0^{(3)} \left(  \frac{\| x - y\|^2}{  \epsilon  {\rho}( x)  } \right)  
\left( \frac{\| x - y\|^2}{  \epsilon {\rho}( x) } ( \frac{ {\rho}( x)}{ \rho( y)  }  - 1) \right)^3
f(y)  dV(y)  \\
& = \rho(x)^{\frac{d}{2}} G_{\epsilon \rho(x)}[k_3] (g_3)(x) \\
&= \rho^{\frac{d}{2}}  \left( m_0[k_3] g_3 + \epsilon \rho \frac{m_2[k_3]}{2}( \omega g_3 + \Delta g_3 ) 
	+ O^{ [ g_3^{(\le 4)} ]}(\epsilon^2) \rho^2 \right)\\
&=   O^{ [ g_3^{(\le 4)} ] }(\epsilon^2) \rho^{\frac{d}{2}+2 }.
\end{align*}
Collecting the leading terms, we have
\[
\textcircled{1}_1 + \textcircled{2}_1 + \textcircled{3}_1
= \rho^{\frac{d}{2}}
 \left(  m_0 f + \epsilon \rho \frac{m_2}{2} ( \omega f + \Delta f))
+    \epsilon \rho \frac{m_2[k_1]}{2}(  \Delta g_1 )
+ \frac{1}{2}   \epsilon \rho \frac{m_2[k_2]}{2}(  \Delta g_2 ) \right).
\]
Note that 
\begin{align*}
m_2[k_1] & = \frac{1}{d} \int_{\R^d}   k_0'( |u|^2) |u|^4 du = - \frac{d+2}{2} m_2[k_0], \\
m_2[k_2] & = \frac{1}{d} \int_{\R^d}   k_0''( |u|^2) |u|^6 du = -\frac{d+4}{2} m_2[k_1] =  \frac{d+4}{2}\frac{d+2}{2} m_2[k_0],
\end{align*}
\begin{align*}
\Delta g_1 & =  \rho f \Delta \frac{1}{\rho} + 2 \rho \nabla f \cdot \nabla  \frac{1}{\rho} 
= 2f \rho^{-2} |\nabla \rho|^2 - f\rho^{-1} \Delta \rho -2 \rho^{-1} \nabla f \cdot \nabla \rho, \\
\Delta g_2 & = 2f \rho^2  |\nabla \frac{1}{\rho}  |^2 = 2f\rho^{-2} |\nabla \rho|^2,
\end{align*}
then we have
\begin{align*}
& \quad\  \frac{m_2}{2}  \Delta f 
+   \frac{m_2[k_1]}{2}  \Delta g_1 
+  \frac{1}{2}   \frac{m_2[k_2]}{2}   \Delta g_2  \\
& = \frac{m_2}{2}  
\left( \Delta f 
 - \frac{d+2}{2}  ( 2f \rho^{-2} |\nabla \rho|^2 - f\rho^{-1} \Delta \rho -2 \rho^{-1} \nabla f \cdot \nabla \rho)
+   \frac{d+4}{2}\frac{d+2}{2}  f\rho^{-2} |\nabla \rho|^2
\right) \\
& = \frac{m_2}{2}  
\left( \Delta f 
 + (\frac{d}{2} +1) (f\rho^{-1} \Delta \rho +2 \rho^{-1} \nabla f \cdot \nabla \rho)
+    \frac{d}{2}(\frac{d}{2} +1) f\rho^{-2} |\nabla \rho|^2
\right) \\
& =\frac{m_2}{2}  
\rho^{-1-d/2} \Delta( f \rho^{1+d/2}),
\end{align*}
and this proves that 
$\textcircled{1}_1 + \textcircled{2}_1 + \textcircled{3}_1$ equals the leading term in \eqref{eq:G-R-expansion-3}.

To prove the Lemma, 
it remains to specify the constants in
\begin{equation}\label{eq:resi-to-bound-1}
\rho^{\frac{d}{2}+2}
\left(
O^{[f^{(\le 4)}]}(  \epsilon^2)  +
O^{[ g_1^{(\le 4)} ]}(  \epsilon^2)   +
O^{[ g_2^{(\le 4)} ]}(  \epsilon^2)  +
O^{[ g_3^{(\le 4)} ]}(  \epsilon^2)  \right) +
|\textcircled{5}|\,.
\end{equation}
Observe the bound of $r_\epsilon^{(2)}$ in \eqref{eq:G-R-expansion-3}. 
By definition of $g_1$, $g_2$, $g_3$, there is $c_j^{\rho}(\rho_{min},\rho_{max}) > 0$, $j=1,2,3$, such that
\[
\sum_{l=0}^4 \| {\bf D}^{(l)} g_s\|_\infty 
\le c_j^{\rho} \left(1+\sum_{l=0}^4 \| D^{(l)}  f \|_\infty \right) \left(1+\sum_{l=0}^4  \| D^{(l)}  \rho^{-1}\|_\infty\right),
\quad j=1,2,3,
\]
and together with \eqref{eq:bound-5},
the constant in front of $\epsilon^2$ in \eqref{eq:resi-to-bound-1} is bounded by
\begin{align*}
& \rho_{max}^{d/2+2} \left\{
\sum_{l=0}^4 \| {\bf D}^{(l)} f\|_\infty
+ \left(\sum_{j=1}^3c_j^\rho + c_5^\rho\right) 
\left(1+\sum_{l=0}^4 \| D^{(l)}  f \|_\infty \right) \left(1+\sum_{l=0}^4  \| D^{(l)}  \rho^{-1}\|_\infty\right)
\right\} \\
=\,& c^\rho \left(1+\sum_{l=0}^4 \| D^{(l)}  f \|_\infty \right) \left(1+\sum_{l=0}^4  \| D^{(l)}  \rho^{-1}\|_\infty\right),
\end{align*}
where constant $c^\rho$  equals a finite summation of  certain powers and ratios of $\rho_{min}$ and  $\rho_{max}$. 
Thus the bound in  \eqref{eq:G-R-expansion-3} holds. 
\end{proof}
%
%

%
%
%
\begin{proof}[Proof of Lemma \ref{lemma:right-operator-repalce-rho}]
Under the condition,
\begin{equation}\label{eq:bound-tilderho-0.1}
0.9 \rho_{min} < 0.9 \rho(x) < \tilde{\rho}(x) < 1.1 \rho(x) < 1.1 \rho_{max},
 \quad \forall x\in {\calM}.
\end{equation}
By definition,
\[
G_\epsilon^{(\tilde{ \rho})} f (x) - G_\epsilon^{(\rho)} f(x)
 =   \epsilon^{-d/2} \int_{\calM} 
\left(
{  k_0 \left(  \frac{\| x - y\|^2}{  \epsilon  { \tilde{\rho}}( y)  } \right)}
- {  k_0 \left(  \frac{\| x - y\|^2}{  \epsilon  { {\rho}}( y)  } \right)} \right)
   f(y)  dV(y), 
\]
and 
\[
k_0 \left(  \frac{\| x - y \|^2}{ \epsilon  \tilde{\rho}( y ) } \right) -
k_0 \left(  \frac{\| x - y \|^2}{ \epsilon   {\rho}( y ) } \right) 
= k_0'(\xi) \frac{\| x - y \|^2}{ \epsilon  {\rho}( y ) }  \left(  \frac{{\rho}(y)  }{ \tilde{\rho}(y) } - 1\right),
\]
where $\xi$ is between $  \frac{\| x - y \|^2}{ \epsilon  \tilde{\rho}( y ) } $ and $\frac{\| x - y \|^2}{ \epsilon {\rho}( y ) } $.
Then, by \eqref{eq:bound-tilderho-0.1},
$
\xi \ge \frac{\| x - y \|^2}{ \epsilon 1.1 {\rho}( y ) }$,
and then by Assumption \ref{assump:h-selftune}(C2),
\[
|k_0'(\xi) | \le a_1 e^{-a \xi } 
\le  a_1 e^{- \frac{a}{ 1.1}  \frac{\| x - y \|^2}{ \epsilon {\rho}( y ) } }.
\]
Thus, also by that 
\[
\left|  \frac{{\rho}(y)  }{ \tilde{\rho}(y) } - 1 \right|  
\le \frac{ \varepsilon \rho(y) }{0.9 \rho(y)} = \frac{\varepsilon}{ 0.9},
\]
we have that
\begin{align*}
&\quad\ | G_\epsilon^{(\tilde{ \rho})} f (x) - G_\epsilon^{(\rho)} f(x) |
\le 
\epsilon^{-d/2} \int_{\calM} 
|k_0'(\xi)| \frac{\| x - y \|^2}{ \epsilon  {\rho}( y ) }  \left|  \frac{{\rho}(y)  }{ \tilde{\rho}(y) } - 1 \right| 
 |  f(y) |  dV(y)  \\
& \le    
\frac{\varepsilon}{ 0.9}  \| f \|_\infty
\epsilon^{-d/2} \int_{\calM} 
a_1 e^{- \frac{a}{ 1.1}  \frac{\| x - y \|^2}{ \epsilon {\rho}( y ) } }
\frac{\| x - y \|^2}{ \epsilon  {\rho}( y ) }     dV(y)
= \frac{\varepsilon}{ 0.9}  \| f \|_\infty
G_\epsilon^{(\rho)}[k_1] {\bf 1}(x),
\end{align*}
where we define
\[
k_1( r) := 
a_1  e^{- \frac{a}{1.1} r } r,
 \quad r \ge 0,
\]
and $k_1$ satisfies Assumption \ref{assump:h-selftune}.
Next, applying Lemma \ref{lemma:right-operator-3}  gives that 
\[
G_\epsilon^{(\rho)}[k_1] {\bf 1} (x)
= m_0[k_1]  {\rho}^{\frac{d}{2}}
+ O^{[ \rho]} (\epsilon).
\]
Thus, with sufficiently small $\epsilon$, uniformly for all $x$,
\[
 | G_\epsilon^{(\tilde{ \rho})} f (x) - G_\epsilon^{(\rho)} f(x) |
 \le \frac{\varepsilon}{0.9} \| f \|_\infty ( m_0[k_1] {\rho}^{\frac{d}{2}}
+ O^{[\rho]} (\epsilon) )
\le c_{\rho}'  \| f \|_\infty \varepsilon,
\]
where $c_{\rho}'$ is $O(1)$ constant depending on $k_1$ multiplied by certain powers of $\rho_{max}$ or  $\rho_{min}$.
\end{proof}

\subsection{Proofs of Lemma \ref{lemma:double-int-f2} and \ref{lemma:replace-hatrhoalpha-GR-hatrho}}

%
%
\begin{proof}[Proof of Lemma \ref{lemma:double-int-f2}]
Define 
\[
\textcircled{1}  :=
 \epsilon^{-\frac{d}{2}}
\int_{\calM} \int_{\calM}  f(x)^2 
k_0 \left(  \frac{\| x - y \|^2}{ \epsilon \hat{\rho}(x)  \hat{\rho}( y ) } \right)
\frac{p(x) p(y) }{ \hat{\rho}(x)^\alpha \hat{\rho}(y )^\alpha}  dV(x) dV(y),
\]
and 
\[
\textcircled{2}  :=
 \epsilon^{-\frac{d}{2}}
\int_{\calM} \int_{\calM}  f(x)^2 
k_0 \left(  \frac{\| x - y \|^2}{ \epsilon \bar{\rho}(x)  \hat{\rho}( y ) } \right)
\frac{p(x) p(y) }{ \bar{\rho}(x)^\alpha \hat{\rho}(y )^\alpha}  dV(x) dV(y).
\]
Then, same as in the analysis of \eqref{eq:circle2-2},  
applying Lemma \ref{lemma:right-operator-3} only to the $O(\epsilon)$ term gives that
\begin{align*}
\textcircled{2} 
& = \int  
( p \hat{\rho}^{\frac{d}{2} - \alpha})(y)
G^{(\bar{\rho})}_{\epsilon \hat{\rho}(y) }( \frac{ f^2 p } {\bar{\rho}^{\alpha}} )(y) dV(y) \\
& = \int  
( p \hat{\rho}^{\frac{d}{2} - \alpha})
\left( m_0[k_0] f^2 p \bar{\rho}^{\frac{d}{2}-\alpha} 
+  \hat{\rho}   r_{1}^{(1)} \right),
\quad
\| r_1^{(1)} \|_\infty = O^{[f,p]}(\epsilon). 
\end{align*}
By \eqref{eq:hatrho-varepsilon} and \eqref{eq:bound-hatrho-0.1},
\[
\left| \int  
p (  \hat{\rho}^{\frac{d}{2} - \alpha} -   \bar{\rho}^{\frac{d}{2} - \alpha} )
( m_0[k_0] f^2 p \bar{\rho}^{\frac{d}{2}-\alpha}  ) \right|  
\le O^{[f,p]} (\varepsilon_\rho),
\]
and then
\[
\left| \int  p \hat{\rho}^{\frac{d}{2} - \alpha +1}   r_{1}^{(1)}\right|
\le O^{[f,p]}(\epsilon)\int  p \bar{\rho}^{\frac{d}{2} - \alpha +1}  \max\{ 0.9^{\frac{d}{2} - \alpha +1}, 1.1^{\frac{d}{2} - \alpha +1}\}
= O^{[f,p]}(\epsilon),
\]
which gives that  
\[
\textcircled{2}
= 
m_0[k_0]\int  
p^2  f^2   \bar{\rho}^{d - 2\alpha} 
+ O^{[f,p]} (\epsilon, \, \varepsilon_\rho).
\]

To bound $| \textcircled{2} - \textcircled{1} |$, introduce
\[
\textcircled{3} := 
\epsilon^{-\frac{d}{2}}
\int_{\calM} \int_{\calM}  f(x)^2 
k_0 \left(  \frac{\| x - y \|^2}{ \epsilon \bar{\rho}(x)  \hat{\rho}( y ) } \right)
\frac{p(x) p(y) }{ \hat{\rho}(x)^\alpha \hat{\rho}(y )^\alpha}  dV(x) dV(y)\,.
\]
Then, 
\[
\textcircled{3} - \textcircled{2}
= \epsilon^{-\frac{d}{2}}
\int_{\calM} \int_{\calM}  f(x)^2 
k_0 \left(  \frac{\| x - y \|^2}{ \epsilon \bar{\rho}(x)  \hat{\rho}( y ) } \right)
\frac{p(x) p(y) }{  \bar{\rho}(x)^{\alpha} \hat{\rho}(y )^\alpha} 
\left( \frac{\bar{\rho}(x)^{\alpha}}{ \hat{\rho}(x)^{\alpha}} -1 \right) dV(x) dV(y).
\]
By non-negativity of $k_0$, $p$, $\bar{\rho}$, $\hat{\rho}$ and $f^2$,
and $\textcircled{2} \ge 0$ and is $O^{[f,p]}(1)$,
similar as in \eqref{eq:circle3-circle2},
we have
\begin{align*}
 |\textcircled{3} - \textcircled{2}|
&  \le O^{[p]}(\varepsilon_\rho) \textcircled{2} 
=  O^{[f,p]}(\varepsilon_\rho).
\end{align*}
This gives that 
$\textcircled{3} = 
 m_0[k_0]\int  
p^2  f^2   \bar{\rho}^{d - 2\alpha} 
+ O^{[f,p]} (\epsilon, \, \varepsilon_\rho) = O^{[f,p]}(1)$.
Meanwhile,   $ \textcircled{3} \ge 0$ by definition.
Finally, we have
\[
\textcircled{3} - \textcircled{1}
= \epsilon^{-\frac{d}{2}}
\int_{\calM} \int_{\calM}  f(x)^2 
\left( k_0 \left(  \frac{\| x - y \|^2}{ \epsilon \bar{\rho}(x)  \hat{\rho}( y ) } \right) 
- k_0 \left(  \frac{\| x - y \|^2}{ \epsilon \hat{\rho}(x)  \hat{\rho}( y ) } \right) 
\right)
\frac{p(x) p(y) }{ \hat{\rho}(x)^\alpha \hat{\rho}(y )^\alpha}  dV(x) dV(y)\,.
\]
And same as in \eqref{eq:bound-circle3-circle1-withk1}, by introducing $k_1(r)$
and using the non-negativity of $f^2$, $p$ and $\hat{\rho}$, 
one can show that
\[
|\textcircled{3} - \textcircled{1}| \le O^{[p]}(\varepsilon_\rho) \cdot (\text{$\textcircled{3}$ with $k_1$})
 = O^{[f,p]} (\varepsilon_\rho).
\]
Putting together, we have $|\textcircled{2} - \textcircled{1}| = O^{[f,p]} (\varepsilon_\rho)$,
and this proves the lemma.
\end{proof}
 \begin{proof}[Proof of Lemma \ref{lemma:replace-hatrhoalpha-GR-hatrho}]
 By definition,
 \[
 G_\epsilon^{ (\tilde{\rho})} \left( \frac{f }{ \tilde{\rho}^\alpha}  - \frac{f }{ {\rho}^\alpha} \right) \rev{(x)}
 = \epsilon^{-d/2} \int_{\calM} 
{  k_0 \left(  \frac{\| x - y\|^2}{  \epsilon  { \tilde{\rho}}( y)  } \right)}
\frac{   f(y) }{\rho(y)^\alpha} 
\left( \frac{\rho(y)^\alpha}{ \tilde{\rho}(y)^\alpha } -1 \right)  dV(y).
 \]
To proceed,
by that $\sup_{x \in {\calM}} \frac{|\tilde{\rho}(x) - \rho(x)|}{ \rho(x) } < \varepsilon < 0.1$ and \eqref{eq:bound-tilderho-0.1},
 for some $c_{\rho,1}$ 
 and $c_{\rho,2}$
 equaling certain powers of $\rho_{max}$ or $\rho_{min}$ multiplied by $\Theta^{[\alpha]}(1)$ constant,  we have
 \[
\left| \frac{\rho(x)^\alpha}{ \tilde{\rho}(x)^\alpha } -1 \right|
\le c_{\rho,1}  \varepsilon,
\quad
\frac{1}{\rho(x)^\alpha}
\le c_{\rho,2},
\quad \forall x\in {\calM}.
 \]
Then,
\begin{align*}
 \left| G_\epsilon^{ (\tilde{\rho})} \left( \frac{f}{ \tilde{\rho}^\alpha}  - \frac{f}{ {\rho}^\alpha} \right) \rev{(x)} \right|
 &\le
 c_{\rho,1}  \varepsilon  \cdot c_{\rho,2}   \|f\|_\infty 
 \left( 
  \epsilon^{-d/2} \int_{\calM}
{  k_0 \left(  \frac{\| x - y\|^2}{  \epsilon  { \tilde{\rho}}( y)  } \right)}  dV(y)
\right) \\
& =  c_{\rho,1}  \varepsilon  \cdot c_{\rho,2}  \|f\|_\infty 
G_{\epsilon}^{(\tilde{\rho})} {\bf 1}(x).
 \end{align*}
By Lemma \ref{lemma:right-operator-3} and \ref{lemma:right-operator-repalce-rho},
 \[
|G_{\epsilon}^{(\tilde{\rho})} {\bf 1}(x) -  m_0 \rho^{\frac{d}{2}}(x) |\le O^{[\rho]}(\epsilon) + c_\rho' \varepsilon,
 \]
 this proves the lemma with $c_\rho''  = \Theta^{[1]}( c_{\rho,1} c_{\rho,2} (m_0 \rho_{max}^{d/2} + 0.1 c_\rho' ) )$
 when $\epsilon$ gets sufficiently small.
 \end{proof}

\section{Point-wise Convergence to ${\calL}_{\hat{\rho}}^{(\alpha)}$}
\label{app:another-pt-limit}

In parallel to Theorems \ref{thm:limit-pointwise-rw-3} and \ref{thm:limit-pointwise-un-3},
we show the  point-wise convergence of $L_N f(x)$ to another limiting operator involving ${\calL}_{\hat{\rho}}^{(\alpha)}$,
in Theorems \ref{thm:limit-pointwise-rw} and \ref{thm:limit-pointwise-un}.
The analysis is by adopting the approach  in \cite{berry2016variable} after conditioning on a fixed $\hat{\rho}$,
yet the difficulty is to handle the a.s. differentiability of the kNN-estimated $\hat{\rho}$.

As pointed out by Section \ref{subsec:C1-inconsist}, $\hat{\rho}(x)$ at any point of differentiability equals 
$ \Theta( ( {k}/{N_y})^{-1/d} )$ which diverges to $\infty$ asymptotically.
We first derive a lemma to bound the derivatives of $\hat{\rho}(x)$ by certain inverse powers of $\hat{R}(x)$.
 The proof of Lemma \ref{lemma:knn-rx} shows that,
when $Y$ has distinct points,
 the estimated  $\hat{\rho}$ from $Y$  
 is piecewise $C^\infty$ on $\R^D$,
and it has the structure that on each of the finitely many polygon ${\bf p}$,
$\hat{\rho}(x) = (\frac{1}{m_0[ h]} \frac{k }{N_y})^{-1/d} \| x- y_{\bf p} \|$,
 for a some point $ y_{\bf p}$ outside ${\bf p}$. 
 We then can upper bound the derivatives of $\hat{R}$ as below.

\begin{lemma}\label{lemma:knn-rx-part2}
Under the condition of Lemma \ref{lemma:knn-rx},
for any $x \in \R^D \backslash E $,
\[
| {\bf D}^{(l)} \hat{R}(x) | \le (l!) \hat{R}(x)^{-l+1}, \quad l=0,1,\cdots, 4,
\]
where the $l$-th derivative ${\bf D}^{(l)} \hat{R}(x)$ is an $l$-way tensor,
 and for any  $l$-way tensor  ${\bf T}: \R^D \times \cdots \times \R^D \to \R$,
\[
| {\bf T} | = \sup_{v \in \R^D, \, \|v\|_2 \le 1} |T (v, \cdots, v)|.
\]
\end{lemma}

The claim extends to higher order derivatives $l  > 4$, 
and we only need up to the fourth derivative in the diffusion kernel analysis. 
The peicewise architecture of $\hat{R}$ 
also allows us to construct smooth uniform approximators to $\hat{\rho}$ without enlarging the derivatives.
\begin{lemma}\label{lemma:knnhatrho-smooth}
Under the condition of Lemma \ref{lemma:knn-rx},
for any $s > 0$,
$\exists \hat{\rho}_s \in C^\infty( {\calM} )$ 
s.t. $\sup_{x \in {\calM}} | \hat{\rho}_s (x)- \hat{\rho} (x)|  < s$, 
and 
\begin{equation}\label{eq:bound-Drhos-by-Drho}
\sup_{x \in {\calM}} | {D}^{(l)}_{{\calM}} \hat{\rho}_s (x) |
 \le   \sup_{x \in {\calM} \backslash E}|{D}^{(l)}_{\calM} \hat{\rho} (x)|, 
\quad
l=0,1,\cdots,4.
\end{equation}
\end{lemma}

Combined with Lemma \ref{lemma:knn-rx-part2} and \eqref{eq:hatrho-varepsilon},
we have the following:
\begin{proposition}\label{prop:smooth-hatrhos}
There is a constant $ C_p >0$ depending on $({\calM}, p)$ 
such that when $\sup_{x \in {\calM}} |\hat{\rho} - \bar{\rho}|/ \bar{\rho} < \varepsilon_\rho < 0.1$, 
for any $s>0$,
 $\exists \hat{\rho}_s \in C^\infty( {\calM})$ 
s.t. $\sup_{x \in {\calM}} | \hat{\rho}_s (x)- \hat{\rho} (x)|  < s$,
and
\begin{equation}\label{eq:upperbound-Dhatrhos}
\|   D^{(l) }_{\calM} \hat{\rho}_s \|_{\infty, {\calM}}
  \le C_p \left( \left(\frac{k}{N_y}\right)^{-l/d}\right), 
 \quad l = 0,\cdots 4.
 \end{equation}
\end{proposition}

Because we can make $s$ arbitrarily small, it is equivalent to prove the graph Laplacian convergence with $\hat{\rho}_s$,
which satisfies \eqref{eq:upperbound-Dhatrhos} by the proposition
and also \eqref{eq:hatrho-varepsilon} by the uniform approximation Theorem \ref{thm:hatrho}.
Below, we write $\hat{\rho}_s$ as $\hat{\rho}$.
We then have the following:

\begin{theorem}\label{thm:limit-pointwise-rw}
Suppose  Theorem \ref{thm:hatrho} holds and 
as $N_y \to \infty$ and $N_x \to \infty$,
\[
\epsilon = o(1), 
\quad \epsilon^{d/2+1} N_x = \Omega (\log N_x),
\quad \epsilon = o \left( (\frac{k}{N_y})^{4/d} \right)\,.
\]
Then, for any $f \in C^\infty({\calM})$,
for sufficiently large  $N_x$ and $N_y$, w.p. higher than $1- 4N_x^{-10} -2N_y^{-10}$,
\[
L^{(\alpha)}_{ rw'} f(x) 
= {\calL}^{(\alpha)}_{\hat{\rho}} f(x)
+ O^{[f,p]} \left(  \left(\frac{k_y}{N_y}\right)^{-4/d} \epsilon \right)
+O^{[1]} \left(  |\nabla f (x)| p(x)^{1/d}   
\sqrt{     \frac{\log N}{   N  \epsilon^{d/2+1}}   }  \right).
\]
\end{theorem}

\begin{theorem}\label{thm:limit-pointwise-un}
Under the same setting as in Theorem \eqref{thm:limit-pointwise-rw}
and
in the same sense of w.h.p,
\[
L^{(\alpha)}_{ un} f(x) 
= (p \hat{\rho}^{d+2-2\alpha})(x) {\calL}_{\hat{\rho}}^{(\alpha)} f(x)
+ O^{[f,p]} \left(  \left(\frac{k_y}{N_y}\right)^{-4/d} \epsilon \right)
+O^{[1]} \left(  |\nabla f (x)| p(x)^{\frac{\alpha - 1}{d}}   
\sqrt{     \frac{\log N}{   N  \epsilon^{d/2+1}}   }  \right).
\]
\end{theorem}

In both theorems, 
the error rates can be worse than those in Theorems \ref{thm:limit-pointwise-rw-3} and \ref{thm:limit-pointwise-un-3}.
It also gives different optimal scaling when choosing $\epsilon$ and $k$ 
so as to balance the bias and variance errors there.
The reason is due to 
that the bounds of magnitudes of derivatives of $\hat{\rho}$ are scaled with inverse powers of $({k}/{N})^{1/d}$.

The proofs of Theorem \ref{thm:limit-pointwise-rw} and \ref{thm:limit-pointwise-un}
are basically the same as those of Theorems \ref{thm:limit-pointwise-rw-3} and \ref{thm:limit-pointwise-un-3}.
The difference is replacing the usage of Lemma \ref{lemma:right-operator-repalce-rho} by a vanilla application of Lemma \ref{lemma:right-operator-3}
with $\rho$ being $\hat{\rho}$ (which is $\hat{\rho}_s$), 
and details are omitted. 
\vspace{5pt}

\begin{proof}[Proof of Lemma \ref{lemma:knn-rx-part2}]
Note that for any $x \in \R^D \backslash E$,
as shown in the proof of Lemma \ref{lemma:knn-rx},
$x$ is in a polygon ${\bf p}$,
and 
$\hat{R}(x) = \| x- y_{\bf p}\|$ for some $ y_{\bf p}$ outside ${\bf p}$.
For $l=0$, the claim is identity.
For $l=1$, $\nabla r(x) = \frac{ x }{\| x\|}$ and $ |\nabla r(x) | =1  $.
When $l=2$,
\[
{\bf D}^{(2)} r(x) = \frac{ \|x\|^2 I_d - x x^T }{ \| x\|^3}\,.
\]
Thus,  for any $v \in \R^D$, $\|v\| =1$,
\[
| {\bf D}^{(2)} r(x)(v,v) | = \frac{ |  \|v\|^2 \|x\|^2 - (v^Tx)^2 | }{\| x\|^3} 
\le  \frac{ \|v\|^2 \|x\|^2}{\| x\|^3} 
= \frac{1}{\|x\|}, 
\]
{and hence} $ |{\bf D}^{(2)} r(x)| \le \frac{2}{ \| x\|}$.
When $l=3$ and 4,
one can verify by definition that 
\[
 |{\bf D}^{(3)} r(x)|  < \frac{6}{ \|x\|^2},
 \quad
  |{\bf D}^{(4)} r(x)|  < \frac{6 \times 4}{ \|x\|^3}.
\]
This proves that 
$|{\bf D}^{(l)} \hat{R}(x)|  \le \frac{ l!}{ \hat{R}(x)^{l-1} }$, for $l=1,\cdots, 4$.
\end{proof}
\vspace{5pt}
\begin{proof}[Proof of Lemma \ref{lemma:knnhatrho-smooth}]
Because $\hat{\rho} = (\frac{1}{m_0[ h]} \frac{k}{N_y})^{-1/d} \hat{R}$,
we consider the smooth approximation of $\hat{R}$ called $\hat{R}_s$,
and let $\hat{\rho}_s = (\frac{1}{m_0[ h]} \frac{k}{N_y})^{-1/d} \hat{R}_s$.
By Lemma \ref{lemma:knn-rx} and its proof,
$\hat{R}$  is continuous on $\R^D$ and $C^\infty$ on $\R^D \backslash E$ which is a finite union of (possibly unbounded) polygons. 
We consider the restriction of $\hat{R}$ on ${\calM}$,
and because ${\calM}$ is $C^\infty$,
$\hat{R}$ is $C^\infty$ on ${\calM} \backslash E$.
Under the probability assumption on $p$ in Assumption \ref{assump:M-p},
the set $E$ intersects with ${\calM}$ 
over finitely many hypersurfaces of dimensionality $(d-1)$ w.p. 1.
Then, there is a finite partition of ${\calM}$ into pieces with peicewise $C^\infty$ boundaries. 
For $s>0$, a function $\hat{R}_s \in C^\infty( {\calM})$   can be constructed 
to  uniformly approximate $\hat{R}$ on ${\calM}$ to within $s$,
and in addition, 
${ D}^{(l)}_{\calM} \hat{R}_s(x)$
 for $l=0,\cdots, 4$ is smoothly averaged from the values of 
${ D}^{(l)}_{\calM} \hat{R}(x)$ on a neighborhood of $x$.
This means that at $x \in {\calM} \cap E$,
${ D}^{(l)}_{\calM} \hat{R}_s(x)$ is smoothly interpolating between the values  
of  ${D}^{(l)}_{\calM} \hat{R}$ on each sides of the hypersurface
(at intersection of multiple hypersurfaces,
i.e. ``corners'', 
${D}^{(l)}_{\calM}\hat{R}_s(x)$ is  interpolating among the multiple values).
Then whether $x$ is near $ {\calM} \cap E$ or not,
we have that
$|{ D}^{(l)}_{\calM} \hat{R}_s(x)| \le \sup_{x \in {\calM} \backslash E}|{D}^{(l)}_{\calM} \hat{R} (x)| $.
This can be done, e.g., by convolving $\hat{R}$ on ${\calM}$ using a Gaussian kernel in $\R^d$ with small bandwidth and under the manifold metric.
The uniform approximation of $|\hat{R}_s(x) - \hat{R}(x)|$ is then guaranteed by that $\hat{R}$ is Lipschitz-1 on $\R^D$,
and thus is globally Lipschitz-1 on $ {\calM}$ with respect to the manifold geodesic metric.
This proves \eqref{eq:bound-Drhos-by-Drho}.
\end{proof}
%
%

%
%
\begin{proof}[Proof of Proposition \ref{prop:smooth-hatrhos}]
Under the good event in Theorem \ref{thm:hatrho}, \eqref{eq:hatrho-varepsilon} equivalently gives that 
\begin{equation}\label{eq:bound-hatR-R-epsrho}
\sup_{x \in {\calM}} \frac{|\hat{R}(x) - \bar{R}(x)|}{ \bar{R}(x)} < \varepsilon_\rho < 0.1.
\end{equation}
Meanwhile,
 $\bar{\rho}(x) = p(x)^{-1/d}$ and is uniformly bounded from below and above by $\rho_{min}$ and  $\rho_{max} $ which are constants depending on $p$.
We write $m_0[h]$ as $m_0$ in this proof.
\begin{equation}\label{eq:lower-bound-barR}
\bar{R}(x)  =   \left(\frac{1}{m_0} \frac{k }{ N_y} \right)^{1/d} \bar{\rho}(x) 
\in \left[ c_{p,1} \left(\frac{k }{ N_y} \right)^{1/d}, \, c_{p,2} \left(\frac{k }{ N_y} \right)^{1/d}  \right].
\end{equation}
Lemma 
\ref{lemma:knn-rx-part2}
gives that, for $l=0,1,\cdots, 4$, ${\bf D}^{(l)}$ being the derivatives in $\R^D$,
\[
| {\bf D}^{(l)} \hat{R}(x) | \le   \frac{ l! }{ \hat{R}(x)^{l-1}},
\quad
\forall x \in {\calM} \backslash E.
\] 
The manifold derivatives are determined by ambient space derivatives via 
\[
{ D}^{(l)}_{\calM} \hat{R}(x) =  \sum_{m = 0}^l A_m(x) ({\bf D}^{(m)} \hat{R}(x)   ),
\]
where $A_m(x)$ are linear transforms determined by extrinsic manifold coordinates and their derivatives at $x$, 
and are in $C^\infty( {\calM})$. 
Thus,
\[
| { D}^{(l)}_{\calM} \hat{R}(x) | \le
  \sum_{m = 0}^l | A_m(x) | | {\bf D}^{(m)} \hat{R}(x)   |
  \le c_{\calM}   \sum_{m = 0}^l   \frac{ m! }{ \hat{R}(x)^{m-1}},
\]
where $c_{\calM}$ is a constant depending on ${\calM}$.
This gives that,  $\forall x \in {\calM} \backslash E$,
\begin{align*}
| { D}^{(l)}_{\calM} \hat{\rho}(x)  |
& = \left(\frac{1}{m_0} \frac{k}{N_y}\right)^{-1/d}   | { D}^{(l)}_{\calM} \hat{R}|
\le \left(\frac{1}{m_0} \frac{k}{N_y}\right)^{-1/d}  c_{\calM}   \sum_{m = 0}^l   \frac{ m! }{ \hat{R}(x)^{m-1}} \\
& \le  c_{\calM}' \frac{c_{p,2}}{ \bar{R}(x)} \sum_{m = 0}^l    \hat{R}(x)^{-m+1}  
\quad \text{(by \eqref{eq:lower-bound-barR}, $c_{\calM}' $ depending on ${\calM}$ )} \\
& <  c_{\calM}'  c_{p,2}  1.1 \sum_{m = 0}^l    \hat{R}(x)^{-m}  
\quad \text{(by \eqref{eq:bound-hatR-R-epsrho})} \\
& \le c_{\calM}'  c_{p,2}  1.1 \sum_{m = 0}^l    ( 0.9 \bar{R}(x) )^{-m}   
\le c_{\calM}'  c_{p,2}  1.1 \sum_{m = 0}^l    \left( 0.9 c_{p,1} (\frac{k}{N_y})^{1/d} \right)^{-m},
\end{align*}
which means that 
\[
\sup_{x \in {\calM} \backslash E} | { D}^{(l)}_{\calM} \hat{\rho}(x)  |
 \le C_p  \left(\frac{k}{N_y}\right)^{-l/d},
 \]
where $C_p$ is a constant depending on $p$, for $l$ up to 4.
Finally, Lemma \ref{lemma:knnhatrho-smooth} constructs $\hat{\rho}_s$
satisfying \eqref{eq:bound-Drhos-by-Drho}, 
and then \eqref{eq:upperbound-Dhatrhos} follows.
\end{proof}

\end{document}